\documentclass[microtype,12pt]{article} 
\usepackage{custarticle}

\date{23rd October 2019}

\title{Introduction to Gauge Theory}
\author{Andriy Haydys}

\begin{document}

\maketitle
	\begin{center}
	\small
		\begin{minipage}[c]{14cm}
			This is lecture notes for a course given at the PCMI Summer School ``Quantum Field Theory and Manifold Invariants'' (July 1 -- July 5, 2019). 
			I describe basics of gauge-theoretic approach to construction of invariants of manifolds. 
			The main example considered here is the Seiberg--Witten gauge theory.  
			However, I tried to present the material in a form, which is  suitable for other gauge-theoretic invariants too. 
		\end{minipage} 	
	\end{center}

\tableofcontents

\section{Introduction}
\label{Sect_Intro}

Gauge theory by now is a vast subject with many connections in geometry, analysis, and physics. 
In these notes I focus on  gauge theory as it is used in the construction of manifolds invariants, other uses of gauge theory remain beyond the scope of these notes. 

The basic scheme of construction invariants of manifolds via gauge theory is quite simple. 
To be more concrete, let me describe some details.
Thus, let $\cG$ be a Lie group acting on a manifold $\cC$. 
The common convention, which I will follow, is that $\cG$ acts \emph{on the right}, although this is clearly nonessential. 
The quotient $\cB:=\cC/\cG$ may fail to be a nice space\footnote{Ideally, one wishes the quotient to be a smooth manifold.} in a few ways, for example due to the presence of points with nontrivial stabilizers.
Denote by $\cC_{irr}\subset \cC$ the subspace consisting of all points with the trivial stabilizer.   
Then $\cG$ acts on $\cC_{irr}$ and the quotient $\cB_{irr}:=\cC_{irr}/\cG$ is better behaved.
Let me assume that $\cB_{irr}$ is in fact a manifold.

Pick a $\cG$--representation $V$ and a smooth $\cG$--equivariant map $F\colon \cC\to V$, where the equivariancy means the following:
\begin{equation*}
	F(a\cdot g) = g^{-1}\cdot F(a).
\end{equation*}  
Here $V$ is thought of as a left $\cG$--module.

More to the point, assuming that $0\in V$ is a regular value of  $F$, we obtain a submanifold $\cM_{irr}:=F^{-1}(0)\cap \cC_{irr}/\cG\subset \cB_{irr}$. 
Furthermore, assume also $d:=\dim\cM_{irr}<\infty$. 
$\cM_{irr}$ will be referred to as the `moduli space', although the terminology may seem odd at this moment.  
If $\cM_{irr}$ is compact and oriented, it has the fundamental class $[\cM_{irr}]\in H_d(\cB_{irr};\, \Z)$. 
In particular, for any cohomology class $\eta\in H^d(\cB_{irr};\,\Z)$ we obtain an integer
\begin{equation*}
	\bigl \langle [\cM_{irr}],\, \eta \bigl \rangle = \int_{\cM_{irr}}\eta,
\end{equation*}
where $\eta$ is thought of as a closed form of degree $d$ on $\cB_{irr}$. 
This is the `invariant' we are interested in. 

One way to construct cohomology classes on $\cB_{irr}$ is as follows. 
Assume there is a normal subgroup $\cG_0\subset \cG$ so that $\rG:=\cG/\cG_0$ is a Lie group. 
Then the `framed moduli space' $\hat \cM_{irr}:= F^{-1}(0)\cap \cC_{irr}/\cG_0$ is equipped with an action of $\cG/\cG_0 = \rG$ such that $\hat \cM_{irr}/\rG = \cM_{irr}$. 
In other words, $\hat\cM_{irr}$ can be viewed as a principal $\rG$ bundle, whose characteristic classes yield the cohomology classes we are after.
More details on this is provided in Section~\ref{Sect_ChernWeil}. 

While this scheme is clearly very general, the details in each particular case may differ to some extend. 
The aim of these notes is to explain this basic scheme in some details rather than variations enforced by a concrete setup.
As an illustration I consider the Seiberg--Witten theory in dimension four, where this scheme works particularly well.  
Due to the tight timeframe of the lectures this remains essentially the only example of a gauge--theoretic problem considered in these notes. 

\medskip

At present a number of monographs in gauge theory is available, for example~\cite{Morgan96_SWequations, Moore:96, DonaldsonKronheimer:90, FreedUhlenbeck91_InstantonsFourMflds, Donaldson:02_FloerHomologyGps, KronheimerMrowka07_MonopolesAndThreeMflds} just to name a few.
All of them are however pretty much specialized to a concrete setting or problem while many features are common to virtually any theory, which is based on counting solutions of non-linear elliptic PDEs, even not necesserily of gauge--theoretic origin. 
While the presentation here does not differ substantially from the treatment of~\cite{Morgan96_SWequations, DonaldsonKronheimer:90}, I tried to keep separate general principles from peculiar features of concrete problems.
Also, I hopefully streamlined somewhat some arguments, for example the proof of the compactness for the Seiberg--Witten moduli space. 
 Although I did not cover any other examples of gauge--theoretic invariants except the Seiberg--Witten invariant, the reader will be hopefully well prepared to read advanced texts on other enumerative invariants such as~\cite{DonaldsonKronheimer:90} on his own.

The reader may ask why should we actually care about gauge theories beyond the Seiberg--Witten one, since most of the results  obtained by gauge--theoretic means  can be reproduced with the help of the Seiberg--Witten gauge theory. 
One answer is that the Seiberg--Witten theory does not generalize to higher dimensions, whereas for example the anti--self--duality equations do generalize and are currently a topic of intense research. 
Even in low dimensions physics suggests that gauge theories based on non-abelian groups such as $\SL(2,\C)$ or more traditional $\SU(2)$ may provide insights that the Seiberg--Witten theory is not capable of.
For example, the Kapustin--Witten equations play a central r\^ole in Witten's approach to the construction of the Jones polynomial.

\medskip

These notes are organized as follows.
In introductory Sections~\ref{Sec_BundlesConnections}\,--\,\ref{Sec_DiracOp} I describe those properties of the `points' in $\cC$ that are important for the intended applications. 
Sections~\ref{Sec_LinEllOp} and~\ref{Sect_FredholmMaps} are devoted to analytic tools used in the studies of gauge--theoretic moduli spaces. 
This can be viewed as studies of properties of $F$ from a (somewhat) abstract point of view. 
The concepts and theorems considered in this part are not really specific to gauge theory, the same toolbox is used in virtually any enumerative problem based on elliptic PDEs.
The last section is devoted to the Seiberg--Witten theory.

I tried to keep these notes self-contained where this did not require long and technical detours. 
I hope this will make the notes available to early graduate students. 
Part of such course is necessarily certain tools from PDEs. 
I tried to state clearly the facts that play a r\^ole  in the main part of these notes, however I did not intend  to give a detailed description or complete proofs of those.
The reader may wish to consult a more specialized literature on this topic, for example relevant chapters of~\cite{Evans10_PDEs, Wells80_AnalOnCxMflds}.

\section{Bundles and connections}
\label{Sec_BundlesConnections}

\subsection{Vector bundles}

\subsubsection{Basic notions}
Roughly speaking, a vector bundle is just a family of vector spaces parametrized by points of a manifold (or, more generally, of a topological space). 

More formally, the notion of a vector bundle is defined as follows. 
\begin{defn}
	\label{Defn_VB}
	Choose a non-negative integer $k$. 
	A real smooth vector bundle of rank $k$ is a triple $(\pi, E, M)$ such that the following holds:
	\begin{enumerate}[(i)]
		\item $E$ and $M$ are smooth manifolds, $\pi\colon E\to M$ is a smooth submersion (the differential is surjective at each point);
		\item For each $m\in M$ \emph{the fiber} $E_m:=\pi^{-1}(m)$ has the structure of a vector space and $E_m\cong\R^k$;
		\item For each $m\in M$ there is a neighborhood $U\ni m$ and a smooth map $\psi_U$ such that the following diagram	
		\begin{center}
			\includegraphics[width=0.4\textwidth]{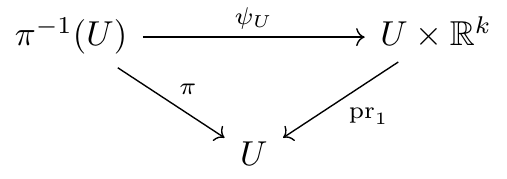}
		\end{center}
	commutes.
	 Moreover, $\psi_U$ is a fiberwise linear isomorphism. 
	\end{enumerate}
\end{defn}

The following terminology is commonly used: $E$ is the total space, $M$ is the base, $\pi$ is the projection, and $\psi_U$ is the local trivialization (over $U$). 

\begin{ex}$\phantom{a}$
	\begin{enumerate}[(a)]
		\item The product bundle: $M\times \R^k$;
		\item The tangent bundle $TM$ of any smooth manifold $M$.
	\end{enumerate}
\end{ex}

Let $E$ and $F$ be two vector bundles over a common base $M$.
A \emph{homomorphism} between $E$ and $F$ is a smooth map $\varphi\colon E\to F$ such that the diagram
\begin{center}
	\includegraphics[width=0.28\textwidth]{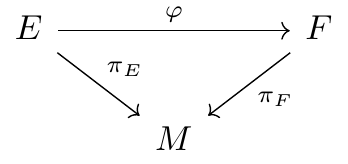}
\end{center}
commutes and $\varphi$ is a fiberwise linear map.

Two bundles $E$ and $F$ are said to be isomorphic, if there is a homomorphism $\varphi$, which is fiberwise an isomorphism.

A bundle $E$ is said to be \emph{trivial}, if $E$ is isomorphic to the product bundle.

\subsubsection{Operations on vector bundles}

Let $E$ and $F$ be two vector bundles over a common base $M$. 
Then we can construct new bundles $E^*,\, \Lambda^pE,\ E\oplus F,\ E\otimes F,$ and  $\Hom(E,\, F)$ as follows:
\begin{enumerate}
	\item[($\ast$)] $(E^*)_m = (E_m)^*$;
	\item[($\Lambda$)] $(\Lambda^p E)_m = \Lambda^p(E_m)$; 
	\item[($\oplus$)] $(E\oplus F)_m:=E_m\oplus F_m$;
	\item[($\otimes$)] $(E\otimes F)_m:= E_m\otimes F_m$;
	\item[(Hom)] $\Hom(E,\, F)_m:=\Hom(E_m,\, F_m)$.
\end{enumerate} 

\medskip

If $f\colon M'\to M$ is a smooth map, we can define the \emph{pull-back} of $E\to M$ via
\[
(f^*E)_{m'}:=E_{f(m')}.
\]
For example, if $M'$ is an open subset of $M$ and $\iota$ is the inclusion, then $E|_{M'}:=\iota^*E$ is just the restriction of $E$ to $M'$. 

The reader should check that the families of vector spaces defined above satisfy the properties required by Definition~\ref{Defn_VB}.

\begin{exercise}
	Prove that $E^*\otimes F$ is isomorphic to $\Hom(E, F)$.
\end{exercise}

\begin{exercise}
	\label{Exerc_TS2nontriv}
	Prove that the tangent bundle of the 2-sphere is non-trivial. (Hint: Apply the hairy ball theorem). 
\end{exercise}

\subsubsection{Sections}
\label{Sec_Sections}

\begin{defn}
	A smooth map $s\colon M\to E$ is called a section, if $\pi\comp s = \id_M$. 
\end{defn}

In other words, a section assigns to each point $m\in M$ a vector $s(m)\in E_m$ such that  $s(m)$ depends smoothly on $m$.

Sections of the tangent bundle $TM$ are called vector fields. 
Sections of $\Lambda^p T^*M$ are called differential $p$-forms. 

\begin{ex}
	 If $f$ is a smooth function on $M$, then the differential $df$ is a 1-form on $M$. More generally, given functions $f_1,\dots, f_p$ we can also construct a $p$-form $\om:= df_1\wedge\dots \wedge  df_p$.  
\end{ex}

\begin{exercise}
	\label{Exerc_LocTrivSections}
	Let $E\to M$ be a vector bundle of rank $k$ and $U\subset M$ be an open subset.
	Prove that $E$ is trivial over $U$ if and only if there are $k$ sections $e=(e_1,\dots, e_k),\ e_j\in\Gamma(U;\, E)$, such that $e(m)$ is a basis of $E_m$ for each $m\in U$. More precisely, given $e$ show that $\psi_U$ can be constructed according to the formula
	\begin{equation*}
		\psi_U^{-1}\colon U\times\R^k\longrightarrow E|_{U}, \qquad (m, x)\mapsto e(m)\cdot x.
	\end{equation*}
	In fact this establishes a one-to-one correspondence between  $k$-tuples of pointwise linearly independent  sections and local trivializations of $E$.
\end{exercise}

We denote by $\Gamma(E)=\Gamma(M;\, E)$ the space of all smooth sections of $E$. 
Clearly, $\Gamma(E)$ is a vector space, where the addition and multiplication with a scalar are defined pointwise. 
In fact, $\Gamma(E)$ is a $C^\infty(M)$-module. 

\medskip

Given a local trivialization $e$ over $U$ (cf.~Exercise~\ref{Exerc_LocTrivSections}) and a section $s$, we can write
\begin{equation*}
	s(m)=\sum_{j=1}^k \sigma_j(m)e_j(m)
\end{equation*}
for some functions $\sigma_j\colon U\to\R$. 
Thus, locally any section of a vector bundle can be thought of as a map $\sigma\colon U\to \R^k$.

It is important to notice that $\sigma$ depends on the choice of a local trivialization. Indeed, if $e'$ is another local trivialization of $E$ over $U'$, then there is a map 
\begin{equation}
	\label{Eq_LocFrames}
	g\colon U\cap U'\longrightarrow \GL_n(\R)\quad\text{such that}\quad e = e'\cdot g.
\end{equation}
If $\sigma'\colon U'\to \R^k$ is a local representation of $s$ with respect to $e'$, we have
\begin{equation*}
	s = e'\sigma' = eg^{-1}\sigma' = e\sigma\qquad\implies\qquad \sigma' = g\sigma.
\end{equation*}

\subsubsection{Covariant derivatives}

The reader surely knows from the basic analysis course that the notion of the derivative is very useful.
It is natural to ask whether there is a way to differentiate sections of bundles too.

To answer this question, recall the definition of the derivative of a function $f\colon M\to \R$. 
Namely, choose a smooth curve $\gamma\colon (-\e,\e)\to M$ and denote $m:=\gamma(0), \ \mathrm v:=\dot\gamma(0)\in T_mM$. 
Then
\begin{equation}
	\label{Eq_AuxDerivative}
	df(\mathrm v)=\lim_{t\to 0}\frac {f(\gamma(t)) - f(m)}{t}.
\end{equation}
Trying to replace $f$ by a section $s$ of a vector bundle, we immediately run into a problem, namely the difference $s(\gamma(t)) - s(m)$ is ill-defined in general since these two vectors may lie in different vector spaces.

Hence, instead of trying to mimic~\eqref{Eq_AuxDerivative} we will define the derivatives of sections axiomatically, namely asking that the most basic property of the derivative---the Leibnitz rule---holds.

\begin{defn}
	\label{Defn_CovDer}
	Let $E\to M$ be a vector bundle. 
	\emph{A covariant derivative} is an $\R$-linear map $\nabla\colon \Gamma(E)\to \Gamma(T^*M\otimes E)$ such that 
	\begin{equation}
		\label{Eq_LeibnizRule}
		\nabla (fs) = df\otimes s + f\nabla s
	\end{equation}
holds for all $f\in C^\infty(M)$ and all $s\in \Gamma(E)$.
\end{defn} 

\begin{example}
	Let $M\subset \R^N$ be an embedded submanifold. 
	Then the tangent bundle $TM$ is naturally a subbundle of the product bundle $\underline\R^N:= M\times \R^N$. 
	In particular, any section $s$ of $TM$ can be regarded as a map $M\to \R^N$. 
	With this at hand we can define a connection on $TM$ as follows
	\begin{equation*}
		\nabla s:= \mathrm{pr}(ds),
	\end{equation*}
	where $\mathrm{pr}$ is the orthogonal projection onto $TM$. 
	A straightforward computation shows that this satisfies the Leibniz rule, i.e., $\nabla$ is a connection indeed. 
\end{example}

\begin{thm}
	\label{Thm_ConnIsAnAffSpace}
	For any vector bundle $E\to M$ the space of all connections $\cA(E)$ is an affine space modelled on $\Om^1(\End E)=\Gamma\bigl (T^*M\otimes\End(E)\bigr )$.  
\end{thm}

To be somewhat more concrete, the above theorem consists of the following statements:
\begin{enumerate}[(a), itemsep=-3pt]
	\item\label{It_ConnNonempty}
	$\cA(E)$ is non-empty.
	\item \label{It_AuxDiffConnections}For any two connections $\nabla$ and $\hat\nabla$ the difference $\nabla - \hat\nabla$ is a 1-form with values in $\End(E)$;
	\item 
	\label{It_AuxNablaPlusA}
	For any $\nabla\in\cA(E)$ and any $a\in \Om^1(\End E)$ the following 
	\begin{equation*}
		(\nabla +a) s:= \nabla s + as
	\end{equation*}
is a connection. 
\end{enumerate}

For the proof of Theorem~\ref{Thm_ConnIsAnAffSpace} we need the following elementary lemma, whose proof is left as an exercise.
\begin{lem}
	\label{Lem_Tensorial}
	Let $A\colon\Gamma(E)\to\Om^p(F)$ be an $\R$-linear map, which is also $C^\infty(M)$-linear, i.e., 
	\[
	A(fs) = fA(s)\qquad \forall f\in C^\infty(M)\qandq  \forall s\in\Gamma(E).
	\]
	Then there exists $a\in\Om^p\bigl (  \Hom(E, F) \bigr )$ such that $A (s) = a\cdot s$.\qed
\end{lem}

\begin{proof}[Proof of Theorem~\ref{Thm_ConnIsAnAffSpace}]
	Notice first that $\cA(E)$ is convex, i.e., for any $\nabla,\hat\nabla\in\cA(E)$ and any $t\in [0,1]$ the following $t\nabla + (1-t)\hat\nabla$ is also a connection.

	If $\psi_U$ is a local trivialization of $E$ over $U$, then we can define a connection $\nabla_U$ on $E|_U$ by declaring 
	\begin{equation*}
		\nabla_U s:= \psi_U^{-1}\,d(\psi_U(s)).
	\end{equation*}
	Using a partition of unity and the convexity property, a collection of these local covariant derivatives can be sewed into a global covariant derivative just like in the proof of the existence of  Riemannian metrics on manifolds, cf.~\cite{BardenThomas03_IntroDiffMflds}*{Thm.\,3.3.7}.
	This proves~\ref{It_ConnNonempty}.

	By~\eqref{Eq_LeibnizRule}, the difference $\nabla - \hat\nabla$ is $C^\infty(M)$--linear. 
	Hence, \ref{It_AuxDiffConnections} follows by Lemma~\ref{Lem_Tensorial}.
	
	\medskip
	
	The remaining step, namely~\ref{It_AuxNablaPlusA}, is straightforward.
	This finishes the proof of this theorem.
\end{proof}

While Theorem~\ref{Thm_ConnIsAnAffSpace} answers the question of the existence of connections, the reader may wish to have a more direct way to put his hands on a connection.
One way to do this is as follows.

Let $e$ be a local trivialization.
Since $e$ is a pointwise base we can write
\begin{equation}
	\label{Eq_Conn1Form}
	\nabla e = e\cdot A,
\end{equation}
where $A=A(\nabla, e)$ is a $k\times k$-matrix, whose entries are $1$-forms defined on $U$. 
$A$ is called the connection matrix of $\nabla$ with respect to $e$. 

If $\sigma$ is a local representation of a section $s$, then
\begin{equation*}
	\nabla s = \nabla (e\sigma)=\nabla(e)\sigma + e\otimes d\sigma = e\bigl ( A\sigma + d\sigma\bigr ). 
\end{equation*}
Hence, it is common to say that locally
\begin{equation*}
	\nabla = d + A,
\end{equation*} 
which means that $d\sigma + A\sigma$ is a local representation of $\nabla s$.
In particular, $\nabla$ is uniquely determined by its connection matrix over $U$ and any $A\in\Om^1(U; \gl_k(\R))$ appears as a connection matrix of some connection (cf. Theorem~\ref{Thm_ConnIsAnAffSpace}). 

Just like $\sigma$, the connection matrix also depends on the choice of $e$. 
If $e' = eg$, we have
\begin{equation*}
	\nabla e' = \nabla(eg) = (\nabla e)g + e\otimes dg = e\bigl (Ag + dg\bigr ) = e' \bigl (g^{-1}Ag + g^{-1}dg\bigr ).
\end{equation*}
Hence, the connection matrix $A'$ of $\nabla$ with respect to $e'$ can be expressed as follows:
\begin{equation}
	\label{Eq_ConnMatrGauge}
	A' = g^{-1}Ag + g^{-1}dg.
\end{equation}

\subsubsection{The curvature}

While Definition~\ref{Defn_CovDer} yields a way to differentiate sections, some properties very well known from multivariable analysis are \emph{not} preserved.
One of the most important cases is that covariant derivatives with respect to two  variables do not need to commute.
The failure of the commutativity of partial covariant derivatives is closely related to the notion of curvature, which is described next.  

\medskip

Denote $\Om^p(E):=\Gamma\bigl (\Lambda^pT^*M\otimes E \bigr)$. 
We can extend the covariant derivative $\nabla$ to a map $d_\nabla\colon \Om^p(E)\to \Om^{p+1}(E)$ as follows. 
If $s\in\Gamma(E)$ and $\om\in\Om^p(M)$ we declare
\begin{equation*}
	d_\nabla (\om\otimes s) := d\om \otimes s + (-1)^p\om\wedge \nabla s,
\end{equation*} 
which yields a unique $\R$-linear map $d_\nabla$ defined on all of  $\Om^p(E)$.
This satisfies the following variant of the Leibniz rule
\begin{equation*}
	d_\nabla(\a\wedge\om) = d\a\wedge \om + (-1)^{q}\a\wedge d_\nabla\om,\qquad\text{for all } \a \in \Om^q (M)\ \text{ and } \om\in \Om^p(M;\, E).
\end{equation*}
Thus, we obtain a sequence
\begin{equation*}
	0\to\Om^0(E)\xrightarrow{\ d_\nabla = \nabla\ }\Om^1(E)\xrightarrow{\ d_\nabla\ }\dots\to \Om^n(E)\to 0,
\end{equation*}
where $n=\dim M$.
However, unlike in the case of the de Rham differential, the above does not need to be a complex.

\begin{proposition}
	\label{Prop_Curvature}
	There is a 2-form $F_\nabla$ with values in $\End E$ such that
	\begin{equation}
	\label{Eq_CurvForm}
		d_\nabla\comp d_\nabla = F_\nabla.
	\end{equation}
\end{proposition}

The above equality means that for any $\om\in\Om^p(E)$ we have $d_\nabla\bigl ( d_\nabla(\om)\bigr) = F_\nabla\wedge\om$, where the right hand side of the letter equality is a combination of the wedge-product and the natural contraction $\End(E)\otimes E\to E$.

\begin{proof}[Proof of Proposition~\ref{Prop_Curvature}.]
	We prove the proposition for $p=0$ first. 
	Applying the Leibniz rule twice, we see that $d_\nabla\comp \nabla\colon\Om^0(E)\to \Om^2(E)$ is $C^\infty(M)$--linear. 
	Hence, by \autoref{Lem_Tensorial} we obtain a 2-form $F_\nabla$ such that 
	\begin{equation*}
		d_\nabla (\nabla s) = F_\nabla s
	\end{equation*}
	holds for any $s\in\Om^0(E)$.
	
	It remains to consider the case $p>0$. 
	For any $\eta\in\Om^p(M)$ and $s\in\Gamma(E)$ we have
	\begin{align*}
		d_\nabla\bigl ( d_\nabla\(\eta\otimes s\) \bigr ) &= 
		d_\nabla\bigl ( d\eta\otimes s + (-1)^p\eta\wedge\nabla s \bigr ) \\
		&= 0 + (-1)^{p+1}d\eta\wedge \nabla s + (-1)^p d\eta\wedge \nabla s + (-1)^{2p}\eta\wedge d_\nabla \bigl( \nabla s \bigr)\\
		&= \eta\wedge F_\nabla s\\
		&=F_\nabla\wedge (\eta\otimes s).
	\end{align*}
	Here the last equality holds because $F_\nabla$ is a differential form of even degree.
\end{proof}

\begin{defn}
	The 2-form $F_\nabla$ defined by~\eqref{Eq_CurvForm} is called the curvature form of $\nabla$. 
\end{defn}

Our next aim is to clarify somewhat the meaning of the curvature form. 
If $\mathrm v$ is a tangent vector at some point $m\in M$, we call
\begin{equation*}
	\nabla_{\mathrm v}\, s:= \imath_{\mathrm v}\nabla s
\end{equation*}  
the covariant derivative of $s$ in the direction of $\mathrm v$. 

Choose local coordinates $(x_1,\dots, x_n)$ and denote by 
$\del_i=\frac \del{\del x_i}$ the tangent vector of the curve
\begin{equation*}
\gamma(t) = (x_1^0,\dots, x_{i-1}^0,\; t,\; x_{i+1}^0,\dots, x_n^0).
\end{equation*}
We may call
\begin{equation*}
	\nabla_i\, s:=\nabla_{\del_i}\; s
\end{equation*}
``partial covariant derivatives''.
With these notations at hand we have the expression
\begin{equation*}
	\nabla s = \sum_{i=1}^n dx_i\otimes\nabla_i\, s,
\end{equation*}
which is just a form of the familiar expression for the differential of a function: $df=\sum \tfrac {\del f}{\del x_i}dx_i$.

Furthermore, we have
\begin{align*}
	d_\nabla( \nabla s ) 
	&=-\sum_{i=1}^n dx_i\wedge \nabla \bigl ( \nabla_i\, s \bigr)
	=-\sum_{i,j=1}^n dx_i\wedge dx_j \otimes \nabla_j(\nabla_i s)\\
	&=\sum_{i,j=1}^n dx_i\wedge dx_j \otimes \nabla_i(\nabla_j s).
\end{align*}
From this we conclude 
\begin{equation*}
	F_\nabla\bigl (\del_i,\,\del_j \bigr) =  \nabla_i(\nabla_j s) -  \nabla_j(\nabla_i s),
\end{equation*}
i.e., the curvature form measures the failure of partial covariant derivatives to commute as mentioned at the beginning of this section.

\medskip

Sometimes it is useful to have an  expression of  the curvature form in a local frame. 
Thus, pick a local frame $e$ of $E$ and let $A$ be the connection form of $\nabla$ with respect to $e$. 
If $\sigma$ is a local representation of some $s\in\Gamma(E)$, then we have
\begin{align*}
	d_\nabla\bigl ( \nabla (e\sigma) \bigr) &= d_\nabla \bigl ( e(d\sigma + A\sigma)  \bigr)
	= \nabla e\wedge (d\sigma + A\sigma)  + e (dA\,\sigma - A\wedge d\sigma)\\
	&=e \Bigl ( A\wedge d\sigma + A\wedge A\,\sigma + dA\,\sigma -A\wedge d\sigma \Bigr)\\
	&=e\bigl ( dA + A\wedge A \bigr)\, \sigma.
\end{align*} 
Hence, we conclude that locally
\begin{equation}
	\label{Eq_CurvLoc}
	F_\nabla = dA + A\wedge A.
\end{equation}
In particular, the curvature form is a first order non-linear operator in terms of the connection form.  

\begin{remark}
	It turns out to be useful to think of $A$ as a 1--form with values in the Lie algebra $\gl_k(\R)=\End(\R^k)$. 
	From this perspective it is more suitable to write~\eqref{Eq_CurvLoc} in the following form
	\begin{equation}
		\label{Eq_CurvLocLieBrackets}
		F_\nabla = dA + \frac 12 [A\wedge A],
	\end{equation}
	where the last term is a combination of the wedge product and the Lie brackets. 
\end{remark}

If $e'$ is another local frame such that $e = e'\cdot g$, then using~\eqref{Eq_ConnMatrGauge} one can show that
\begin{equation*}
	F_\nabla' = dA' + A'\wedge A' = g^{-1} F_\nabla\, g . 
\end{equation*}
The reader is strongly encouraged to check the details of this computation.

Local expression~\eqref{Eq_CurvLoc} implies immediately the following.
\begin{proposition}
	If $a\in \Om^1(\End E)$, then the curvature forms of $\nabla$ and $\nabla +a$ are related by the equality:
	\begin{equation*}
		\pushQED{\qed}
		F_{\nabla + a} = F_\nabla + d_\nabla a + a\wedge a.
		 \qedhere\popQED
	\end{equation*}
\end{proposition} 

\subsubsection{The gauge group}

Pick a vector bundle $E$ and consider
\begin{equation*}
	\cG=\cG(E):=\Bigl\{  g\in \Gamma\bigl (\End(E)\bigr ) \mid \forall m\in M\ g(m)\in\GL(E_m)\Bigr \},
\end{equation*} 
which is endowed with the $C^\infty$--topology. 
If $E$ is endowed with an extra structure, for example an orientation or a scalar product, we also require that gauge transformations respect this structure.

Clearly, $\cG$ is a topological group, where the group operations are defined pointwise. 
$\cG$ is called the group of gauge transformations of $E$ or simply \emph{the gauge group}.

If $\nabla$ is a connection on $E$ and $g\in\cG$, we can define another connection as follows
\begin{equation*}
	\nabla^g s:= g^{-1}\nabla (gs).
\end{equation*}
This yields a right action of $\cG$ on $\cA(E)$.

\begin{defn}
	Two connections are called gauge equivalent, if there is a gauge transformation that transforms one of the connections into the other one. 
\end{defn}

Pick a local trivialization $e$ of $E$ over an open subset $U$. 
Let $A=A(e,\nabla)\in \Om^1(U;\gl_k(\R))$ be a local representation of $\nabla$, i.e., $\nabla e= e\cdot A$. 
Let us compute the local representation of $\nabla^g$. 
Notice first, that using $e$ we may think of $g$ as a map $U\to \GL_k(\R)$, i.e., $g(e) = e\cdot g$. 
Then 
\begin{equation*}
	\nabla^g e = g^{-1}\bigl(\nabla (e\cdot g)\bigr ) = g^{-1}\bigl(e\cdot Ag +e\cdot dg\bigr ) = g^{-1}(e)\cdot (Ag+ dg) = e\cdot \bigl ( g^{-1}Ag + g^{-1}dg\bigr). 
\end{equation*} 
We conclude 
\begin{equation*}
	A(\nabla^g, e) = g^{-1}Ag + g^{-1}dg = A(\nabla,\; e\cdot g),
\end{equation*}
where the last equality follows by~\eqref{Eq_ConnMatrGauge}.
Thus,~\eqref{Eq_ConnMatrGauge} expresses both the change of the connection matrix under the change of local trivializations and the action of the gauge group.

\subsection{Principal bundles}

The computations we met in the previous sections on the dependence of our objects on the choice of a local frame become formidable quite soon. 
The notion of a principal bundle is useful in dealing with this and turns out to have other advantages as we will see below. 
The idea is to consider all possible frames at once rather than choosing local trivializations when needed.

\subsubsection{The frame bundle and the structure group}

Let $G$ be a Lie group. 

\begin{defn}
	\label{Defn_PrincipalBundle}
	A principal bundle with the structure group $G$ is a triple $(P,M,\pi)$, where 
	\begin{enumerate}[(i),itemsep=-3pt]
		\item $P$ and $M$ are smooth manifolds and $\pi\colon P\to M$ is a surjective submersion;
		\item $G$ acts on $P$ on the right such that $\pi(p\cdot g)=\pi(p)$ for all $p\in P$ and all $g\in G$; 
		\item $G$ acts freely and transitively on each fiber $\pi^{-1}(m)$;
		\item \label{It_LocTrivPrincBundle}
		 For each $m\in M$ there is a neighborhood $U\ni m$ and a map $\psi_U$ such that the diagram
		 \begin{center}
		 	\includegraphics[width=0.35\textwidth]{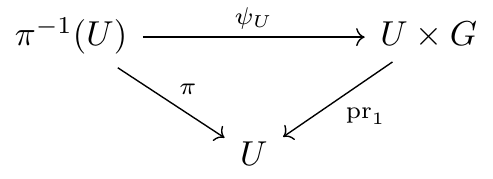}
		 \end{center}
		commutes. Moreover, $\psi_U$ is $G$--equivariant  $\psi_U(p\cdot g) = \psi_U(p)\cdot g$, where $G$ acts on $U\times G$ by the multiplication on the right on the second factor. 
	\end{enumerate}
\end{defn}

A fundamental example of a principal bundle is the frame bundle of a vector bundle $E\to M$.
We take a moment to describe the construction in some detail.

Thus, for a fixed $m\in M$ let $\Fr(E_m)$ denote the set of all bases of $E_m$. 
The group $\GL_k(\R)$ acts freely and transitively on $\Fr(E_m)$ so that we can in fact identify $\Fr(E_m)$ with $\GL_k(\R)$ even though in a non-canonical way. 
In any case, $\Fr(E_m)$ can be viewed as an open subset of $\R^{k^2}$.

Consider 
\begin{equation*}
	\Fr(E):=\bigsqcup_{m\in M} \Fr(E_m).
\end{equation*}
Clearly, there is a well-defined projection $\pi\colon \Fr(E)\to M$ determined uniquely by the property: $\pi(p)=m$ if and only if $p\in \Fr(E_m)$.

We introduce a smooth structure on $\Fr(E)$ as follows. 
Pick a chart $U$ on $M$. 
By shrinking $U$ if necessary we can assume that there is a local frame $e$ of $E$ defined on $U$.  
Then we have the bijective map
\begin{equation}
	\Psi_U\colon U\times\GL_k(\R)\longrightarrow \pi^{-1}(U),\qquad (m, h)\mapsto e(m)\cdot h.
\end{equation}
Using this we can think of $\pi^{-1}(U)$ as an open subset of an Euclidean space so that we can declare $\pi^{-1}(U)$ to be a chart on $\Fr(E)$ with an obvious choice of coordinates. 

Let $U'$ be another chart on $M$ such that there is a local frame $e'$ defined on $U'$. 
A straightforward computation yields
\begin{equation*}
	\Psi_{U'}^{-1}\comp\Psi_U(m, h) = \bigl (m,\, g(m)h\bigr ),
\end{equation*}
where $g$ is defined by~\eqref{Eq_LocFrames}.
This implies that the transition maps between $\pi^{-1}(U)$ and $\pi^{-1}(U')$ are smooth, i.e., we have constructed a smooth atlas on $\Fr (U)$. 
The rest of the properties required in the definition of the principal bundle are clear from the construction.

\begin{exercise} Show that local triviality of a principal bundle, i.e., Property~\ref{It_LocTrivPrincBundle} of Definition~\ref{Defn_PrincipalBundle}, is equivalent to the existence of local sections.
		More precisely, if $P$ admits a trivialization over $U$, then there is a section of $P|_U$ and conversely, if $P|_U$ admits a section, then $P|_U$ is also trivializable over $U$. 
		In particular, show that the frame bundle of $TS^2$ does not admit any global sections.
\end{exercise}

\medskip

Often vector bundles come equipped with an extra structure, for example orientations of each fiber and/or scalar product on each fiber. 
In the language of principal bundles this corresponds to the notion of a $G$-structure.

\begin{defn}
	Let $G$ be a Lie subgroup of $\GL_k(\R)$. 
	A $G$-structure on $E$ is a $G$--subbundle $P$ of the frame bundle. 
	In this case $G$ is called the structure group of $E$. 
\end{defn}   

To illustrate this notion, let us consider the following example. 
Assume $E$ is an Euclidean vector bundle, which means that each fiber $E_m$ is equipped with an Eucliden scalar product $\langle \cdot, \cdot \rangle_m$, which depends smoothly on $m$. 
Here the dependence is said to be smooth if  for any two smooth sections $s_1$ and $s_2$ the function $\langle s_1, s_2\rangle$ is also smooth.

It is natural to consider the subset 
\begin{equation*}
	\O(E):=\bigl \{  e\in\Fr(E)\mid e \text{ is orthonormal}\, \bigr \}.
\end{equation*}
The restriction of $\pi$ yields a surjective map $\O(E)\to M$, which is still denoted by $\pi$. 
If $e$ is any local frame of $E$ over an open subset $U$, the Gram-Schmidt orthogonalization process shows that there is also a smooth pointwise orthonormal frame $e_{O}$ defined on $U$.
Just like in the case of $\Fr(E)$ we can cover $\O(E)$ by open subsets $\pi^{-1}(U)$ such that 
\begin{equation*}
	\Psi_U\colon U\times \O(k)\to \pi^{-1}(U),\qquad (m, h)\mapsto e_O(m)\cdot h
\end{equation*}
is a bijection. 
While $\O(k)$ is not an open subset of an Euclidean space, it is a manifold, and therefore we can cover $\O(k)$---hence, also $\pi^{-1}(U)$---by a collection of charts. The same argument as in the case of the frame bundle shows that the transition functions are smooth so that $\O(E)$ is a principal $\O(k)$--bundle.   
  
We see that an Euclidean structure on $E$ determines an $\O(k)$--structure on $E$.
Conversely, an $\O(k)$-structure $P\subset \Fr(E)$ determines an Euclidean structure on $E$. 
Indeed, pick any  $p\in P_m$ and any two vectors $\mathrm v_1,\, \mathrm v_2\in E_m$. 
Since $p$ is a basis of $E_m$, we can write $\mathrm v_i= p\cdot \mathrm x_i$, where $\mathrm x_i\in\R^k$. 
We define a scalar product on $E_m$ by
\begin{equation*}
	\langle \mathrm v_1, \mathrm v_2\rangle = \mathrm x_1^t\mathrm x_2.
\end{equation*}
It is straightforward to check that this does not depend on the choice of $p$. 
Moreover, the scalar product defined in this way depends smoothly on $m$.

To summarize, an Euclidean structure on a vector bundle is equivalent to an $\O(k)$--structure.

\begin{exercise}
	A fiberweise volume form is by definition  a nowhere vanishing section of $\Lambda^k E^*$, where $k=\rk E$.
	Show that there is a one-to-one correspondence between fiberwise volume forms and $\SL_k(\R)$--structures.   
\end{exercise}

Let $V$ be a complex vector space of complex dimension $k$. 
One can view $V$ as a real vector space of dimension $2k$ equipped  with an endomorphism $I\in\End_\R(V),\ I\mathrm v=i\mathrm v$ so that $I^2 = -\id$. 
Conversely, given a real vector space equipped with an endomorphism $I\in\End_\R(V)$ such that $I^2 = -\id$ we can regard $V$ as a complex vector space, where  $i\cdot \mathrm v := I\mathrm v$.
In this case $\dim_\R V $ is necessarily even. 
The map $I$ is called a complex structure. 

With the above understood, a complex vector bundle is just a real vector bundle equipped with $I\in\Gamma\bigl ( \End E\bigr)$ such that $I^2 = -\id$. 
In particular, each fiber $E_m$ is endowed with a complex structure $I(m)$. 
Thus, a complex vecor bundle is essentially a family of complex vector spaces parametrized by points of the base. 

\begin{exercise}$\phantom{as}$
	\begin{enumerate}[(i),itemsep=-3pt]
		\item Show that a complex vector bundle can be defined as a locally trivial family of complex vector spaces akin to Definition~\ref{Defn_VB};
		\item Show that there is a one-to-one correspondence between complex structures on a real vector bundle $E$ or rank $2k$ and $\GL_k(\C)$--structures. 
		Here $\GL_k(\C)$ is viewed as a subgroup of $\GL_{2k}(\R)$
		\begin{equation*}
			\GL_k(\C) = \bigl \{ A\in \GL_{2k}(\R)\mid A\comp I_{st} = I_{st}\comp A\; \},
		\end{equation*}
		 where $I_{st}$ is the standard complex structure on $\R^{2k}$:
		 \begin{equation*}
		 	I_{st}(x_1, y_1,\dots, x_k, y_k) := (-y_1, x_1,\dots, -y_k, x_k). 
		 \end{equation*}
	\end{enumerate}
\end{exercise} 

\begin{exercise}
	\label{Exerc_HermStructures}
A Hermitian structure on a complex vector bundle is a smooth family of Hermitian scalar products on each fiber.
Prove that the following holds:
\begin{enumerate}[(i),itemsep=-3pt]
	\item Any complex vector bundle admits a Hermitian structure;
	\item There is a one-to-one correspondence between Hermitian structures and $\U(k)$--structures. 
\end{enumerate} 	
\end{exercise}

\subsubsection{The associated vector bundle}

Let $\pi\colon P\to M$ be a principal $G$--bundle.
For any representation $\rho\colon G\to \GL(V)$, where $V$ is a vector space, we can construct a vector bundle over $M$ as follows.

Define the right $G$--action on the product bundle $P\times V$ via
\begin{equation*}
(p, \mathrm v)\cdot g = \bigl (p\cdot g,\; \rho(g^{-1})\mathrm v \bigr).
\end{equation*}
Clearly, this action is free and properly discontinuous so that the quotient 
\begin{equation}
	\label{Eq_AssVB}
	P\times_{\rho}V := \bigl ( P\times V\bigr )/G 
\end{equation} 
is a smooth manifold. 
The map $\pi$ yields a well-defined projection  
\begin{equation*}
	P\times_{\rho}V\to M,\qquad [p, \mathrm v]\mapsto \pi(p)
\end{equation*}
which we still denote by the same letter $\pi$.
Its fibers are isomorphic to $V$, hence these have a canonical structure of vector spaces.  
Moreover, given a local section $s\in\Gamma(U;\, P)$ we can construct a local trivialization of $E$ via the map
\begin{equation*}
	\psi_U^{-1}\colon U\times V \to \pi_E^{-1}(U),\qquad (m,\mathrm v)\mapsto [s(m), \, \mathrm v].
\end{equation*}
Hence, if $P$ is locally trivial, so is $E$. 

\begin{defn}
	The vector bundle $E=E(P,\rho, V)$ defined by~\eqref{Eq_AssVB} is called the vector bundle associated with $(P,\rho)$, or simply \emph{the associated bundle}.
\end{defn}

\begin{example}
	Let $P=\Fr(E)$, $V=\R^k$, and $\rho=\id$ be the tautological representation of $\GL_k(\R)$. 
	Then the map
	\begin{equation*}
		\Fr(E)\times \R^k\longrightarrow E,\qquad (e, x)\mapsto \sum_{i=1}^n x_ie_i
	\end{equation*}
	induces an isomorphism $E\bigl (\Fr(E), \id\bigr )\cong E$.
\end{example}

This example shows that the construction of an associated bundle allows one to recover a vector bundle from its frame bundle.
However, by varying the representation $\rho$ we can obtain other bundles as well. 
The following example illustrates this. 


\begin{example}
	Consider $P=\Fr(E)$ again, however this time we take $V=\Lambda^p(\R^k)^*$ and $\rho$ the natural representation of $\GL_k(\R)$ on $\Lambda^p(\R^k)^*$, i.e.,
	\begin{equation}
	\label{Eq_ReprLamdaP}
	\rho\colon \GL_{k}(\R)\times \Lambda^p(\R^k)^*  \to \Lambda^p(\R^k)^*,\qquad (g,\a)\mapsto \a\bigl (g^{-1}\cdot\;,\dots , g^{-1}\cdot \;\bigr ).
	\end{equation}
	We have the map
	\begin{equation*}
	\Fr(E)\times \Lambda^p(\R^k)^*\to \Lambda^pE^*,\qquad  (e, \a)\mapsto \a\bigl ( e^{-1}\cdot\;,\dots, 	e^{-1}\cdot\;)
	\end{equation*}
	where we think of a frame $e$ as an isomorphism $\R^k\to E_{\pi(e)}$.
	This induces an isomorphism of vector bundles 
	\begin{equation*}
		\Fr(E)\times_{\rho} \Lambda^p(\R^k)^*\cong \Lambda^pE^*.
	\end{equation*} 
	Thus, $\Lambda^pE^*$ can be recovered as the vector bundle associated with $(\Fr(E), \rho)$, where $\rho$ is given by~\eqref{Eq_ReprLamdaP}.
\end{example}

\begin{exercise}
	Let $V$ be the space of $k\times k$-matrices $M_k(\R)\cong \End(\R^k)$ viewed as a $\GL_k(\R)$--representation as follows:
	\begin{equation*}
		\rho\colon(g, A)\mapsto gA g^{-1}.
	\end{equation*}
	Show that $\Fr(E)\times_\rho \End(\R^k)$ is isomorphic to $\End(E)$. 
\end{exercise}

\medskip

Denote  
\begin{equation*}
	C^\infty(P;\, V)^G:=\bigl\{ \hat s\colon P\to V\mid s(p\cdot g) = \rho(g^{-1})\hat s(p)\quad\forall p\in P\text{ and } g\in G\, \bigr\}.`
\end{equation*}

Pick any $\hat s \in C^\infty(P;\, V)^G$ and denote by $\varpi\colon P\times V\to P\times_\rho V$ the natural projection. 
Consider the map
\begin{equation*}
	\varpi\comp (\id, \hat s)\colon P\longrightarrow P\times_\rho V,\qquad p\mapsto [p, \,\hat s(p)].
\end{equation*}
The letter map is $G$-invariant, where $G$ acts trivially on the target space. 
Hence, there is a unique map $s\colon M\to P\times_\rho V$ such that 
\begin{equation}
	\label{Eq_SectEquivMaps}
	s\comp\pi = \varpi\comp (\id, \hat s).
\end{equation}
In other words, $s$ is defined by requiring that the diagram
\begin{center}
	\includegraphics[width=0.3\textwidth]{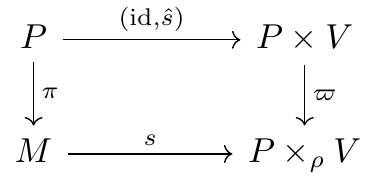}
\end{center}
commutes.  

\begin{proposition}
	\label{Prop_EquivMapsSections}
	The map
	\begin{equation*}
		C^\infty(P;\, V)^G\to \Gamma\bigl ( P\times_\rho V \bigr),\qquad \hat s\mapsto s,
	\end{equation*}
	where $s$ is defined by~\eqref{Eq_SectEquivMaps}, is a bijection. 
\end{proposition}
\begin{proof}
	We only need to construct the inverse map.
	Thus, let $s$ be a section of $P\times_\rho V$. 
	For any $p\in P$ there is a unique $\hat s(p)\in V$ such that 
	\begin{equation*}
		s\bigl (\pi(p)\bigr) = [p,\; \hat s(p)].
	\end{equation*}
	It is straightforward to check that $\hat s$ is equivariant. 
\end{proof}

It will be useful below to have a description of differential forms with values in an associated bundle in the spirit of  \autoref{Prop_EquivMapsSections}. 
Before stating the claim, we need a few notions. 

For $g\in G$ denote $R_g\colon P\to P$, \ $R_g(p) = p\cdot g$. 
The infinitesimal action of the Lie algebra $\fg$ is given by the vector field
\begin{equation*}
K_\xi(p):=\Bigl.\frac d{dt}\Bigr|_{t=0} \bigl ( p\cdot\exp(t\xi)\bigr),\qquad \xi\in\fg.
\end{equation*}  

Since $G$ acts freely on the fibers, we have
\begin{equation*}
\cV_p:=\ker \pi_*|_p = \{  K_\xi(p)\mid \xi\in \fg \}\cong \fg.
\end{equation*}
$\cV_p$ is called \emph{the vertical subspace}.

\begin{defn}$\phantom{a}$
	\begin{itemize}[itemsep=-15pt]
		\item A $q$--form $\om$ with values in some $G$-representation $V$ is said to be \emph{$G$-equivariant}, if $R_g^*\,\om =\rho(g^{-1})\om$;\\
		\item A $q$--form $\om$  with values in some $G$-representation $V$ is said to be \emph{basic}, if $\om (\mathrm v_1,\dots,\mathrm v_q)=0$ whenever one of the arguments belongs to the vertical subspace. 
	\end{itemize}
\end{defn}
 
 Denote by $\Om^q_{bas}(P;\, V)^G$ the space of all basic and $G$-equivariant $q$-forms on $P$ with values in  $V$. 

\begin{proposition}
	\label{Prop_FormsPullBackPrincBundle}
	For any $a\in \Om^q\bigl (M;\, P\times_\rho V\bigr )$ the pull-back $\pi^*a$ can be viewed as an element of $\Om^q_{bas}(P;\, V)^G$. 
	Moreover, the map $a\mapsto\pi^*a$ establishes a bijective correspondence between $\Om^q\bigl (M;\, P\times_\rho V\bigr )$ and $\Om^q_{bas}(P;\, V)^G$.
\end{proposition} 
\begin{proof}
	For any tangent vectors $\hat{\mathrm v}_1,\dots, \hat {\mathrm v}_q$ to $P$ the equality
	\begin{equation*}
		a\bigl (\pi_*\hat{\mathrm v}_1,\dots, \pi_*\hat{\mathrm v}_q\bigr )= [p, \;\hat a_p(\hat{\mathrm v}_1,\dots, \hat{\mathrm v}_q)]
	\end{equation*}
	determines uniquely $\hat a\in\Om^q(P;\, V)$.
	Since by definition the vertical space is $\ker \pi_*$, the fact that $\pi^*a$ is basic is clear. 
	The equivariancy of $\hat a$ follows from a straighforward computation. 
	
	Conversely, let  $\hat a$ be given.  Pick any $m\in M$ and any $p\in\pi^{-1}(m)$.
	Pick also any  $\mathrm v_1,\dots,\mathrm v_q\in T_mM$ and choose $\hat{\mathrm v}_1,\dots, \hat {\mathrm v}_q \in T_pP$ such that $\pi_*(\hat {\mathrm v}_j) = \mathrm v_j$ for all $j\in\{ 1,\dots, q\}$. 
	These lifts do exist because $\pi_*$ is surjective, however these need not be unique. 
	With this at hand, we can define $a$ by the equality
	\begin{equation*}
		a_m\bigl (\mathrm v_1,\dots, \mathrm v_q\bigr )= [p, \;\hat a_p(\hat{\mathrm v}_1,\dots, \hat{\mathrm v}_q)].	
	\end{equation*}
	Since $\hat a$ is basic and equivariant, $a$ does not depend on the choices involved.   
\end{proof}

\subsubsection{Connections on principal bundles}

The Lie algebra $\fg$ can be viewed as a $G$-representation, where the action is the adjoint one: $\ad\colon G\to \GL(\fg)$.
For example, if $G$ is a subgroup of $\GL_k(\R)$, then 
\begin{equation*}
	\ad_g\xi = g\,\xi\, g^{-1}.
\end{equation*}

\begin{defn}
	\emph{A connection form}, or simply \emph{a connection}, on a principal $G$-bundle $P$ is a $G$--equivariant 1--form $a$ with values in the Lie algebra $\fg$ such that 
	\begin{equation*}
		a (K_\xi) = \xi\qquad \forall\xi\in\fg.
	\end{equation*}
\end{defn}

Denote 
\begin{equation*}
	\ad P = P\times_{\ad}\fg.
\end{equation*}

\begin{thm}
	\label{Thm_SpaceOfAllConnPrincBundle}
	For any principal bundle the space of all connections $\cA(P)$ is an affine space modelled on $\Om^1(\ad P)$.
\end{thm}

This theorem  can be proved in the same manner as~\autoref{Thm_ConnIsAnAffSpace}. 
Instead of going through the details, we describe the only essential modification of the argument used in the proof of~\autoref{Thm_ConnIsAnAffSpace}. 
Namely, given $a,\,a'\in\cA(P)$ the difference $b:=a -a'$  is basic and $G$-invariant so that by applying \autoref{Prop_FormsPullBackPrincBundle} one can think of $b$ as a 1--form on $M$ with values in $\ad P$.

Notice that \autoref{Thm_SpaceOfAllConnPrincBundle} states in particular, that $\cA(P)$ is non-empty. 

\begin{example}
	\label{Ex_ConnHopfBundle}
		$\U(1)$ acts freely on 
		\begin{equation*}
		S^{2n+1}:=\bigl \{ z\in \C^{n+1}\mid |z_0|^2+\dots + |z_n|^2 = 1 \bigr \}
		\end{equation*}
		in the diagonal manner. 
		This represents $S^{2n+1}$ as the total space of the $\U(1)$-bundle
		\begin{equation}
		\label{Eq_HopfFibration}
		\pi\colon S^{2n+1}\to \CP^n, \qquad z\mapsto [z],
		\end{equation}
		which is sometimes called the \emph{Hopf fibration}.

		The infinitesimal action of $U(1)$ on $S^{2n+1}$ is given by the vector field 
		\begin{equation*}
		v(z) = \bigl (iz_0,\dots, iz_n \bigr ). 
		\end{equation*}
		There is a unique connection $a\in\Om^1\bigl (S^{2n+1}; \R i \bigr)$ such that $\ker a = v^\perp$. 
		Explicitly, 
		\begin{equation*}
		a_z(u) = \langle v(z), u\rangle\, i,\qquad \text{where } u\in T_zS^{2n+1}. 
		\end{equation*}
		Since the action of $\U(1)$ preserves the Riemannian metric of the sphere, $a$ is an invariant 1-form. 
		It remains to notice that for abelian groups the notions of equivariant and invariant forms coincide.		
		 Thus, $a$ is a connection 1-form on~\eqref{Eq_HopfFibration}.
\end{example}

\begin{exercise}
	\label{Exerc_TautolLineBundle}
	Prove that the associated bundle $S^{2n+1}\times_{\U(1)}\C$ is (canonically) isomorphic to \emph{the tautological line bundle:}
	\begin{equation*}
	\cO(-1):=\Bigl\{ \bigl( [z], w \bigr)\in \CP^n\times \C^{n+1}\mid w\in [z]\cup \{ 0 \} \Bigr \}.
	\end{equation*}
\end{exercise} 

\medskip

The definitions of a connection on a vector and principal bundles differ significantly and the reader may wonder what is the relation between these two notions.
The following result yields an answer to this question.

First notice that by differentiating $\rho\colon G\to \GL(V)$ we obtain an action of $\fg$ on $V$, i.e., a Lie algebra homomorphism 
\begin{equation}
\label{Eq_DRho}
\rho_*|_{\mathbbm 1}\colon \fg \to \End(V).
\end{equation}  
This way one can think of a connection $a$ as a 1-form with values in $\End(V)$ provided a representation $V$ is given.  

\begin{thm}$\phantom{a}$
\begin{enumerate}[(i),itemsep=-3pt]
	\item \label{It_FromConnOnVBToFrameBundle}
	Let $E$ be a vector bundle. Any connection $\nabla$ on $E$ determines a unique connection $a$ on the frame bundle such that
	\begin{equation*}
		e^*a = A(\nabla, e).
	\end{equation*}
	Here we think of a local frame $e$ as a local section of $\Fr(E)$ and $A(\nabla, e)$ is the local connection 1--form of $\nabla$ with respect to $e$. 
	\item \label{It_ConnOnAssBundle}
	Let $P$ be a principal $\rG$--bundle. 
	Any connection $a$ on $P$ induces a unique connection $\nabla$ on any associated vector bundle $P\times_\rho V$ such that
	\begin{equation}
	\label{Eq_ConnInTermsOfEquivMaps}
		\pi^*\nabla s = d\hat s + a\cdot \hat s.
	\end{equation} 
	Here the right hand side is a basic and $G$-invariant 1--form on $P$ and the equality is understood in the sense of \autoref{Prop_FormsPullBackPrincBundle}. 
\end{enumerate}	
\end{thm} 
\begin{proof}
	The proof consists of a number of the following steps.
	\setcounter{step}{0}
	\begin{step}
		\label{Step_DiffOfRhat}
		Denote
		\begin{equation*}
			\hat R\colon P\times G\longrightarrow P,\qquad \hat R(p,g) = p\cdot g.
		\end{equation*}
		Then the differential $\hat R_*$ at $(p,g)$ satisfies:
		\begin{equation*}
			\hat R_* (\mathrm v, \mathrm w) = (R_g)_*\mathrm v + K_{(L_{g^-1})_*\mathrm w}(p\cdot g),
		\end{equation*}
		where $L_{g^{-1}}\colon G\to G,\ h\mapsto g^{-1}h$, $\mathrm v\in T_pP$, and $\mathrm w\in T_gG$.  
	\end{step}

	The proof of this step is a simple computation, which is left as an exercise.
	
	\begin{step}
		\label{Step_ConnFormsLoc}
		Let $e$ be a local frame of $E$ over $U$.  
		For any 1-form $A\in\Om^1(U;\,\gl_k(\R))$ there is a unique connection $a$ on $\Fr(E)|_U$ such that
		\begin{equation*}
			e^*a = A.
		\end{equation*}  
	\end{step}

	Since $e$ is a section of $\Fr(E)|_U$, we have $\pi_*\comp e_*=\id_U$ and therefore $e_*$ is injective. 
	Hence, by the dimensional reasons, we have $T_{e(m)}P = \cV_{e(m)}\oplus \Im e(m)_*$. 
	Therefore, we can define 
	\begin{equation*}
		a_{e(m)}\bigl (K_\xi + e_*\mathrm v\bigr) = \xi + A(\mathrm v).
	\end{equation*}
	This extends to a unique $G$-invariant 1-form on $\Fr(E)|_{U}$. 
	The equality 
	\begin{equation*}
		(R_g)_* K_\xi  = K_{\ad_{g^{-1}}\xi},
	\end{equation*}
	which can be checked by a straightforward computation, implies that $a$ is a connection. 
	
	\begin{step}
		We prove~\textit{\ref{It_FromConnOnVBToFrameBundle}}.
	\end{step}

	Choose two local frames $e$ and $e'$, which we may assume to be defined on the same open set $U$ (restrict to the intersection of the corresponding domains if necessary). 
	The equalities
	\begin{equation*}
		a_{e(m)}\bigl ( e_*\cdot\; \bigr) = A_m(\cdot)\qquad\text{and}\qquad 
		a_{e'(m)}'\bigl ( e'_*\cdot\; \bigr) =A'_m(\cdot) 
	\end{equation*}	
	determine unique $\GL_k(\R)$--invariant 1-forms $a$ and $a'$ on $\Fr(E)|_U$ by Step~\ref{Step_ConnFormsLoc}.
	
	We have
	\begin{align*}
		e'^*a (\cdot )&=a_{e'}\bigl (e'_*\;\cdot\; \bigr ) = a_{e\cdot g}\Bigl (\bigl (\hat R\comp (e,g)\bigr)_*\;\cdot\;\Bigr)
		=a_{e\cdot g}\Bigl (\bigl (R_g)_*\comp e_*\;\cdot\;\Bigr) + a_{e\cdot g}\bigl (K_{g^{-1}dg}\bigr (e\cdot g))\\
		&=g^{-1}a_e(e_*\cdot)g + g^{-1}dg = A'(\cdot).
	\end{align*} 
	Here the second equality follows by Step~\ref{Step_DiffOfRhat} and the last one by~\eqref{Eq_ConnMatrGauge}. 
	Thus, we conclude that $a = a'$ on the intersection of the domains of local frames $e$ and $e'$. 
	Thus, $a$ is globally well-defined.
	
	\begin{step}
		We prove~\textit{\ref{It_ConnOnAssBundle}}.
	\end{step}
	
	By the equivariancy of $\hat s\colon P\to V$ we have
	\begin{equation*}
		\bigl (d\hat s + a\cdot \hat s\bigr )(K_\xi) = -\xi\cdot \hat s + a(K_\xi)\cdot \hat s = 0,
	\end{equation*}
	i.e., $d\hat s + a\cdot \hat s$ is a basic 1-form, which is $G$-invariant as well. 
	Hence, there is a 1-form $\nabla s$ on $M$ such that~\eqref{Eq_ConnInTermsOfEquivMaps} holds. 
	Moreover, for any $G$-invariant function $f$ on $P$ we have the equality
	\begin{equation*}
		d(f\cdot \hat s) + a\cdot (f\cdot \hat s) = df\otimes \hat s + f \bigl (d\hat s + a\cdot  \hat s\bigr ),
	\end{equation*}
	which shows that $\nabla$ is a connection. 
	This finishes the proof of this theorem.
\end{proof}

One immediate corollary of this theorem is that a connection on $E$ induces a connection on $E^*, \ \End(E)$ ans so on. 
Indeed, these bundles can be interpreted as associated bundles for $P=\Fr(E)$ and the corresponding $\GL_k(\R)$--representations.
A more direct characterization of the induced connections is given in the following.
\begin{exercise}Let $\nabla$ be a connection on $E$.
	Show that the following holds.
	\begin{enumerate}[(i),itemsep=-3pt]
		\item There is a unique connection on $E^*$, still denoted by $\nabla$, such that
		\begin{equation*}
			d \langle \a, s\rangle = \langle \nabla\a, s\rangle + \langle \a, \nabla s\rangle\qquad \forall \a\in\Gamma(E^*)\qandq\forall s\in\Gamma(E),
		\end{equation*}
		where $\langle\cdot,\cdot\rangle$ denotes the natural pairing $E^*\otimes E\to \R$.
		This connection coincides with the induced one. 
		\item There is a unique connection on $\End(E)$, still denoted by $\nabla$, such that
		\begin{equation*}
		\nabla \bigl (\varphi( s) \bigr) = (\nabla\varphi )(s) + \varphi (\nabla s)\qquad \forall \varphi\in\Gamma(\End(E))\qandq\forall s\in\Gamma(E),
		\end{equation*}
		This connection coincides with the induced one.
	\end{enumerate}	
\end{exercise}   

\subsubsection{The curvature of a connection on a principal bundle}

Local expression~\eqref{Eq_CurvLocLieBrackets} of the curvature form suggests the following construction.
Let $a$ be a connection on a principal $G$--bundle $P$. 
The 2-form $da + \frac 12[a\wedge a]$ with values in $\fg$ is clearly $G$-invariant. 
This form is also basic as the following computation shows:
\begin{align*}
	\imath_{K_\xi}\bigl (  da +\frac 12 [a\wedge a] \bigr) 
	&= \cL_{K_\xi}a - d\bigl ( \imath_{K_\xi}a\bigr) + \frac 12 \bigl [a(K_\xi), a(\cdot)\bigr ] -\frac 12 \bigl [a(\cdot), a(K_\xi)\bigr]\\
	&= -[\xi,a] + 0 +\frac 12[\xi,a] + \frac 12 [\xi, a] \\
	&=0.
\end{align*} 
Here $\cL_K$ denotes the Lie derivative with respect to a vector field $K$ and the first equality uses Cartan's magic formula
\begin{equation*}
	\cL_{K} = d\,\imath_K + \imath_K d.	
\end{equation*} 
Hence, by \autoref{Prop_FormsPullBackPrincBundle} we obtain that there is some $F_a\in \Om^2(M;\ad\, P)$ such that
\begin{equation}
	\label{Eq_CurvPrincBundle}
	\pi^*F_a = da + \frac 12 [a\wedge a].
\end{equation}

\begin{defn}
	The 2--form $F_a$ defined by~\eqref{Eq_CurvPrincBundle} is called \emph{the curvature form} of $a$.  
\end{defn}

\begin{example}
	\label{Exam_CurvatureTautBundle}
	Let us compute the curvature of the connection $a$ that was constructed in Example~\ref{Ex_ConnHopfBundle} in the simplest case $n=1$. 
	On the 3-sphere we have the following vector fields
	\begin{equation*}
		v_1:= (-x_1, x_0, -x_3, x_2), \quad 
		v_2:= (-x_2, x_3, x_0, -x_1),\qand\quad 
		v_3:= (-x_3, -x_2 ,x_1 , x_0),
	\end{equation*}
	which at each point yield an orthonormal oriented basis of the tangent space to the sphere. 
	It is worthwhile to notice that $v_1$ coincides with  the vector field $v$ from Example~\ref{Ex_ConnHopfBundle}. 
	
	By the definition of the connection form $a$, we have
	\begin{equation*}
		a =\bigl ( -x_1\, dx_0 + x_0\, dx_1 -x_3\, dx_2 + x_2\, dx_3 \bigr )\, i.
	\end{equation*} 
	Hence,
	\begin{equation*}
		\pi^*F_a = da = 2\bigl ( dx_0\wedge dx_1 + dx_2\wedge dx_3 \bigr )\, i.
	\end{equation*}
	In particular, we have
	\begin{equation*}
		F_a(\pi_*v_2, \pi_*v_3) = \pi^*F_a(v_2, v_3) = 2\, i.
	\end{equation*}
	It can be shown that the quotient metric on $S^3/\U(1)$ yields the round metric on the sphere of radius $1/2$. 
	Explicitly, the corresponding isometry is given by 
	\begin{equation*}
		(z_0, z_1)\mapsto \bigl ( z_0\bar z_1, \tfrac 12(|z_0|^2 - |z_1|^2) \bigr ). 
	\end{equation*}   
	Here we think of $S^3$ as a subset of $\C^2$.
	
	Since $(\pi_*v_2, \pi_*v_3)$ is an oriented orthonormal basis of the tangent space of $S^2_{1/2}$, we conclude that $F_a = 2\,\vol_{S^2_{1/2}}\, i$, where $\vol_{S^2_{1/2}}$ is the volume form of the standard round metric on $S^2_{1/2}$. 
	
	Notice that 
	\begin{equation*}
		\int_{S^2} F_a =2\Vol (S^2_{1/2})\, i = 2\pi\, i. 
	\end{equation*}
\end{example}

The proof of the following proposition is left as an excercise. 

\begin{proposition}
	Let $a$ be a connection on a principal bundle $P$ and let $V$ be a  $G$--representation.
	Using~\eqref{Eq_DRho} we can think of $F_a$ as a 2-form with values in $\End(P\times_\rho V)$. 
	With these identifications in mind, the curvature of the induced connection $\nabla$ on $P\times_\rho V$ equals  $F_a$.\qed
\end{proposition}

Given a local trivialization of $P$, i.e., a section $\sigma$ over an open subset $U\subset M$, we say that $A:=\sigma^*a\in\Om^1(U;\fg)$ is a local representation of $a$ with respect to $\sigma$.
Then over $U$ we have
\begin{equation*}
	F_a = \sigma^*\pi^* F_a = \sigma^*\bigl (da + \frac 12 [a\wedge a]\bigr ) = dA + \frac 12 [A\wedge A]. 
\end{equation*}

One of the most basic properties of the curvature form is the so called Bianchi identity.

\begin{proposition}[Bianchi identity]
	\label{Prop_BianchiId}
	Let $a$ be a connection on $P$. Then the curvature form $F_\nabla$ satisfies
	\begin{equation*}
		d_{\nabla^a} F_a = 0,
	\end{equation*}
	where $\nabla^a$ is the connection on $\ad P$ induced by $a$.
\end{proposition}
\begin{proof}
	Let $A$ be a local representation of $a$ as above.
	We have
	\begin{equation*}
		\begin{aligned}
		d_{\nabla^a} F_a &= d\Bigl ( dA + \frac 12[A\wedge A] \Bigr ) + \Bigl [ A\wedge \bigl ( dA + \frac 12[A\wedge A] \bigr ) \Bigr ]\\
		& = \frac 12 [dA\wedge A] -\frac 12[A\wedge dA]+ [A\wedge dA] +\frac 12 \bigl [ A\wedge [A\wedge A] \bigr ]\\
		&=0.
		\end{aligned}
	\end{equation*}
	Here the first three summands sum up to zero because $[\om,\eta] = -[\eta,\om]$ for any $\om\in \Om^2(U;\, \fg)$ and $\eta\in\Om^1(U;\,\fg)$; The vanishing of the last summand follows from the Jacobi identity. 
\end{proof}

\begin{remark}
	Thinking of $\hat F_A:=\pi^*F_A = dA +\frac 12 [A\wedge A]$
	 also as a curvature form of $A$, the computation in the proof of \autoref{Prop_BianchiId} yields the following equivalent form of the Bianchi identity:
	 \begin{equation}
	 \label{Eq_BianchiIdUpstairs}
	 	d\hat F_A = [\hat F_A\wedge A]. 
	 \end{equation}
	 Indeed, this can be obtained by the computation: $d\hat F_A = [dA\wedge A] = [\hat F_A\wedge A]$.
\end{remark}

\medskip

Given a bundle $P\to M$ and a map $f\colon N\to M$ we can construct the pull-back bundle $f^*P\to N$ just like in the case of vector bundles.
Informally speaking, we have the equality of fibers $(f^*P)_n = P_{f(n)}$ and this defines $f^*P$.
More formally, define 
\begin{equation*}
	f^*P:= \bigl \{  (p,n)\in P\times N \mid f(n)=\pi (p)  \bigr \}. 
\end{equation*}
It is easy to see that $f^*P$ is a submanifold of $P\times N$. Moreover, we have a projection $\varpi\colon f^*P\to N$, which is just the restriction of the natural projection $(p,n)\mapsto n$. 
The structure group $\rG$ clearly acts on $P$ such that the action on the fibers $\varpi^{-1}(n)=P_{f(n)}$ is transitive.  
Moreover, we also have a natural $\rG$-equivariant map $\hat f\colon f^*P\to P$, which covers $f$, i.e., the diagram 
\begin{center}
	\includegraphics[width=0.28\textwidth]{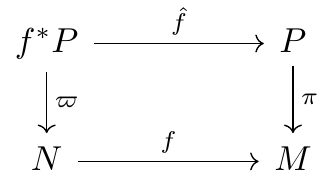}
\end{center}
commutes.

For future use we note the following statement, whose proof is left as an exercise.

\begin{proposition}
	\label{Prop_PullBackConn}
	Let $A$ be a connection on $P$. Then $f^*A:=\hat f^*A$ is a connection on $f^*P$ and $F_{f^*A} = f^*F_A$. \qed
\end{proposition}

\subsubsection{The gauge group}

For a principal bundle $P\to M$, the automorphism group
\begin{equation*}
	\cG(P):= \bigl\{ \psi\colon P\to P\mid \pi\comp\psi = \pi,\ \psi (pg) = \psi(p)g\quad \forall p\in P, \ g\in G\, \bigr \}
\end{equation*}
is called the gauge group. 

Since $\psi(p)$ lies in the same fiber as $p$ and $\rG$ acts transitively on the fibers, we can write
\begin{equation}
	\label{Eq_AuxAutomOfPrincBundle}
	\psi(p) = p\, \hat f(p),
\end{equation}
where $\hat f\colon P\to \rG$ is some map. 
The equivariancy of $\psi$ is then equivalent to 
\begin{equation}
	\label{Eq_AuxEquivOfF}
	\hat f(pg) = g^{-1}pg.	
\end{equation}
In other words, $\hat f$ can be identified with some section $f$ of the bundle
\begin{equation*}
	Ad\, P:= P\times_\rG \rG,
\end{equation*}
where $G$ acts on itself by conjugation. 

Conversely, given $f\in\Gamma(Ad\, P)$, we can construct an automorphism $\psi$ of $P$ via~\eqref{Eq_AuxAutomOfPrincBundle}. 
In other words we have a natural bijective map
\begin{equation*}
	\cG(P)\to \Gamma(Ad\, P). 
\end{equation*}

\begin{exercise}
	Fix some non-negative integer $k$. 
	Prove that $C^k(M;\, Ad\, P)$  is a Banach Lie group. 
	Moreover, show that the Lie algebra of this group is $C^k(M;\, \ad P)$.  
\end{exercise}

\begin{example}
	Assume $\rG$ is abelian. 
	Then~\eqref{Eq_AuxEquivOfF} just means that $\hat f$ is an \emph{invariant} map so that the corresponding section $f$ can be identified with a map $M\to \rG$. 
	This shows that \emph{for an abelian group} $\rG$ we have
	\begin{equation*}
		\cG(P)\cong C^\infty(M;\, \rG).
	\end{equation*}
\end{example}

Just like in the case of vector bundles, the gauge group acts naturally on the space of all connections, i.e., we have a map
\begin{equation}
\label{Eq_GaugeGpActionPrincBundles}
\cA(P)\times\cG(P) \to \cA(P),\qquad (a,\psi)\mapsto a\cdot\psi:=\psi^*a.
\end{equation}

\begin{exercise}
	Let $E:=P\times_{G,\,\rho} V$ be an associated vector bundle. 
	Show that the following: 
	\begin{enumerate}[(i),itemsep=-3pt]
		\item The representation $\rho$ induces a natural homomorphism of gauge groups $\gamma\colon\cG(P)\to \cG(E)$;
		\item The corresponding infinitezimal map  $d\rho_e\colon\mathfrak g\to \End V$ induces a Lie algebra homomorphism $C^k(M;\, \ad P)\to C^k(M;\, \End E)$;
		\item $\nabla^{a\cdot g} = g^{-1}\nabla^a g$. Here $a\in \cA(P)$, $g\in \Gamma(Ad\, P)$,  and $\nabla^a$ denotes the connection on $E$ induced by $a$.
	\end{enumerate}
\end{exercise}

\begin{exercise}
	Show that the infinitezimal action of the gauge group  is given by 
	\begin{equation*}
		\cA(P)\times \Gamma(\ad P)\to \Om^0(\ad P),\quad  (a, \xi) \mapsto - d_a\xi.
	\end{equation*}
\end{exercise}

\subsection{The Levi--Civita connection}

As we have seen above, any vector bundle of positive rank over a manifold of positive dimension admits a large family of connections.
Moreover, in general there is also no preferred connections.
The situation is somewhat different in the case of the tangent bundle of a manifold, where it turns out that a choice of a preferred connection does exist under certain circumstances.

Thus, pick a connection $\nabla$ on $TM\to M$.
For any vector fields $v,w$ on $M$ define \emph{the torsion} of $\nabla$ by
\begin{equation*}
	T(v,w):= \nabla_v w -\nabla_w v -[v, w].
\end{equation*} 

Clearly, $T$ is antysymmetric. 
It is also easy to check that $T$ is tensorial in both $v$ and $w$, i.e., $T(f_1 v, f_2w)= f_1f_2T(v,w)$ for any $f_1, f_2\in C^\infty(M)$.  
Hence, applying \autoref{Lem_Tensorial} we obtain that $T\in\Om^2(M;\, TM)$. 

\begin{defn}
	A connection $\nabla$ on the tangent bundle $TM$ is called \emph{torsion-free}, if $T\equiv 0$, i.e., if the following
	\begin{equation}
		\label{Eq_TorsFree}
		\nabla_v w - \nabla_w v = [v, w]
	\end{equation}
	holds for any vector fields $v$ and $w$.
\end{defn}

\begin{exercise}
	Let $\nabla$ be a connection on the tangent bundle. 
	Slightly abusing notations, denote also by $\nabla$ the induced connection on $T^*M$. 
	Consider the map
	\begin{equation}
		\label{Eq_AuxDeRhamDiff}
		\Om^1(M)\xrightarrow{\ \nabla\ }\Gamma(T^*M\otimes T^*M)\xrightarrow {\ Alt\ }\Om^2(M),
	\end{equation}	
	where the last map is the natural projection (alternation) $T^*M\otimes T^*M\to \Lambda^2T^*M$.
	Show that $\nabla$ is torsion-free if and only if~\eqref{Eq_AuxDeRhamDiff} coincides with the de Rham differential. 
\end{exercise}

\begin{thm}
	For any Riemannian manifold $(M, g)$ there is a unique metric torsion-free connection $\nabla$.\qed 
\end{thm}

The proof of this classical result, which is sometimes called the fundamental theorem of Riemannian geometry, can be found for instance in~\cite{Salamon:89}*{Prop.\,2.1}.
The unique connection described in the above theorem is called \emph{the Levi-Civita connection.}
It is uniquely characterized by~\eqref{Eq_TorsFree} together with
\begin{equation*}
	d \bigl( g(v, w) \bigr) = g(\nabla v, w) + g(v, \nabla w).
\end{equation*} 

The curvature form $R$ of the Levi-Civita connection can be viewed as a section of $\Lambda^2T^*M\otimes \End (TM)$.
This can be decomposed into various components.
For example, in terms of a local orthonormal frame $e_i$ we can define a quadratic form
\begin{equation*}
	Ric_g(\rv,\rw):= - \sum_i g\bigl ( R(e_i, \rv)e_i, \rw \bigr),
\end{equation*}
which is called \emph{the Ricci curvature.}
The trace of the Ricci curvature, i.e.,
\begin{equation*}
	s_g:=\sum_i Ric(e_i, e_i)
\end{equation*}
is called \emph{the scalar curvature}.
These are important characteristics of the  metric $g$. 
An interested reader may wish to consult~\cite{Joyce07_RiemannianHolGpsCalibr, Salamon:89, Besse08_EinsteinMflds} for more information on these matters. 
For us, the importance of the scalar curvature is explained by its appearance in the Weitzenb\"ock formula below, see~\autoref{Cor_WeitzenboeckForm}.

\subsection{Classification of $\U(1)$ and $\SU(2)$ bundles}

It is more convenient in this section to work in a topological category rather than the smooth one.
The reader will have no difficulties to adopt the corresponding notions to this setting.

\medskip

\begin{defn}
	Let $G$ be a compact Lie group. 
	A topological space $E$ equipped with an action of $G$ is said to be \emph{a classifying bundle} for $G$, if $E$ is contractible and the $G$--action is free. 
\end{defn} 

Denoting $B:=E/G$, we obtain a natural projection $E\to B$ so that $E$ could be thought of as a principal\footnote{The notations $E$ and $B$ are traditional in this context so I will keep to this tradition. The reader should not be confused by the fact that here $E$ denotes a principal rather than vector bunlde.} $G$--bundle over $B$.
If $E$ exists, it is easy to see that $E$ is unique up to a homotopy equivalence. 

One can prove that for a compact Lie group a classifying space always exist. An interested reader may wish to consult~\cite{GuilleminSternberg99_SUSYEquivDeRham}*{Sect.\,1.2}.
Also, we take the following result as granted.

\begin{thm}[\cite{GuilleminSternberg99_SUSYEquivDeRham}*{Thm.\,1.1.1 and Rem. 2}]
	\label{Thm_PrincipalBundlesHomotopyClasses}
	Let $P\to M$ be a (topological) principal $G$--bundle over a manifold $M$. 
	Then there exists a continuous map $f\colon M\to B$ such that  $P$ is isomorphic to $f^*E$.  
	In fact $f$ is unique up to a homotopy so that the map
	\begin{equation*}
		f\mapsto f^*E
	\end{equation*}
	yields a bijective correspondence between the set of isomorphism classes of principal  $G$-bundles over $M$ and the set $[M;\, B]$ of homotopy classes of maps $M\to B$. \qed
\end{thm}

In some sense, the above theorem yields a classification of  principal---hence, also vector---bundles.
In praxis it is not always easy to describe the set $[M;\, B]$ though. 
One way to deal with this problem is via the so called characteristic classes, which I describe next.

\begin{defn}
	Let $f\colon M\to B$ be a continuous map.
	Pick any  $c\in H^\bullet (B;\, R)$, where $R$ is a ring. 
	Then $f^*c\in H^\bullet(M;\, R)$ is called a characteristic class of $f^*E$. 
\end{defn}

It is worth pointing out that characteristic classes depend on the isomorphism class of the bundle only. 

Usually, characteristic classes are easier to deal with than the set of homotopy classes of maps.  
The most common choices for the ring $R$ are $\Z,\ \Z/n\Z,\ \R$, and $\C$. 

In some cases the classifying bundle can be constructed fairly explicitly.
I restrict myself to the following two cases, namely $G=\U(1)$ and $G=\SU(2)$, which are most commonly used in gauge theoretic problems. 

\subsubsection{Complex line bundles}

It follows from Exercise~\ref{Exerc_HermStructures} that the classification problems for complex line bundles and principal $\U(1)$-bundles are equivalent.
Even though what follows below can be described in terms of vector bundles only, the language of $\U(1)$--bundles has certain advantages and will be mainly used below. 

Consider the following commutative diagram
\begin{center}
	\includegraphics[width=0.5\textwidth]{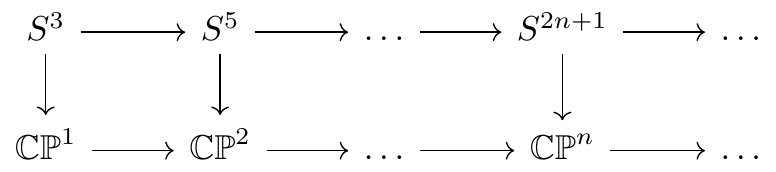}
\end{center}
where the horizontal arrows are natural inclusions. 
For example, the inclusion of the spheres is given by 
\begin{equation*}
	\C^n\supset S^{2n-1}\ni\quad (z_0,\dots, z_{n-1})\mapsto  (z_0,\dots, z_{n-1}, 0)\quad \in S^{2n+1}\subset \C^{n+1},
\end{equation*}
which is a $\U(1)$--equivariant map.

The direct limit construction yields a CW-complex $S^\infty$ equipped with a free  $\U(1)$-action. 
$S^\infty$ can be shown to be contractible.
Furthermore, we have also a CW-complex $\CP^\infty$ so that the natural quotient map 
\begin{equation*}
	S^\infty\to \CP^\infty
\end{equation*}
is the classifying bundle for the group $\U(1)$.

\begin{example}[Classification of line bundles on 2--manifolds]
	\label{Ex_ClassifBundlesOnSurfaces}
	Let $M$ be an oriented two--manifold. A continuous map $f\colon M\to \CP^\infty$ is homotopic to a map, which takes values in the 2-skeleton $\CP^1\subset \CP^\infty$ so that we have the equality $[M;\, \CP^\infty] = [M;\, \CP^1]$. Topologically, $\CP^1$ is just the 2-sphere so that we have a well--defined degree-map
	\begin{equation*}
		[M;\, S^2]\to \Z,\qquad [f]\mapsto \deg f,
	\end{equation*}
	which is in fact a bijection. 
	Thus, a complex line bundle $L$ on an oriented two--manifold is classified by an integer $d$, which is called the degree of $L$. 
\end{example}

It is easy to see that the cohomology ring $H^\bullet\bigl ( \CP^\infty;\, \Z\bigr )$ is generated by a single element $a\in H^2\bigl ( \CP^\infty;\, \Z\bigr )$. 
We can fix the choice of the generator by requiring 
\begin{equation}
	\label{Eq_NormOfGeneratorH2P1}
	\langle a, [\CP^1]\rangle =1,
\end{equation}
where $\langle\cdot, \cdot \rangle$ is the pairing between the homology and cohomology groups.

With this understood, to any principal $\U(1)$-bundle $P\to M$ (equivalently, to any complex line bundle $L\to M$) we can associate a cohomology class as follows. 
If $f\colon M\to \CP^\infty$ is a map such that $P$ is the pull-back of the bundle $S^\infty\to \CP^\infty$, then 
\begin{equation}
	\label{Eq_FirstChernClassDefn}
	c_1(P):= - f^*a\in H^2(M;\, \Z).
\end{equation}
The minus sign in this definition is a convention.

\begin{defn}
	\label{Defn_FirstChernClass_Top}
	The class $c_1(P)$ defined by~\eqref{Eq_FirstChernClassDefn} is called \emph{the first Chern class} of $P$.
\end{defn}

If $L\to M$ is a complex line bundle, we can choose a Hermitian scalar product, or, in other words, we can choose a $\U(1)$-structure $P\subset \Fr(L)$.  
Then 
\begin{equation*}
	c_1(L):=c_1(P)
\end{equation*}
is also called the first Chern class of $L$.

\begin{exercise}
	Check that the first Chern class of $L$ does not depend on the choice of the Hermitian scalar product on $L$.
\end{exercise}

\begin{thm}
	The first Chern class of the complex line bundle has the following properties:
	\begin{enumerate}[(i)]
		\item $c_1(\underline \C) =0$, where $\underline \C$ is the product bundle;
		\item $c_1(L_1\otimes L_2)= c_1(L_1) + c_1(L_2)$ for all line bundles $L_1$ and $L_2$ over the same base $M$;
		\item $c_1(L^\vee) = -c_1(L)$, where $L^\vee:= \Hom (L;\C)$ is the dual line bundle;
		\item $c_1(f^*L) = f^*c_1(L)$ for all line bundles $L\to M$ and all (continuous) maps $f\colon N\to M$.\qed
	\end{enumerate}
\end{thm}

\subsubsection{Quaternionic line bundles}

Let $\H$ denote the $\R$-algebra of quaternions. 
One can think of the (compact) symplectic group
\begin{equation*}
	\Sp(1):=\bigl \{ q\in\H \mid |q|^2 = q\bar q =1 \bigr \}
\end{equation*}
as a quaternionic analogue of $\U(1)$.
It is easy to see that $\Sp(1)$ is isomorphic to $\SU(2)$. 
Indeed, it is easy to write down an isomorphism explicitly: 
\begin{equation*}
	\Sp(1)\to \SU(2), \qquad q= z+ wj\mapsto 
	\begin{pmatrix}
	z & w\\
	-\bar w & \bar z
	\end{pmatrix}.
\end{equation*}  

Just like $\U(1)$ acts freely on $S^{2n+1}$, $\Sp(1)$ acts freely on  $S^{4n+3}=\bigl\{ (h_0,\cdots, h_n)\in \H^n\mid |h_0|^2+\cdots + |h_n|^2 = 1 \bigr\}$ in the diagonal manner so that we have a natural principal $\Sp(1)$-bundle:
\begin{equation*}
	S^{4n+3}\to \mathbb{HP}^n.
\end{equation*}
The commutative diagram
\begin{center}
		\includegraphics[width=0.5\textwidth]{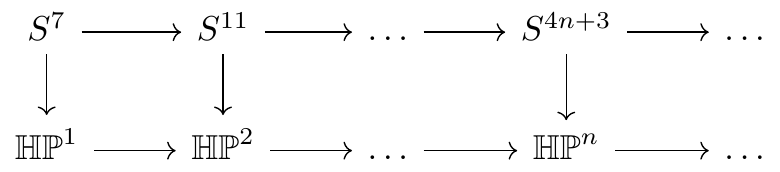}
\end{center}
leads to the construction of the classifying bundle for $\Sp(1)$:
 \begin{equation*}
 	S^\infty\to \mathbb{HP}^\infty.
 \end{equation*}
Notice that $\mathbb{HP}^\infty$ is (an infinite-dimensional) CW-complex, which has exactly one cell of dimension $4n$. 
This implies in particular that the cohomology ring $H^\bullet (\mathbb{HP}^\infty;\, \Z)$ is generated by a single generator $b$ of degree $4$. 
The choice of $b$ can be fixed for example by requiring 
\begin{equation*}
	\langle b, [\mathbb{HP}^1] \rangle =1. 
\end{equation*} 

\begin{proposition}
	\label{Prop_SU2BundlesAreTrivial}
	Let $M$ be a manifold of dimension at most $3$. 
	Any $\Sp(1)$-bundle over $M$ is trivial. 
\end{proposition}
\begin{proof}
 Given a principal $\Sp(1)$-bundle $P$, by \autoref{Thm_PrincipalBundlesHomotopyClasses} we can find a map $f\colon M\to \mathbb{HP}^\infty$ such that $f^*S^\infty$ is isomorphic to $P$. 
 Furthermore, $f$ is homotopic to a map $f_1$ that maps into the 3-skeleton of $\mathbb{HP}^\infty$, which is a point. 
 Hence, $P$ is trivial. 
\end{proof}

Let $M$ be a manifold of arbitrary dimension. 
To any map $f\colon M\to \mathbb{HP}^\infty$ we can associate a class $f^*b\in H^4(M;\, \Z)$, where $b$ is a generator of  $H^\bullet (\mathbb{HP}^\infty;\, \Z)$ as above.

\begin{defn}
	Let $P\to M$ be a principal $\Sp(1)$--bundle. If $f\colon M\to\mathbb{HP}^\infty$ is a map such that $f^*S^\infty\cong P$, then
	\begin{equation*}
		c_2(P):= - f^*b\in H^4(M;\, \Z)
	\end{equation*} 
is called \emph{the second Chern class} of $P$. 
\end{defn}

\begin{remark}
	The terminology may seem to be somewhat strange at this point. The reason is that for a principal $\U(r)$--bundle $P$ one can define $r$ characteristic classes $c_1(P), c_2(P),\dots, c_r(P)$, where $c_j(P)\in H^{2j}(M;\, \Z)$. 
	In the particular cases described above, this yields the constructions of the first and the second Chern classes. 
\end{remark}

\section{The Chern--Weil theory}
\label{Sect_ChernWeil}

\subsection{The Chern--Weil theory}

In this section we could equally well work with both $\R$ and $\C$ as ground fields. 
I opt for $\C$ mainly for the sake of definiteness. 
The modifications needed for the case of $\R$ as a ground field are straightforward. 

\medskip

Let $p\colon \fg\to \C$ be an $ad$--invariant homogeneous polynomial of degree $d$. 
This means the following:
\begin{itemize}
	\item Given a basis $\xi_1,\dots,\xi_n$ of $\fg$, the expression $p\bigl (x_1\xi_1 + \dots + x_n\xi_n\bigr )$ is a polynomial of degree $d$ in $x_1,\dots, x_n$;
	\item $p(ad_g\,\xi) = p(\xi)$ for all $g\in G$ and all $\xi\in \fg$;
	\item $p(\l \xi) = \l^d p(\xi)$ for all $\l\in \R$ and all $\xi\in\fg$.
\end{itemize} 

\begin{example}\label{Ex_InvPolynomials}$\phantom{A}$
	\begin{enumerate}[(a)]
		\item For $\fg = \fu(r)$ the map $p_d(\xi) = i\tr \xi^d$ is an $ad$-invariant polynomial of degree $d$. 
		\item Choose $\fg=\fu(r)$ again and define polynomials $c_1,\dots, c_r$ of degrees $1,\dots, r$ respectively by the equality
		\begin{equation*}
			\det \Bigl ( \l\mathbbm 1 + \frac i{2\pi}\xi \Bigr ) = \l^r + c_1(\xi)\l^{r-1} + \dots + c_r(\xi).
		\end{equation*}
		For example, $c_r(\xi) = \frac {i^r}{(2\pi)^r}\det\xi$ and $c_1(\xi) =\frac i{2\pi}\tr \xi$. 
		Notice also that the equality
		\begin{equation*}
			\overline{\det \Bigl ( \l\mathbbm 1 + \frac i{2\pi}\xi \Bigr )} = \det \Bigl ( \bar\l\mathbbm 1 + \frac i{2\pi}\xi \Bigr )
		\end{equation*}
		implies that each $c_j$ takes values in $\R$. 
	\end{enumerate}
\end{example}

\medskip

Let $P\to M$ be a principal $\rG$--bundle equipped with a connection $a\in \Om^1(P;\, \fg)$. 
Think of the curvature form $\pi^*F_a = da + \tfrac 12 [a\wedge a]$ as a matrix, whose entries are 2-forms on $P$. Since forms of even degrees commute, the expression $p(\pi^*F_a)$ makes sense as an $\R$-valued differential form of degree at most $2d$ on $P$. 
Since each entry of $\pi^*F_a$ is basic, so is  $p(\pi^*F_a)$.
Moreover, the $ad$--invariancy of $p$ implies that $p(\pi^*F_a)$ is $\rG$--invariant. 
By \autoref{Prop_FormsPullBackPrincBundle} applied in the case of the trivial $\rG$--representation we obtain that there is a form $p(F_a)$ on $M$ of degree $2d$ such that 
\begin{equation*}
	\pi^*p(F_a) = p(\pi^*F_a). 
\end{equation*}

\begin{lem}
	\label{Lem_InvPolyCurvature}
	The following holds:
	\begin{enumerate}[(i)]
		\item \label{It_AuxPFAClosed}
		$p(F_a)$ is closed;
		\item \label{It_AuxPFAindepOfConn}
		The de Rham cohomology class of $p(F_a)$ does not depend on the choice of connection $a$.
	\end{enumerate}
\end{lem}
\begin{proof}
The proof consists of the following steps.
\setcounter{step}{0}
\begin{step}
	We prove~\textit{\ref{It_AuxPFAClosed}}.
\end{step}
Pick any $\xi, \xi_1,\dots, \xi_{d}\in\fg$.
Thinking of $\fg$ as a matrix Lie algebra, I write temporarily $ad_{e^{t\xi}}\xi_j = e^{t\xi}\xi_j e^{-t\xi}$.  
Slightly abusing notations, denote by $p\colon Sym^d(\fg)\to \R$  the $d$-multilinear function whose restriction to the diagonal yields the original polynomial $p$.  
Then, differentiating the equality 
\begin{equation*}
	p\bigl (  e^{t\xi}\xi_1e^{-t\xi},\, \dots,\,   e^{t\xi}\xi_de^{-t\xi} \bigr ) = p(\xi_1,\dots,\xi_d)
\end{equation*}
with respect to $t$, yields
\begin{equation}
	\label{Eq_AuxInfinitInvariancy}
	p\bigl ([\xi, \xi_1],\, \xi_2,\dots, \xi_d  \bigr) +
	p\bigl (\xi_1,\, [\xi, \xi_2],\,\dots, \xi_d  \bigr)+\dots +
	p\bigl (\xi_1,\, \xi_2,\dots, [\xi, \xi_d]  \bigr) =0.
\end{equation}

Denote $\hat F_A:= \pi^*F_A$. 
Then~\eqref{Eq_AuxInfinitInvariancy} implies
\begin{equation*}
p\bigl ( [\hat F_A\wedge A],\, \hat F_A, \dots, \hat F_A \bigr ) + 
p\bigl ( \hat F_A, [\hat F_A\wedge A], \dots, \hat F_A \bigr ) +\dots + 
p\bigl ( \hat F_A,\,\hat F_A, \dots, [\hat F_A \wedge A]\bigr ) =0.
\end{equation*} 
Hence,
\begin{equation*}
	\begin{aligned}
	d\bigl (  p(\hat F_A)  \bigr ) &= p(d\hat F_A,\, \hat F_A, \dots,\, \hat F_A) + p(\hat F_A,\, d\hat F_A, \dots,\, \hat F_A) + \dots + p(\hat F_A,\, \hat F_A, \dots,\, d\hat F_A)\\
	&= p\bigl ( [\hat F_A\wedge A],\, \hat F_A, \dots, \hat F_A \bigr ) + 
	p\bigl ( \hat F_A, [\hat F_A\wedge A], \dots, \hat F_A \bigr ) +\dots + 
	p\bigl ( \hat F_A,\,\hat F_A, \dots, [\hat F_A\wedge A]\bigr )\\ 
	&=0.
	\end{aligned}
\end{equation*}
Here the second equality uses~\eqref{Eq_BianchiIdUpstairs}. 
This finishes the  proof of~\textit{\ref{It_AuxPFAClosed}}.

\begin{step}
	Let $I=[0,1]$ be the interval and $\imath_0,\imath_1\colon M\to M\times I$ be the natural inclusions correspodnding to the endpoints of the interval. 
	There exist linear maps $Q\colon \Om^k(M\times I)\to \Om^{k-1}(M)$ such that for any $\om \in \Om^k(M\times I)$ we have
	\begin{equation*}
		\imath_1^*\om - \imath_0^*\om = d\, Q\om - Q\, d\om. 
	\end{equation*}
\end{step}

The argument goes just like in the proof of the Poincar\'e lemma, see for example~\cite{BottTu82_DiffForms}*{Prop.\,4.1.1}. 
I omit the details.

\begin{step}
	We prove~\ref{It_AuxPFAindepOfConn}.
\end{step}

Pick any two connection $A_0$ and $A_1$ and think of $A_t:= (1-t)A_0 + tA_1$ as a connection on $\varpi^*P\to M\times I$, where $\varpi\colon M\times I\to M$ is the natural projection.
Then 
\begin{equation*}
	p(F_{A_1}) - p(F_{A_0})=\imath_1^* p(F_{A_t}) - \imath_0^*p(F_{A_t}) = d\, Qp(F_{A_t}) 
\end{equation*}
by the previous step. 
This proves~\textit{\ref{It_AuxPFAindepOfConn}}.
\end{proof}

\subsubsection{The Chern classes}

Let $c_j$ be the polynomial of degree $j$ from Example~\ref{Ex_InvPolynomials}.

Let $P\to M$ be a principal $\U(r)$--bundle with a connection $A$.
By \autoref{Lem_InvPolyCurvature}, $c_j(F_A)$ is closed, real valued, and the de Rham cohomology class of $c_j(F_A)$ does not depend on the choice of the connection.

\begin{defn}
	\label{Defn_ChernClassCurvature}
	 The class $c_j(P):= [c_j(F_A)]\in H^{2j}_{dR}(M;\R)$ is said to be the $j$th Chern class of $P$ and 
	 \begin{equation*}
	 	c(P):=1 + c_1(P) + \dots+ c_r(P)\in H^\bullet(M;\R)
	 \end{equation*} 
	 is called the total Chern class  of $P$.
\end{defn}

\begin{rem}
	The above definition yields Chern classes as elements of the de Rham cohomology groups only. 
	In fact, one can show that they lie in the image of  $H^\bullet (M;\Z)\to H^\bullet_{dR}(M;\R)$. 
	I will discuss this briefly in the case of the first two Chern classes below.
	 Also, at this point we have two a priori unrelated definitions of the first two Chern classes. 
	We will see below that in fact they agree.  
\end{rem}

\begin{rem}
	Let  $E$ be a complex vector bundle of rank $r$. 
	Choosing a fiberwise Hermitian structure on $E$, we obtain a principal $\U(r)$--bundle $\Fr_\U$ so that we can define $c_j(E):=c_j(\Fr_\U)$. 
	It is easy to show that this does not depend on the choice of the Hermitian structure. 
\end{rem}

\begin{thm}
	The Chern classes satisfy the following properties
		\begin{enumerate}[(i)]
			\item $c_0(E) =1$ for any vector bundle $E$;
			\item $c(f^*E)=f^*c(E)$ for all vector bundles $E\to M$ and all maps $f\colon N\to M$;
			\item $c(E_1\oplus E_2) = c(E_1)\cup c(E_2)$;
			\item $c\bigl (\cO(-1)\bigr ) =1 -a$, where $\cO(-1)$ is the tautological line bundle over $\mathbb P^1$ and $a$ is the  generator of the cohomology group of $\mathbb P^1$ such that~\eqref{Eq_NormOfGeneratorH2P1} holds.
		\end{enumerate}
\end{thm}

The first property above is jut the definition, the remaining properties  are called naturality, Whitney sum formula, and normalization respectively. 

\begin{proof}
	The naturality follows immediately from \autoref{Prop_PullBackConn}. 
	The normalization is equivalent to $\frac i{2\pi}\int_{\mathbb P^1} F_a = -1$, where $a$ is a unitary connection on the tautological line bundle. 
	This was established in Example~\ref{Exam_CurvatureTautBundle}. 
	
	Thus, we only need to prove the Whitney sum formula.  
	If $\nabla_1$ and $\nabla_2$ are unitary connections on $E_1$ and $E_2$ respectively, then the curvature of the corresponding connection on the Whitney sum is a block diagonal matrix. 
	More precisely, this means that the curvature is a 2-form with values in $\End(E_1)\oplus \End(E_2)$.   
	If  $A$ and $B$ are any square matrices,  we have
	\begin{equation*}
	 \begin{aligned}
		 \det 
		 \begin{pmatrix}
		 A & 0\\
		 0 & B
		 \end{pmatrix}
		 =\det A\, \det B\quad\implies\quad
		 c(F_{\nabla_1\oplus \nabla_2}) &=\det 
		 \begin{pmatrix}
		 \mathbbm 1 + \tfrac i{2\pi}F_{\nabla_1} & 0\\
		 0 & \mathbbm 1 + \tfrac i{2\pi} F_{\nabla_2}
		 \end{pmatrix}\\
		 &=c(F_{\nabla_1})\wedge c(F_{\nabla_2}).
	 \end{aligned}
	\end{equation*} 
	The latter equality clearly implies the Whitney sum formula.
\end{proof}

\begin{exercise}
	Let $E$ be a vector bundle. Prove that the following holds:
	\begin{enumerate}[(a)]
		\item The Chern classes depend on the isomorphism class of $E$ only;
		\item $c_j(E^\vee) = (-1)^jc_j(E)$ for all $j$;
		\item If $E$ is trivial, then $c(E)=1$;
		\item If $E\cong E_1\oplus \underline{\C}^k$, then $c_j(E) =0$ for $j > \rk E-k$. 
	\end{enumerate}
\end{exercise}

\begin{exercise}
	Show that the tangent bundle of $S^2$ is non-trivial. 
\end{exercise}

\begin{thm}
	Let $L$ be a complex line bundle. 
	Then the first Chern class in the sense of Definition~\ref{Defn_ChernClassCurvature} coincides with the image in $H^2_{dR}(M;\,\R)$ of the first Chern class in the sense of Definition~\ref{Defn_FirstChernClass_Top}.
\end{thm}
\begin{proof}[Sketch of proof]
	Let $L\to M$ be a complex line bundle. It is not too hard to show that there is $N<\infty$ and a smooth map $f\colon M\to \mathbb P^N$ such that $f^*\cO(-1)$ is isomorphic to $L$, where $\cO(-1)$ is the tautological bundle of $\mathbb P^N$. 
	Notice that  $H^2_{dR}(\mathbb P^N;\, \R)$ is one dimensional and generated by the class Poincar\'e dual to $[\mathbb P^1]$, where $\imath\colon\mathbb P^1\subset \mathbb P^N$ is a standard embedding. 
	
	Pick a Hermitian structure on $\cO(-1)$ and a Hermitian connection $\nabla$.   
	Then $\imath^*\cO(-1)$ is the tautological bundle of $\mathbb P^1$, so that $\frac i{2\pi}\imath^* F_\nabla$ represents the first Chern class of $\cO_{\mathbb P^1}(-1)$. 
	Hence, by Example~\ref{Exam_CurvatureTautBundle} we have
	\begin{equation*}
		\Bigl \langle  [\tfrac i{2\pi}\imath^* F_\nabla] , \; [\mathbb P^1] \Bigr \rangle = -1
	\end{equation*}
	Hence, $c_1(\cO(-1)) = -a$ so that $c_1(L) = c_1(f^*L) = -f^*a$. 
\end{proof}
 
 \begin{rem}
 	One can prove arguing along similar lines that in the case of $\SU(2)$--bundles the two definitions of the second Chern class agree on the level of the de Rham cohomology groups. I leave the details to the readers. Moreover, one can also show that the infinite Grassmannian $\Gr_k(\C^\infty)$ is a classifying space for the group $\U(k)$.  Thus one could also define the Chern classes as pull-backs of certain classes on  $\Gr_k(\C^\infty)$.  
 \end{rem}
 
 \begin{rem}
 	One corollary of Definition~\ref{Defn_FirstChernClass_Top} is follows. 
 	Let  $L\to M$ be a Hermitian line bundle and $h\colon \Sigma\to M$ a smooth map, where $\Sigma$ is a compact oriented two--manifold. 
 	Then we have
 	\begin{equation*}
 		\frac i{2\pi}\int_M h^*F_\nabla \in \Z,
 	\end{equation*}
 	where $\nabla$ is any Hermitian line bundle. 
 	This property may be quite surprising if one's starting point is Definition~\ref{Defn_ChernClassCurvature}. 
 	
 	In particular, if $M$ is itself a compact oriented two-dimensional manifold (and $h$ is the identity map), then
 	\begin{equation*}
 		\frac i{2\pi}\int_M F_A
 	\end{equation*} 
 	is an integer, which coincides with the degree of $L$, cf. Example~\ref{Ex_ClassifBundlesOnSurfaces}. 
 \end{rem}
 
 \begin{rem}
 	A straightforward computation yields that for any matrix $\xi\in\su(2)$ we have $\tr \xi^2 =-2\det\xi$. 
 	Hence, for an $\SU(2)$--bundle $P$ we have 
 	 \begin{equation*}
 	 	 c_2(P) = \frac 1{8\pi^2} \bigl [ \tr(F_A\wedge F_A) \bigr  ]\in H^4_{dR}(M; \R).
 	 \end{equation*}
 	 In particular, if $M$ is a closed oriented four-manifold, the integration yields an isomorphism $H^4_{dR}(M;\, \R)\cong \R$. 
 	 In fact, just as in the case of line bundles above, we have
 	 \begin{equation}
	 	 \label{Eq_SecondChernClass}
 	 	c_2(P)= \frac 1{8\pi^2} \int_M \tr(F_A\wedge F_A) \in \Z. 
 	 \end{equation}
 \end{rem}
 
\subsection{The Chern--Simons functional}
\label{Sect_ChernSimons}

In this section I will restrict myself to dimension three and $G=\SU(2)$. 
Thus, let $M$ be a three manifold equipped with an  $\SU(2)$--bundle $P\to M$.
Notice that $P$ is trivial as we have seen in~\autoref{Prop_SU2BundlesAreTrivial}.

As a matter of fact, any closed oriented three-manifold is a boundary of a compact oriented four-manifold, say $\partial X=M$. 
Assume there is an extension of $P$ to $X$, i.e., a bundle $P_X$ such that $P_X|_M=P$. 
In this case any connection $A$ on $P$ can be extended to a connection $A_X$ on $P_X$ so that we can form the integral
\begin{equation*}
	\frac 1{8\pi^2}\int_X\tr \bigl (  F_{A_X}\wedge F_{A_X} \bigr ).
\end{equation*}
If we take any other extension $(X', P_X', A_X')$, we can glue $X$ and $X'$ along their common boundary to form a four-manifold without boundary. 
Strictly speaking, when performing the gluing we have to change the orientation of $X'$ so that $\partial X'$ is  equipped with the orientation opposite to that of $M$ to have the resulting manifold oriented. 
This together with~\eqref{Eq_SecondChernClass} yields, that the difference
\begin{equation*}
		\frac 1{8\pi^2}\int_X\tr \bigl (  F_{A_X}\wedge F_{A_X} \bigr ) - 
			\frac 1{8\pi^2}\int_{X'}\tr \bigl (  F_{A_X'}\wedge F_{A_X'} \bigr )
\end{equation*} 
 is an integer. 
 Hence, 
 \begin{equation}
 	\vartheta (A):= \frac 1{8\pi^2}\int_X\tr \bigl (  F_{A_X}\wedge F_{A_X} \bigr )
 \end{equation}
is well-defined as a function with values in $\R/\Z$. 
This is called \emph{the Chern--Simons functional}.

\medskip

While this definition makes transparent the relation of the Chern--Simons functional with the Chern--Weil theory, it is possible to compute the value of the Chern--Simons functional directly without extending $A$ to a four-manifold. 
In fact, choosing a trivialization of $P$ we can think of $A$ as a 1-form on $M$ with values in $\su(2)$. 
Then 
\begin{equation*}
	\vartheta (A) = \frac 1{8\pi^2}\int_M \tr \bigl ( A\wedge dA + \frac 23 A\wedge A\wedge A \bigr ). 
\end{equation*} 
Notice that this expression does not yield an $\R$--valued function. 
The reason is that by changing the trivialization of $P$ the value of $\vartheta$ changes by an integer so that we obtain again a map to the circle.

\begin{exercise}$\phantom{a}$
	\begin{enumerate}[(a)]
		\item For $A\in\Om^1(X)$, where $X$ is a four-manifold,  prove the equality
		\begin{equation}
			\label{Eq_AuxDChernSimonsIntegrand}
			d\;\tr \Bigl ( A\wedge dA +\frac 23 A\wedge A\wedge A \Bigr ) = \tr\bigl ( F_A\wedge F_A \bigr ).
		\end{equation}
		Notice the following: In the special case $X=M\times \R$ denote by $A_t$ the pull-back of $A$ to $M\times\{ t \}$. 
		Then~\eqref{Eq_AuxDChernSimonsIntegrand} clearly implies
		\begin{equation*}
			\vartheta (A_t)-\vartheta (A_{t_0}) = \int_{M\times [t_0, t]}\tr \bigl ( F_A\wedge F_A  \bigr ).
		\end{equation*}
		\item Prove that the two definition of the Chern--Simons functional agree.
		\item Prove that the values of the Chern--Simons functional with respect to two different trivializations differ by an integer. 
		\item Let $g$ be a gauge transformation, which can be though of as a map $M\to\SU(2)\cong S^3$. Show that
		\begin{equation*}
			\vartheta (A\cdot g) = \vartheta (A) + \deg g.
		\end{equation*}
	\end{enumerate} 
\end{exercise}

Let us compute the differential of $\vartheta$. 
For $a\in \Om^1(M;\, \su(2))$ we have
\begin{equation*}
	d\vartheta_A(a) = \frac 1{8\pi^2}\int_M\tr \bigl ( a\wedge dA + A\wedge da + 2a\wedge A\wedge A  \bigr )
	= \frac 1{4\pi^2}\int_M\tr\bigl ( F_A\wedge a\Bigr).
\end{equation*}
In the second equality the integration by parts is used.  
Hence, we conclude the following.
\begin{proposition}
	The critical points of the Chern--Simons functional are flat connections, i.e., connections $A$ such that $F_A=0$. \qed
\end{proposition}

\subsection{The modui space of flat connections}

Even though in Section~\ref{Sect_ChernSimons} I opted to work with three-manifolds and $G=\SU(2)$, the notion of a flat connection clearly makes sense for any background manifold and any structure group. 
Thus, we do not need to impose these restrictions in this section. 

\medskip

Let $P\to M$ be a principal $\rG$--bundle. 
Denote by $\cA^\flat (P)$ the space of all flat connections on $P$.
The gauge group $\cG(P)$ acts on $\cA^\flat (P)$ so that we can form \emph{the moduli space of flat connections:}
\begin{equation*}
	\cM^\flat (P):=\cA^\flat (P)/ \cG(P). 
\end{equation*}  

Such moduli spaces are typical objects in gauge theory that we will meet many times below. 
The main questions we are interested in are the following: Is $\cM^\flat (P)$ compact? Is $\cM^\flat (P)$ a manifold?    

An important point to notice is that $\cM^\flat$ represents the space of all solutions of a non-linear PDE modulo an equivalence relation, so that in essence the question is to describe topological properties of the space of all solutions of a non-liner PDE. 
In general, this may be a hard question, however, in this particular case we will see below that this can be done with a little technology involved.  

However, why could one be potentially  interested in spaces like $\cM^\flat$? 
The two main reasons are as follows: First, sometimes $\cM^\flat$ encodes a subtle information about the background manifold $M$ (as well as the bundle $P$) and thus can be used for instance as a tool in studies of the topology of $M$;   Secondly, moduli spaces come often equipped with an extra structure, which may be of interest on its own. 
In these notes I will mainly emphasize the first point, while the second one will be only briefly mentioned below.

\subsubsection{Parallel transport and holonomy}
Let $E\to M$ be a vector bundle equipped with a connection $\nabla$. 
For any (smooth) curve  $\gamma\colon [0,1]\to M$,  $\gamma^*\nabla$ is a connection on $\gamma^*E$. 
A section $s\in \Gamma(\gamma^*E)$ is said to be \emph{parallel along $\gamma$}, if $(\gamma^*\nabla) (s) =0$. 
\begin{rem}
	If $\gamma$ is a simple embedded curve, then $s$ can be thought of as a section of $E$ defined along the image of $\gamma$. 
\end{rem}

Since any bundle over an interval is trivial, we can trivialize $\gamma^*E\cong \R^k\times[0,1]$ so that $\gamma^*\nabla$ can be written as $\tfrac d{dt} + B(t)\, dt$, where $B\colon [0,1]\to M_k(\R)$ is a map with values in the space of $k\times k$-matrices. 
Thinking of $s$ as a map $[0,1]\to \R^k$, we obtain that $s$ is parallel along $\gamma$ if and only if $s$ is a solution of the equations:
\begin{equation*}
	\dot s + A(t)s(t) =0.
\end{equation*}

By the main theorem of ordinary differential equations, the above equation has a unique solution for any initial value $s_0$ and this solution is defined on the whole interval $[0,1]$.
\begin{defn}
	If $s\in \Gamma(\gamma^*E)$ is parallel along $\gamma$, then $s(1)\in E_{\gamma(1)}$ is called the parallel transport of  $s(0)=s_0\in E_{\gamma(0)}$ with respect to $\nabla$.
\end{defn} 

The above consideration shows in fact that for any connection $\nabla$ any curve $\gamma$ we have a linear isomorphism 
\begin{equation*}
	\mathrm{PT}_\gamma\colon E_{\gamma(0)}\to E_{\gamma(1)},
\end{equation*}
which is called \emph{the parallel transport.}

In a special case, namely when $\gamma$ is a loop, the parallel transport yields an isomorphism of the fiber. 
If we concatenate two loops, the parallel transport is the composition of the parallel transports corresponding to the initial loops. 
Hence, for a fixed connection the set of all parallel transports is in fact a group.

\begin{defn}
	Pick a point $m\in M$. The group
	\begin{equation*}
		\Hol_{m}(\nabla):=\bigl \{ \, \mathrm{PT}_\gamma\in \GL(E_{m}) \mid \gamma \text{ is a loop based at } m\,\bigr \}
	\end{equation*}
	is called the holonomy group of $\nabla$ based at $m$.
\end{defn}  

Choosing a basis of $E_m$, we can think of $\Hol_m(\nabla)$ as a subgroup of $\GL_k(\R)$. 
A standard argument shows that if $m$ and $m'$ lie in the same connected component, then the holonomy groups $\Hol_m(\nabla)$ and $\Hol_{m'}(\nabla)$ are conjugate, i.e., there is $A\in \GL_k(\R)$ such that $\Hol_{m'}(\nabla) = A\Hol_m(\nabla) A^{-1}$. 
With this understood, we can drop the basepoint from the notation. 
Even though $\Hol(\nabla)$ is defined up to a conjugacy only, it is still commonly referred to as a subgroup of $\GL_k(\R)$.

\begin{exercise}
	Show that the following holds:
	\begin{itemize}
		\item $\nabla$ is Euclidean $\implies$ $\mathrm{PT}_\gamma$ is orthogonal $\implies$ $\Hol(\nabla)\subset O(k)$.
		\item $\nabla$ is complex $\implies$ $\mathrm{PT}_\gamma$ is complex linear $\implies$ $\Hol(\nabla)\subset \GL_{k/2}(\C)$.
		\item $\nabla$ is complex Hermitian $\implies$ $\mathrm{PT}_\gamma$ is unitary $\implies$ $\Hol(\nabla)\subset \U(k/2)$.
	\end{itemize}
\end{exercise}

\begin{rem}
	The concept of the parallel transport also makes sense for connections on principal bundles. 
	The construction does not differ substantially from the case of vector bundles. 
	The details are left to the readers. 
\end{rem}

\subsubsection{The monodromy representation of a flat connection}

Let $\nabla$ be a flat connection on a vector bundle $E$ of rank $k$. 
We can view the parallel transport as a map
\begin{equation*}
	\gamma\mapsto \Hol (\nabla;\gamma),
\end{equation*} 
where $\gamma$ is a loop based at some fixed point $m$. 
If $A$ is flat, this map depends on the homotopy class of $\gamma$ only~\cite{KobayashiNomizu96_FoundDiffGeomI}*{II.9} so that effectively we obtain a representation of the fundamental group: $\rho_A\colon \pi_1(M)\to \GL_k(\R)$.  

\begin{exercise}
	Show that a gauge-equivalent connection yields a conjugate representation. 
\end{exercise}

Conversely, given a representation $\rho\colon \tilde M\to \GL_k(\R)$ we can construct the bundle
\begin{equation*}
	E:=\tilde M\times_{\pi_1(M),\,\rho}\R^k.
\end{equation*}
Here $\pi_1(M)$ acts on $\tilde M$ by the deck transformations. 
This means that $\tilde M$ can be viewed as a principal $\pi_1(M)$--bundle so that $E$ is the associated bundle corresponding to the representation $\rho$. 

This bundle is equipped with a natural flat connection. 
Indeed, interpreting a section $s$ of $E$ as a $\pi_1(M)$--equivariant map $\hat s\colon \tilde M \to \R^k$, we can define $\nabla s$ via
\begin{equation*}
	\pi^*\nabla s = d\hat s,
\end{equation*}    
cf.~\eqref{Eq_ConnInTermsOfEquivMaps}. 

These constructions establish a bijective correspondence between the space of all flat connections $\cM^\flat$ and the representation variety
\begin{equation*}
	\cR(M; \GL_k(\R)):= \bigl\{ \rho\colon \pi_1(M)\to \GL_k(\R) \text{ is a group homomorphism}\,  \bigr \}/\text{Conj},
\end{equation*}
where two representations are considered to be equivalent if they are conjugate. 

\begin{remark}
	We could equally well consider flat connections on Euclidean or Hermitian vector bundles. 
	This requires only cosmetic changes and the outcome is the representation space $\cR(M;\, \O(k))$ and $\cR(M;\, \U(k))$ respectively. 
	Even more generally, the constructions above can be modified to the case of principal $\rG$--bundles so that the space of flat $\rG$--connections correspond to $\cR(M;\, \rG)$. 
	I leave the details to the readers.
\end{remark}

Since the fundamental group of a manifold is finitely presented, 
we can  choose a finite number of generators of $\pi_1(M)$, say $\gamma_1, \dots, \gamma_N$.
Then any representation $\rho$ is uniquely specified by the images of the generators $g_i =\rho(\gamma_i)\in G$, which satisfy a finite number of relations.
This shows the inclusion
\begin{equation*}
	\cR(M;\, \rG)\subset G^N/G,
\end{equation*}
where $G$ acts by the adjoint action on each factor.
This implies in particular the following. 

\begin{proposition}
	Let $M$ be a manifold. 
	If $\rG$ is a compact Lie group, then the space $\cM^\flat$ of all flat $\rG$-connections  is compact.\qed
\end{proposition}
\begin{example}
	For $M=\mathbb T^n$, we have clearly
	\begin{equation*}
		\cR(\mathbb T^n;\, \U(1))=\Hom(\mathbb T^n, \U(1)) =\U(1)^n\cong \mathbb T^n.
	\end{equation*}
\end{example}

\begin{example}
	Let $\Sigma$ be a compact Riemann suface of genus $\gamma$ without boundary. 
	It is well-known that the fundamental group of $\Sigma$ has the following representation
	\begin{equation*}
		\pi_1(\Sigma)\cong \Bigl \langle a_1,\dots, a_\gamma, b_1,\dots, b_\gamma\mid \prod_i [a_i, b_i] = 1 \Bigr\rangle.
	\end{equation*} 
Hence, 
\begin{equation*}
	\cR(\Sigma, \rG)=\Bigl\{ A_1,\dots, A_\gamma, B_1,\dots, B_\gamma\in \rG \mid \prod_i [A_i, B_i] = 1\Bigr\}/\rG.
\end{equation*}
For instance, for $\rG=\SL(n;\C)$ the representation variety has a rich geometric structure, which is being actively studied, see for example~\cite{BradlowEtAl07_WhatISHiggsBundle, Gothen14_RepresSurfGps, Rayan18_AspectsTopHiggsBundles} and references therein. 
\end{example}

\section{Dirac operators}
\label{Sec_DiracOp}

\subsection{Spin groups and Clifford algebras}
In this subsection I recall briefly the notions of Clifford algebra and spin group focusing on low dimensions.  
More details can be found for instance in~\cite{LawsonMichelsohn:89}. 

Since $\pi_1(\SO(n))\cong\Z/2\Z$ for any $n\ge 3$, there is a simply connected Lie group denoted by $\Spin(n)$ together with a homomorphism $\Spin(n)\to \SO(n)$, which is a double covering. 
This characterizes $\Spin(n)$ up to an isomorphism. 
The spin groups can be constructed explicitly with the help of Clifford algebras, however in low dimensions this can be done more directly with the help of quaternions. 

Since $\Sp(1)\cong \SU(2)$ is diffeomorphic to the 3--sphere, this is a connected and simply connected Lie group.
Identify $\ImH=\{ \bar h=- h \}$ with $\R^3$ and consider the homomorphism
\begin{equation}
	\label{Eq_Sp1ToSO3}
	\a\colon \Sp(1)\to \SO(3),\qquad q\mapsto A_q,
\end{equation}
where $A_q h=qh\bar q$. 
It is easy to check that the corresponding Lie-algebra homomorphism is in fact an isomorphism. 
Since $\SO(3)$ is connected, $\a$ is surjective. 
Moreover,  $\ker\a = \{\pm 1\}$. Hence, \eqref{Eq_Sp1ToSO3} is a non-trivial double covering, i.e., $\Spin(3)\cong \Sp(1)$.   

\smallskip

To construct the group $\Spin(4)$,  recall first that the Hodge operator $*$ yields the splitting $\Lambda^2(\R^4)^*=\Lambda_+^2(\R^4)^*\oplus \Lambda^2_-(\R^4)^*$, where $\Lambda^2_\pm(\R^4)^*=\{\om\mid *\om=\pm\om\}$. Since $\mathfrak{so}(4)\cong \Lambda^2(\R^4)^*=\Lambda_+^2(\R^4)^*\oplus \Lambda^2_-(\R^4)^*=\mathfrak{so}(3)\oplus \mathfrak{so}(3)$, the adjoint representation yields a homomorphism $\SO(4)\to \SO(3)\times \SO(3)$.

Identify $\R^4$ with $\H$ and consider the homomorphism\footnote{We adopt the common convention $Sp_\pm(1)=Sp(1)$. The significance of the subscripts ``$\pm$'' will be clear below.}
\[
\b\colon \Sp_+(1)\times \Sp_-(1)\to \SO(4),\qquad (q_+, q_-)\mapsto A_{q_+\!,\, q_-},
\]
where $A_{q_+\!,\, q_-}h= q_+h \bar q_-$. 
An explicit computation shows that the composition $\Sp_+(1)\times \Sp_-(1)\to \SO(4)\to \SO(3)\times \SO(3)$ is given by $(q_+, q_-)\mapsto (A_{q_+}, A_{q_-})$. 
Hence, the Lie algebra homomorphism corresponding to $\b$ is an isomorphism and $\ker\b$ is contained in $\{(\pm 1, \pm 1)\}$. 
As it is readily checked, $\ker\b=\{ \pm(1,1)\}\cong \Z/2\Z$. 
Hence, $\Sp_+(1)\times \Sp_-(1)\cong \Spin(4)$.

\medskip

Let $U$ be an Euclidean vector space. 
Then the Clifford algebra $Cl(U)$ is the tensor algebra $TU=\R\oplus U\oplus U\otimes U\oplus\dots$ modulo the ideal generated by elements $u\otimes u +|u|^2\cdot 1$. 
In other words, $Cl(U)$ is generated by elements of $U$ subject to the relations $u\cdot u= -|u|^2$. 
For instance, $Cl(\R^1)\cong \R[x]/(x^2+1)\cong\C$. 
The algebra $Cl(\R^2)$ is generated by $1, e_1, e_2$ subject to the relations $e_1^2=-1=e_2^2$ and $e_1\cdot e_2=-e_2\cdot e_1$, which follows from $(e_1+e_2)^2=-2$. 
In other words,  $Cl(\R^2)\cong\H$. 
In general, $Cl(\R^n)$ is generated by $1,e_1,\dots, e_n$ subject to the relations $e_i^2=-1$ and $e_i\cdot e_j=-e_j\cdot e_i$ for $i\neq j$. 

It can be shown that the subgroup of $Cl(\R^n)$ generated by elements of the form $v_1\cdot v_2\cdot\ldots\cdot v_{2k}$ is isomorphic to $\Spin(n)$, where each $v_j\in \R^n$ has the unit norm. 
In particular, this shows that $\Spin(n)$ is a subgroup of $\Cl(\R^n)$.

It is convenient to have some examples of modules over Clifford algebras. 
Such module is given by a vector space $V$ together with a map
\[
U\otimes V\to V, \qquad u\otimes v\mapsto u\cdot v,
\] 
which satisfies $u\cdot (u\cdot v)= -|u|^2 v$ for all $u\in U$ and $v\in V$. 
An example of a $Cl(U)$--module is $V=\Lambda U^*$, where the $Cl(U)$--module structure is given by the map
\begin{equation}
	\label{Eq_CliffMultiplForms}
u\otimes \varphi\mapsto \imath_u\varphi - \langle u,\cdot\rangle\wedge\varphi.
\end{equation}

Let $V$ be a quaternionic vector space. 
Then the quaternionic multiplication gives rise to  the map $\ImH\otimes V\to V$, $h\otimes v\mapsto h\cdot v$, which satisfies $h\cdot (h\cdot v)=-h\bar h v=-|h|^2v$. 
Thus, any quaternionic vector space is a $Cl(\R^3)$--module. 
In particular, the fundamental representation $\slS\cong\H$ of $\Sp(1)\cong \Spin(3)$ with the action given by the left multiplication is a $Cl(\R^3)$--module.

The multiplication on the right by $\bar i$ endows $\slS$ with the structure of a complex $\Sp(1)$--representation, which is in fact also Hermitian (this is just another manifestation of the isomorphism $\Sp(1)\cong\SU(2)$).
It is then  an elementary exercise in the representation theory to show the isomorphisms
\begin{equation}
	\label{Eq_BasicIsoOfSUreps}
	\ImH\otimes\C\cong  \End_0(\slS),
\end{equation}
where the left hand side is viewed as an $\Sp(1)$--representation via the homomorphism $\a$ and $\End_0(\slS)$ denotes the subspace of traceless endomorphisms. 
Moreover, the real subspace $\ImH$ can be identified with the subspace of traceless Hermitian endomorphisms. 

Furthermore, for any quaternionic vector space $V$ the space $V\oplus V$ is a $Cl(\R^4)$--module. 
Indeed, the $Cl(\R^4)$--module structure is induced by the map 
\begin{equation}
\label{Eq_ClInDim4}
\H\otimes_\R (V\oplus V)\to V\oplus V,\qquad
h\otimes (v_1, v_2)\mapsto (h v_2,-\bar hv_1)=
\begin{pmatrix}
0 &  h\\
-\bar h & 0
\end{pmatrix}
\begin{pmatrix}
v_1\\ v_2
\end{pmatrix}.
\end{equation}
In particular, the $\Sp_+(1)\times \Sp_-(1)$--representation $\slS^+\oplus \slS^-$ is a $Cl(\R^4)$--module. Here, as the notation suggests, $\slS^\pm$ is the fundamental representation of $\Sp_\pm(1)$.

Just like in the case of dimension three, we have an isomorphism of $\Spin(4)$--representations
\begin{equation*}
	\H\otimes \C\cong\Hom(\slS^+;\, \slS^-),
\end{equation*}
where the left hand side is viewed as a $\Spin(4)$--representation via the homomorphism $\b$.

\subsection{Dirac operators}
\label{Sect_DiracOp}

Let $M$ be a Riemannian oriented manifold of dimension $n$.
The tautological action of $\SO(n)$ on $\R^n$ extends to an action on $Cl(\R^n)$ so that we can construct the associated bundle
\begin{equation*}
	Cl(M):=\Fr_{\SO}\times_{\SO(n)} Cl(\R^n).
\end{equation*}
This can be thought of as the bundle, whose fiber at a point $m\in M$ is $Cl(T_mM)\cong Cl(T_m^*M)$. 
Notice that the Levi--Civita connection yields a connection on $Cl(M)$.
This is denoted by the same symbol $\nabla^{LC}$. 

Let $E\to M$ be a bundle of $Cl(M)$--modules, i.e., there is a morphism of vector bundles 
\[
Cl\colon TM\otimes E\to E,\qquad (v,e)\mapsto v\cdot e,
\]
such that $v\cdot (v\cdot e)=-|v|^2 e$. 
Then $E$ is called a Dirac bundle if it is equipped with an Euclidean scalar product and a connection $\nabla$ such that the following conditions hold:
\begin{itemize}
	\item $\nabla$ is Euclidean;
	\item $\langle v\cdot e_1, v\cdot e_2 \rangle= |v|^2\langle e_1, e_2 \rangle$ for any $v\in T_mM$ and $e_1, e_2\in E_m$;
	\item $\nabla (\varphi\cdot s)= (\nabla^{LC} \varphi)\cdot s + \varphi\cdot\nabla s$ for any $\varphi\in\Gamma(Cl(M))$ and $s\in\Gamma(E)$. 
\end{itemize}
\begin{defn}
	If $E$ is a Dirac bundle, the operator
	\[
	\dirac\colon \Gamma(E)\xrightarrow{\ \nabla\ }\Gamma(T^*M\otimes E)\xrightarrow{\ Cl\ }\Gamma(E)
	\] 
	is called the Dirac operator of $E$. 
\end{defn}

In other words, if $e_1,\dots, e_n$ is a local orthonormal oriented frame of $TM$, then 
\begin{equation*}
	\dirac s =\sum_{i=1}^ne_i\cdot \nabla_{e_i} s 
\end{equation*}

\begin{example}
	\label{Ex_DiracOnForms}
	The bundle $\Lambda T^*M=\oplus_{k=0}^n\Lambda^kT^*M$ has a natural structure of a Dirac bundle, where the Clifford multiplication is given by~\eqref{Eq_CliffMultiplForms}. 
	The corresponding Dirac operator equals $d+d^*$~\cite[Thm~5.12]{LawsonMichelsohn:89}, where $d^*$ is the formal adjoint of $d$, see Section~\ref{Sect_deRhamComplex} below for more details. 
\end{example}

\begin{exercise}
	Show that the Dirac operator on a closed manifold is formally self-adjoint, i.e., for any $s_1, s_2\in\Gamma(E)$ we have
	\begin{equation*}
		\int_M \langle \dirac s_1, \, s_2\rangle = \int_M \langle s_1, \,\dirac s_2\rangle.
	\end{equation*}
	
\end{exercise}

\subsection{Spin and Spin$^c$  structures}
\label{Sect_SpinSpinCstr}

Let $\Fr_{\SO}\to M$ be the principal bundle of orthonormal oriented frames of $M$.   
\begin{defn}
	\label{Defn_SpinCStr}
	$M$ is said to be spinnable, if there is a principal $\Spin(n)$ bundle $P$ equipped with a $\Spin(n)$--equivariant map $\tau\colon P\to\Fr_\SO$, which covers the identity map on $M$ and is a fiberwise double covering.
	Here $\Spin(n)$ acts on $\Fr_\SO$ via the homomorphism $\Spin(n)\to\SO(n)$.
	
	A choice of a bundle $P$ as above is called a spin structure. 
	A manifold equipped with a spin structure is called a spin manifold. 
\end{defn}

For a given $M$ a spin structure may or may not exist.
If $M$ is spinnable, there may be many non-equivalent spin structures. 
The questions on existence and classification of spin structures may be completely answered in terms of the Stiefel--Whitney classes~\cite{LawsonMichelsohn:89}.
However, I will not go into the details here.
For the remaining part of this section I assume throughout that $M$ is spin.

Let $\om\in\Om^1\bigl (\Fr_{\SO};\, \so(n) \bigr)$ be the connection 1-form of the Levi--Civita connection. 
Since the homomorphism $\Spin(n)\to \SO(n)$ is a local diffeomorphism, we have an isomorphism of Lie algebras $\spin(n)\cong \so(n)$. 
Hence, $\tau^*\om\in \Om^1\bigl(P;\, \spin(n)\bigr )$ is a connection on $P$.
Slightly abusing terminology, this is still called the Levi--Civita connection. 

\medskip

For any $n\ge 3$ there is a unique complex representation $\rho\colon \Spin(n)\to \End(\slS)$ distinguished by the property that it extends to a complex irreducible representation of $\Cl(\R^n)$.
Notice that this means neither that $\slS$ is a unique $\Spin(n)$--representation, nor that $\slS$ is an irreducible $\Spin(n)$--representation. 
For example, for $n=3$ this representation coincides with the fundamental representation of $\Sp(1)=\Spin(3)$. 
For $n=4$ we have $\slS = \slS^+\oplus \slS^-$, see~\eqref{Eq_ClInDim4}. 

If $M$ is spin, we can construct \emph{the spinor bundle} 
\begin{equation*}
	\slS:=P\times_{\Spin(n),\,\rho}\slS.
\end{equation*}
Here, as it is quite common, we use the same notation both for the representation and the associated vector bundle.

Since $\rho$ extends to a representation of $Cl(\R^n)$, the spinor bundle is in fact a bundle of $Cl(M)$--modules. 
Hence, the construction of \autoref{Sect_DiracOp} yields the spin Dirac operator
\begin{equation*}
	\slD\colon \Gamma(\slS)\to \Gamma(\slS).
\end{equation*} 

\begin{remark}
	Recall that in the case $\dim M =4$, the spinor bundle splits: $\slS = \slS^+\oplus\slS^-$. 
	By~\eqref{Eq_ClInDim4}, the Clifford multiplication with 1--forms changes the chirality, i.e., for any $\om\in\Om^1(M)$ we have $\om\cdot \colon \slS^\pm\to \slS^{\mp}$. 
	Hence, 
	\begin{equation*}
	\slD=
	\begin{pmatrix}
	0 & \slD^-\\
	\slD^+ & 0
	\end{pmatrix},
	\qquad\text{where } \slD^\pm\colon\Gamma(\slS^\pm)\to\Gamma(\slS^\mp).
	\end{equation*}
\end{remark}

The following two variations of this construction are frequently used. 
First, let $P$ be a principal $\rG$--bundle equipped with a connection $A$ and let $\tau\colon \rG\to \U(n)$ be a unitary representation of $\rG$ so that we have the associated bundle $E= P\times_{\rG, \tau} \C^n$, which is Hermitian. 
Then the twisted spinor bundle $\slS\otimes E$ is also a Dirac bundle so that we have a twisted Dirac operator
\begin{equation*}
	\slD_A\colon \Gamma(\slS\otimes E)\to \Gamma(\slS\otimes E).
\end{equation*}

\begin{example}
	Let us assume that $\dim M =3$ for the sake of definiteness. 
	Choose $E=\slS$, which is equipped with the Levi--Civita connection.
	Then we have
	\begin{equation*}
		\slS\otimes\slS = \Sym^2(\slS)\oplus \Lambda^2\slS \cong T_\C^*M \oplus \C,
	\end{equation*}
	cf.~\eqref{Eq_BasicIsoOfSUreps}. 
	Hence, the twisted spinors can be identified with the complexification of odd forms. 
	Of course, the Hodge $*$-operator yields and isomorphism between odd and even forms so that we can identify the twisted spinors with even forms too.
	We already have seen above a Dirac operator acting on forms, namely
	\begin{equation*}
		d+ d^*\colon \Om^{\mathrm{odd}}(M)\to \Om^{\mathrm{even}}(M),
	\end{equation*}
	cf. \autoref{Ex_DiracOnForms}. 
	One can show that the complexification of $d+d^*$ coincides with  the twisted Dirac operator on $\Gamma(\slS\otimes\slS)$. 
\end{example}
\begin{exercise}
	The Clifford multiplication combined with the map $\ad P\to \End(E)$ yields the `twisted' Clifford multiplication 
	\begin{equation*}
		T^*M\otimes \ad P\to \End (\slS)\otimes \End(E)\cong \End (\slS\otimes E).
	\end{equation*} 
	Show that for $a\in \Om^1(\ad P)$ the following holds:
	\begin{equation*}
		\slD_{A +a}\psi = \slD_A\psi + a\cdot \psi.
	\end{equation*}
\end{exercise}

\medskip

Let me explain the second variation, which in essence is not really much different from the first one, however the details are somewhat involved.
 
Thus, by the construction of $\Spin(n)$ we have the exact sequence
\begin{equation*}
	\{  1 \}\to \{ \pm 1 \}\to \Spin(n)\to \SO(n)\to \{ 1 \}.
\end{equation*}
Moreover, one can show that the kernel $\{ \pm 1 \}$ lies in the center of $\Spin(n)$. 
This is clear anyway in our main cases of interest, namely for $n=3$ and $n=4$. 
Denote 
\begin{equation*}
	\Spin^c(n):=\Spin(n)\times \U(1)/\pm 1,
\end{equation*} 
where $\{ \pm 1\}$ is embedded diagonally. 
Notice that both $\Spin(n)$ and $\U(1)$ are subgroups of $\Spin^c(n)$ and $\U(1)$ lies in fact in the center of $\Spin^c(n)$. 

Furthermore, we have the following exact sequences:
\begin{equation*}
	\begin{aligned}
	  \{  1 \}\to \U(1)\to \Spin^c(n)\xrightarrow{\ \rho_0\ } \Spin(n)/\pm 1 = \SO(n)\to \{ 1 \},\\
	 \{  1 \}\to \Spin(n)\to\Spin^c(n)\xrightarrow{\ \rho_{\mathrm{det}}\ } \U(1)/\pm 1 \cong \U(1)\to \{ 1 \}.
	\end{aligned}
\end{equation*}
These in turn give rise to the exact sequence
\begin{equation}
	\label{Eq_SpincDoubleCov}
	\{  1 \}\to \{ \pm 1 \}\to \Spin^c(n)\xrightarrow{\ (\rho_0,\,\rho_{\mathrm{det}} ) \ } \SO(n)\times \U(1)\to \{ 1 \},
\end{equation}
which shows that $\Spin^c(n)$ is a double covering of $\SO(n)\times\U(1)$.

\begin{example}$\phantom{A}$
	\begin{enumerate}[(a)] 
		\item 
		For $n=3$ we have $\Spin^c(3) = \SU(2)\times \U(1)/\pm 1\cong \U(2)$. In particular, $\rho_{\mathrm{det}}(A) = \det A$ and the sequence~\eqref{Eq_SpincDoubleCov} has the following form
		\begin{equation*}
			\{  1 \}\to \bigl\{ \pm \mathbbm 1\bigr \}\to \U(2)\to \SO(3)\times \U(1)\to\{ 1 \},
		\end{equation*}
		where the homomorphism $\U(2)\to \SO(3) = \mathrm{PU}(2)$ is the natural projection. 
		 
		\item  For $n=4$ we have
		\begin{equation*}
		\begin{aligned}
		\Spin^c(4) &= \Bigl (\bigl ( \SU(2)\times\SU(2) \bigr )\times\U(1)\Bigr )/\pm 1 \\
		&=\bigl \{   (A_+,\, A_-)\in \U(2)\times\U(2)\mid \det A_+ = \det A_- \bigr \}. \phantom{\sum^0}
		\end{aligned}
		\end{equation*}
		In particular, $\rho_{\mathrm{det}} (A_+,\, A_-)=\det A_+=\det A_-$.
		
		Notice that in this case we also have the homomorphisms
		\begin{equation*}
			\rho_\pm\colon \Spin^c(4)\to \U(2),\qquad \rho_\pm (A_+,\, A_-)=A_\pm.
		\end{equation*}  
	\end{enumerate}
\end{example}

\begin{defn}
	A spin$^c$ structure on $M$ is a principal $\Spin^c(n)$--bundle $P\to M$ equipped with a  $\Spin^c(n)$--equivariant  map $P\to \Fr_\SO$ which induces an isomorphism $P/\U(1)\cong \Fr_\SO$.
\end{defn}

Just like in the case of spin structures, spin$^c$ structures may or may not exist. 
If a spin$^c$ structure exists, it is rarely unique. It is also clear that if $M$ is spin, then $M$ is also spin$^c$. 

It turns out that four--manifolds are somewhat special as the following result shows.  
\begin{proposition}
	Any closed oriented four--manifold admits a spin$^c$ structure.\qed
\end{proposition}

A proof of this claim can be found for example in~\cite{Morgan96_SWequations}*{Lem.\,3.1.2}.
Notice however, that there are closed four--manifolds that are not spin.
In dimension three the situation is different: Any closed oriented three--manifold is spin~\cite{GompfStipsicz99_4MfldsKirby}*{Rem.\,1.4.27}, and hence also spin$^c$.  

\medskip

With these preliminaries at hand, let $M$ be a manifold equipped with a spin$^c$ structure $P$. 
Define \emph{the determinant line bundle}
\begin{equation*}
	L_{\mathrm{det}}:=P\times_{\rho_{\mathrm{det}}} \C.
\end{equation*}
This is clearly a Hermitian line bundle and its $\U(1)$--structure is
\begin{equation*}
	P_{\mathrm{det}}:=P/\Spin(n).
\end{equation*}
By~\eqref{Eq_SpincDoubleCov}, we have a double cover map 
\begin{equation*}
	\tau\colon P\to P/\{ \pm 1 \} = \Fr_\SO\times_M P_{\mathrm{det}}.	
\end{equation*}
Hence, if $A$ is a connection on $P_{\mathrm{det}}$, then $\tau^*(\om + A)\in \Om^1\bigl (P;\, \spin^c(n)\bigr )$, where we have used the isomorphism $\spin^c(n)\cong \so(n)\oplus \fu(1)$. 
Thus, the choice of  a unitary connection on the determinant line bundle together with the Levi--Civita connection determines a connection on  the spin$^c$ bundle.  

\medskip

Let $\slS$ be the distinguished representation of $\Spin(n)$. 
This can be clearly extended to a $\Spin^c(n)$--representation, which is still denoted by $\slS$: $[g, z]\cdot s = z\rho(g)\, s$.     
This in turn yields the spin$^c$ spinor bundle
\begin{equation*}
	\slS:=P\times_{\Spin^c(n)}\slS,
\end{equation*}
which is a Dirac bundle.
Hence, we obtain the spin$^c$ Dirac operator $\slD_A\colon \Gamma(\slS)\to\Gamma(\slS)$. 

\begin{example}
	Assume $M$ is spin and pick a spin structure $P_{\Spin}$. 
	Pick also a principal $\U(1)$--bundle $P_0$.
	Then 
	\begin{equation*}
		P:=P_\Spin\times_{M} P_0/\pm 1
	\end{equation*}
	is a principal $\bigl ( \Spin(n)\times \U(1)\bigr )/\pm 1 = \Spin^c(n)$ bundle. 
	In this case the spin$^c$ spinor bundle is just the twisted spinor bundle $\slS\otimes L_0$, where $\slS$ is the pure spinor bundle and $L_0:=P_0\times_{\U(1)}\C$ is the associated complex line bundle. 
	Likewise, the spin$^c$ Dirac operator is just the twisted Dirac operator. 
\end{example}

\begin{remark}
	\label{Rem_SpincBundleAsTwist}
	In some sense, any $spin^c$ spinor bundle can be thought of as the twisted spinor bundle just like in the example above. 
	The problem is that in general neither $\slS$ nor $L_0$ is globally well-defined, however their product $\slS\otimes L_0$ is well-defined. 
	The notion of a spin$^c$ structure is just a convenient way to make this `definition' precise.   	
\end{remark}

\begin{exercise}
	For  $a\in\Om^1(M;\R i)$ prove that 
	\begin{equation}
		\label{Eq_SpincDiracChangeOfConn}
		\slD_{A+a}\psi = \slD_A\psi + \frac 12\, a\cdot \psi.
	\end{equation}
\end{exercise}

The coefficient $\frac 12$ in \eqref{Eq_SpincDiracChangeOfConn}  can be explained as follows: Think of the spin$^c$ spinor bundle as $\slS\otimes L_0$ just like in \autoref{Rem_SpincBundleAsTwist}.
Then the determinant line bundle is $\Lambda^2(\slS\otimes L_0)=L_0^2$ so that $A_0\in \cA(L_0)$ induces some connection $A$ on $L_0^2$. 
Then $A_0 + a_0$ corresponds to $A + 2\,a_0$. 

\subsubsection{On the classification of spin$^c$ structures}

Let $P$ be a spin$^c$ structure with the spinor bundle $\slS$. 
If $L$ is any Hermitian line bundle, then $\slS\otimes L$ is a spin$^c$ spinor bundle corresponding to a spin$^c$ structure $P_L$.  
This defines an action of $H^2(M;\, \Z)$ on the set $\cS = \cS(M)$ of all spin$^c$ structures on $M$. 
This action can be shown to be free and transitive, hence $\cS(M)$ can be identified with $H^2(M;\, \Z)$, however, such an identification is not canonical. 

For example, let $M$ be spin and let $\slS$ be the pure spinor bundle.
The spin structure of $M$ can be also viewed as a distinguished spin$^c$ structure so that $\cS(M)$ has a preferred point, the origin. 
This choice fixes an isomorphism $\cS(M)\cong H^2(M;\,\Z)$.

\subsection{The Weitzenb\"ock formula}

For a function $u\colon \R^4\to \H$ the operators
\begin{equation*}
  \begin{aligned}
	  \slD^+(u) &= \phantom{-}\frac {\del u}{\del x_0} + i\frac {\del u}{\del x_1}  + j\frac {\del u}{\del x_2} + k\frac {\del u}{\del x_3},\\
	  \slD^-(u) &= -\frac {\del u}{\del x_0} + i\frac {\del u}{\del x_1}  + j\frac {\del u}{\del x_2} + k\frac {\del u}{\del x_3} 
  \end{aligned}
\end{equation*}
can be thought of as  four-dimensional analogues of the $\bar\del= \frac 12 (\frac {\del }{\del x} + i\frac {\del }{\del y})$ and $\del = \frac 12 (\frac {\del }{\del x} - i\frac {\del }{\del y})$ operators for complex--valued functions of one complex variable $z=x+yi$ respectively. 
In fact, tracing through the construction, it is easy to see that for the flat four--manifold $\R^4$ the Dirac operator can be written as follows 
\begin{equation*}
	\slD = 
	\begin{pmatrix}
		0 & \slD^-\\
		\slD^+ & 0
	\end{pmatrix},
\end{equation*} 
cf. \eqref{Eq_ClInDim4}.
Here I use the fact, that the spinor bundle of $\R^4$ is (canonically) the product bundle.
A straightforward computation yields
\begin{equation}
	\label{Eq_AuxDirac2Laplacian}
	\slD^2 = 
	\begin{pmatrix}
	\slD^-\slD^+ & 0 \\
	 0 &\slD^+\slD^-
	\end{pmatrix}
	= \Delta,
\end{equation} 
where $\Delta = -\sum_{i=0}^3 \frac{\del^2}{\del x_i^2}$ is the Laplacian. 
This is usually phrased as ``the Dirac operator is the square root of the Laplacian''.

\begin{remark}
	The Dirac operator on $\R^2$ has the form 
	\begin{equation*}
		\begin{pmatrix}
		0\phantom{\sum_{0}} & 2\frac {\del}{\del z}\\
		-2\frac\del {\del \bar z}\quad & 0
		\end{pmatrix}\colon C^\infty (\R^2;\, \C^2)\to C^\infty (\R^2;\, \C^2)
	\end{equation*}
	and squares to the Laplacian just like in dimension four. 
	Notice, however, that the case of dimension two requires a special treatment due to the fact that $\pi_1(\SO(2))\cong \Z$, which was one of the reasons I assumed $n\ge 3$ in this section.
	
	On $\R^3$ the spinor bundle can be identified with the product bundle $\underline{\H}$ so that the corresponding Dirac operator is 
	\begin{equation*}
		\slD(u) =  i\frac {\del }{\del x_1}  + j\frac {\del }{\del x_2} + k\frac {\del }{\del x_3},
	\end{equation*}
	which also squares to the Laplacian.
\end{remark}

On general Riemannian manifolds the relation $\slD^2=\Delta$ still holds up to a zero order operator. 
This is known as the Weitzenb\"ock formula, which I describe next.

Thus, given an Euclidean vector bundle $E$ with a connection $\nabla$ \emph{the connection Laplacian} is defined by  
\begin{equation}
	\nabla^*\nabla\colon \Gamma(E)\xrightarrow{\ \nabla\ }\Om^1(E)\xrightarrow{\;\nabla^* = - *d_{\nabla}*\;}\Gamma(E). 
\end{equation}
\begin{exercise}$\phantom{a}$
	\begin{enumerate}[(a)]
		\item Let $(e_1,\dots, e_n)$ be a local orthonormal frame of $TM$ over an open subset of $M$.
		Show that the connection Laplacian can be expressed as follows
		\begin{equation*}
		\nabla^*\nabla s = -\sum_{i=1}^n \bigl ( \nabla_{e_i}\nabla_{e_i} s - \nabla_{\nabla_{e_i}e_i}\, s\bigr ),
		\end{equation*}
		where $\nabla e_i$ means the Levi--Civita connection applied to $e_i$.
		\item Prove that the connection Laplacian is formally self-adjoint.
		\item In the case $M$ is closed, prove the identity $\langle \nabla^*\nabla s,\, s\rangle_{L^2} = \| \nabla s\|_{L^2}^2$. 
	\end{enumerate}
\end{exercise}

Assume that $E$ is a Dirac bundle and let $R\in \Om^2(\End(E))$ be the curvature 2-form of the corresponding connection $\nabla$. 
Using the Clifford multiplication, we obtain a homomorphism $\Lambda^2T^*M\otimes\End(E)\to \End(E)$, which maps $R$ to an endomorphism $\fR$.
In a local frame $(e_i)$ of $TM$ this can be expressed as follows
\begin{equation*}
	\fR (s):= \frac 12 \sum_{i,j=1}^n e_i\cdot e_j\cdot  R_{e_i, e_j}(s).
\end{equation*}

\begin{thm}[Weitzenb\"ock formula]
	Let $\slD$ be the Dirac operator for the Dirac bundle $E$. 
	Then the following holds:
	\begin{equation*}
		\slD^2 = \nabla^*\nabla + \fR.
	\end{equation*}
\end{thm}
\begin{proof}
	Pick a point $m\in M$ and a local frame $(e_i)$ such that $\nabla_{e_i}e_j$ vanishes at $m$. 
	Then at $m$ we have the following:
	\begin{equation*}
		\begin{aligned}
			\slD^2 s &= \sum_{i=1}^n e_i\cdot \nabla_{e_i}\Bigl (\sum_{j=1}^n e_j\cdot \nabla_{e_j}s \Bigr ) = 
			\sum_{i,j=1}^n e_i\cdot e_j\cdot \nabla_{e_i}\nabla_{e_j}s\\
			&= \sum_{i=1}^n e_i\cdot e_i\cdot\nabla_{e_i}\nabla_{e_i}s + \frac 12 \sum_{i\neq j}e_i\cdot e_j\cdot \bigl ( \nabla_{e_i}\nabla_{e_j} -  \nabla_{e_j}\nabla_{e_i}\bigr)s\\
			& = \nabla^*\nabla s +\fR s.
		\end{aligned}
	\end{equation*}
\end{proof}

In the case of the spin$^c$ spinor bundle, a straightforward computation yields the following corollary, whose proof is left as an exercise.

\begin{corollary}
	\label{Cor_WeitzenboeckForm}
	Let $A$ be a Hermitian connection on the determinant line bundle with the curvature form $F_A$. 
	Then the spin$^c$ Dirac operator satisfies:
	\begin{equation*}
		\slD^2 \psi = \nabla^*\nabla \psi + \frac 14 s_g\, \psi + \frac 12 F_A\cdot\psi,
	\end{equation*}
	where $s_g$ is the scalar curvature of the background metric $g$.\qed
\end{corollary}

In particular, in the case $M$ is spin and $\slS$ is the pure spinor bundle, we have
\begin{equation*}
\slD^2 \psi = \nabla^*\nabla \psi + \frac 14 s_g\, \psi.
\end{equation*}
This implies, for example, that for a metric with positive scalar curvature the space of harmonic spinors $\ker\slD$ is trivial. 

\medskip

Recall that in dimension four, the spinor bundle splits: $\slS = \slS^+\oplus \slS^-$. 
With respect to this splitting the Dirac operator takes the form 
	\begin{equation*}
		\slD = 
		\begin{pmatrix}
			0 & \slD^-\\
			\slD^+ & 0
		\end{pmatrix},
	\end{equation*} 
which we already have seen above in the case of the flat space. 
By \autoref{Cor_WeitzenboeckForm} for $\psi\in\Gamma(\slS^+)$ we obtain
\begin{equation*}
	\slD^-\slD^+ \psi= (\slD^+)^*\slD^+\psi = \nabla^*\nabla\psi +\frac 14 s_g\, \psi +\frac 12 F_A^+\cdot \psi.
\end{equation*}
Here we used the fact that anti-self-dual 2-forms act trivially on $\slS^+$. 

\section{Linear elliptic operators}
\label{Sec_LinEllOp}

\subsection{Sobolev spaces}

Consider the following classical boundary value problem in the theory of PDEs: Let $\Om\subset \R^n$ be a bounded domain with a smooth boundary $\partial \Om$. 
Does there exist a function $u\in C^2(\Om)\cap C^0(\bar \Om)$ such that 
\begin{equation}
	\label{Eq_DirichletBVLaplacian}
	\Delta u =0 \quad\text{in}\ \Om\qquad\text{and}\qquad u|_{\partial\Om} = \varphi,
\end{equation}
where $\varphi$ is a given function on $\partial \Om$?

Consider \emph{the energy functional}
\begin{equation*}
	E(u):=\frac 12 \int_\Om |\nabla u (x)|^2\, dx,
\end{equation*}
where $dx$ denotes the standard volume form on $\R^n$. 
Assume there exists an absolute minimum of $E$, i.e., a function  $u\in C^2(\Om)\cap C^0(\bar \Om)$ such that
\begin{equation*}
	E(u) = \inf \bigl \{  E(v)\mid  v\in C^2(\Om)\cap C^0(\bar \Om), \ v|_{\partial\Om} = \varphi \bigr \}.
\end{equation*}
A straightforward computation using integration by parts yields
\begin{equation*}
	0=\bigl .\frac  d{dt}\bigr |_{t=0} E(u+t w) = \int_\Om w\, \Delta u\, dx
\end{equation*}
for all  $w$  such that  $w|_{\del \Om} =0$. 
This implies that $u$ is harmonic in $\Om$.
Thus, we can find a solution of~\eqref{Eq_DirichletBVLaplacian} if we can prove that the energy functional attains its minimum. 

Hence, a strategy for proving the existence of solutions of~\eqref{Eq_DirichletBVLaplacian} may be the following: Pick a sequence $u_k$ such that $E(u_k)$ converges to the infimum of the energy functional and prove that $u_k$ converges to a limit $u$ possibly after extracting a  subsequence.  
This strategy does work indeed, however, it requires certain technology known as the theory of Sobolev spaces, which turned out to be very useful in gauge theory as well. 
What follows below is a crash course in Sobolev spaces. 
An interested reader should consult more specialized literature for details.  

\medskip

First recall that the space $L^2(\R^n)$ of square integrable functions equipped with a scalar product 
\begin{equation*}
	\langle u, v\rangle_{L^2}:=\int_{\R^n}uv\, dx
\end{equation*}
is a Hilbert space, i.e., $L^2(\R^n)$ is complete with respect to the norm $\|u \|_{L^2}:=\sqrt{\langle u, u\rangle}$. 
This space can be viewed as a completion of the space $C^\infty_0(\R^n)$ of all smooth functions with compact support with respect to the norm $\| \cdot \|_{L^2}$. 

This definition admits a number of variations, which will be of use below.   
First, we can pick any $p>1$ and put
\begin{equation*}
	\|  u \|_{L^p}:= \Bigl (  |u(x)|^p\, dx \Bigr )^{\frac 1p},
\end{equation*}
where $u\in C^\infty_0(\R^n)$. 
The completion of $C^\infty_0(\R^n)$ with respect to $\| \cdot \|_{L^p}$ is then the Banach space $L^p(\R^n)$.

Secondly, we can also consider
\begin{equation*}
	\|  u \|_{W^{k,p}}:= \Bigl ( \sum_{i=0}^k \| \nabla^i u \|_{L^p} \Bigr )^{\frac 1p}.
\end{equation*} 
The completion of $C^\infty_0(\R^n)$ with respect to this norm is denoted by $W^{k,p}(\R^n)$. 
These spaces are called \emph{Sobolev spaces}. 

\begin{remark}
	Many different and inconsistent notations are in use for Sobolev spaces. For example, sometimes $L_r^s$ may mean $W^{r,s}$ and sometimes $W^{s,r}$. 
	Of all notations being used for Sobolev spaces, the symbol $W^{k,p}$ seems to be most consistently used, hence I opted for this one.    
\end{remark}

Third, we can replace the domain $\R^n$ by any open subset of $\R^n$ or, even more generally, by a Riemannian manifold $M$. 
The corresponding spaces will be denoted by $W^{k,p}(M)$. 
Here $\nabla u$ should be understood as the differential of $u$, that is $\nabla u\in \Gamma(T^*M)$.
fixing a connection on $T^*M$, for example the Levi--Civita one, $\nabla u$ can be differentiated  again. 
This clarifies the meaning of $\nabla^2 u\in \Gamma(T^*M\otimes T^*M)$ and so on up to the dependence of the space $W^{k,p}(M)$ on the choice of connection.

This naturally leads us to one more variation of the definitions above. 
Pick a vector bundle $E\to M$ and connections $\nabla^E\in \cA(E), \nabla^M \in \cA(T^*M)$. 
This yields a connection on $T^*M\otimes E$ so that the higher derivatives $\nabla^i(\nabla s)$ are well defined for any $s\in \Gamma(E)$. 
This yields the Sobolev spaces $W^{k,p}(M;\, E)$ in the same manner as above. 
While the norm $\|\cdot\|_{W^{k,p}}$ does depend on $\nabla^E$ and $\nabla^M$, different choices yield equivalent norms so that the resulting topology is independent of the choices made.   

\medskip

With this understood, for any $p>1$ we have the sequence of inclusions
\begin{equation*}
	L^p(M;\, E)=W^{0,p}(M;\, E)\supset W^{1,p}(M;\, E)\supset W^{2,p}(M;\, E)\supset\dots 
\end{equation*}
Relations between all these spaces is given by the following theorem, which is of fundamental importance in the theory of PDEs.

\begin{thm} 
	\label{Thm_SobolevEmbedding}
	Let $M$ be a compact manifold.
	\begin{enumerate}[(i)]
		\item If $s\in W^{k,p}(M;\, E)$, then $s\in W^{m, q}(M;\, E)$ provided
		\begin{equation*}
			k-\frac np\ge m-\frac nq\qand\quad k\ge m,
		\end{equation*}
		where $n=\dim M$, and there is a constant $C$ independent of $s$ such that 
		$\| s \|_{W^{m,s}}\le C \| s \|_{W^{k,p}}$. 
		In other words, the natural embedding
		\begin{equation*}
			j\colon W^{k,p}(M;\, E)\subset W^{m, q}(M;\, E)
		\end{equation*}
		is continuous. 
		\item  $j$ is a compact operator provided
		\begin{equation}
		\label{Eq_CondForCompEmb}
		k-\frac np > m-\frac nq\qand\quad k > m.
		\end{equation}
		This means that any sequence bounded in $W^{k,p}$ has a subsequence, which converges in $W^{m,q}$ provided~\eqref{Eq_CondForCompEmb} holds.
		\item
		We have a natural continuous embedding
		\begin{equation*}
			W^{k,p}(M;\, E)\subset C^r(M;\, E)
		\end{equation*}
		provided $k-\frac np>r$.
		In particular, if $s\in W^{k,p}(M;\, E)$ for some fixed $p$ and for all $k\ge 0$, then $s\in C^\infty(M;\, E)$. 
		\item\label{It_SobolevMultipl}
			\begin{enumerate}
				\item In the case $kp>n$ the space $W^{k,p}(M;\,\R)$ is an algebra.
				\item In the case $kp <n$, we have a bounded map
				\begin{equation*}
				\pushQED{\qed}
					W^{k_1,p_1}\otimes W^{k_2, p_2}\to W^{k,p},\qquad \text{provided}\quad k_1 -\frac n{p_1}\  +\  k_2 -\frac n{p_2}\ge k - \frac np.
				\qedhere\popQED
				\end{equation*}
			\end{enumerate} 
	\end{enumerate}
\end{thm}

\begin{remark}
	Although the proof of this theorem goes beyond the goals of these notes, it may be instructive to see `the spirit of the proof' in one particular case. 
	Thus, let $M=S^1$ and $u\in C^\infty(S^1;\, \R)$. 
	Denote $\bar u:= \frac 1{2\pi}\int u(\theta)\, d\theta$ and $u_0(\theta):= u(\theta) - \bar u$. 
	By the mean value theorem, there is $\theta_0\in S^1$ such that $u_0(\theta_0) =0$.  
	Hence, for any $\theta\in S^1$ we have
	\begin{equation}
		\label{Eq_AuxPoincareIneqCircle}
		|u_0(\theta)|=\Bigl | \int_{\theta_0}^{\theta} u_0'(\varphi)\, d\varphi \Bigr |\le \sqrt{\int_{\theta_0}^\theta |u_0'(\varphi)|^2\, d\varphi}\; \sqrt{\int_{\theta_0}^{\theta} 1^2\, d\varphi} \le
		\sqrt{2\pi}\;\| u_0 \|_{W^{1,2}},
	\end{equation}
	where the first inequality follows by the the Cauchy--Schwarz inequality.  
	This yields the estimate $\| u\|_{C^0}\le C\| u \|_{W^{1,2}}$, which in turn shows that there is a continuous embedding $W^{1,2}(S^1)\subset C^0(S^1)$. 
	
	Furthermore, by tracing through \eqref{Eq_AuxPoincareIneqCircle} it is easy to see  that 
	\begin{equation*}
		|u(\theta_1) - u(\theta_2)|\le \sqrt{2\pi}\, \| u \|_{W^{1,2}}\;\mathrm{dist} (\theta_1, \theta_2)^{\frac 12}.	
	\end{equation*}
	Hence, if $u_k$ is any sequence bounded in $W^{1,2}(S^1)$, then this sequence consists of uniformly bounded and equicontinuous functions on $S^1$. 
	By the Arzela--Ascoli theorem, this sequence has a convergent subsequence in $C^0(S^1)$, thus proving the compactness of the embedding $W^{1,2}(S^1)\subset L^p(S^1)$ for any $p$.   
\end{remark}

\subsection{Elliptic operators}

A map $L\colon C^\infty(\Om;\, \R^r)\to C^\infty (\Om;\, \R^s)$, where $\Om\subset \R^n$ is an open subset, is said to be a linear differential operator of order $\ell$ if $L$ can be expressed in the form 
\begin{equation}
	\label{Eq_LDOP}
	Lf = \sum_{|\a|\le \ell} A_\a(x) \frac {\del^{|a|}}{\del x^\a} f,
\end{equation} 
where $\a= (\a_1,\dots, \a_n), \ \a_i\in\Z_{\ge 0},$ is a multi-index, $|\a| := \sum \a_i$, and $A_\a\in C^\infty\bigl (\Om; \Hom (\R^r;\, \R^s)\bigr )$.
For example, in the case of operator of order $\ell=2$ acting on functions of two variables, we have
\begin{equation*}
	Lf= A_{20}\frac {\del^2}{\del x_1^2} + A_{11}\frac {\del^2}{\del x_1\del x_2} + A_{02}\frac {\del^2}{\del x_2^2} + A_{10}\frac {\del}{\del x_1} + A_{01}\frac {\del}{\del x_2} + A_{00},
\end{equation*} 
where $A_{ij}$ are smooth functions on $\Om$.

It is intuitively clear that the highest order terms determine some essential properties of $L$. 
Thus, we say that 
\begin{equation*}
	\sigma_L(\xi):=\sum_{|\a| = \ell} A_\a(x)\xi^\a = \sum_{|\a| = \ell} A_\a(x)\xi_1^{\a_1}\cdot\ldots\cdot \xi_n^{\a_n}\qquad \text{where }\xi\in (\R^n)^*\cong \R^n
\end{equation*}
is \emph{the principal symbol} of $L$.  
The symbol can be conveniently viewed as a map $\Om\times \R^n\to \Hom (\R^r;\, \R^s)$, which is polynomial in the $\xi$-variable. 

\begin{defn}
	A linear differential operator $L$ is called \emph{elliptic}, if for all $x\in\Om$ and all $\xi\in\R^n,\ \xi\neq 0$, we have: $\sigma_L(\xi)$ is an invertible homomorphism.
\end{defn}

Notice in particular that we must require  $r=s$ to have an elliptic operator.

\begin{example}
	For the (non-negative) Laplacian on $\R^n$ acting on functions we have
	\begin{equation*}
	\sigma_\Delta(\xi) = -\sum_{i=1}^n \xi_i^2 = -|\xi|^2.
	\end{equation*}
	Hence, $\Delta$ is an elliptic operator. 
	This is in fact a prototypical example of an elliptic operator. 
\end{example}

\medskip

The concepts above make sense for a more general setting of vector bundles.
To spell some details, let $E$ and $F$ be vector bundles over a manifold $M$ of rank $r$ and $s$ respectively.
We say that a map $L\colon \Gamma(E)\to\Gamma (F)$ is a linear differential operator of order $\ell$, if for any choice of local coordinates on $\Om\subset M$ and any trivializations of $E|_\Om$ and $F|_\Om$ the map $L$ can be represented as in~\eqref{Eq_LDOP}, where the coefficients $A_\a$ are allowed to depend on the choices made. 

Let $\pi\colon T^*M\to M$ be the natural projection. 
A straightforward computation shows that the symbol makes sense as a section of $\Hom(\pi^*E;\, \pi^*F)$. 
Then $L$ is said to be elliptic if the symbol is a pointwise invertible homomorphism away from the zero section of $T^*M$. 

\begin{example}
	Let $M$ be an oriented Riemannian manifold. 
	Consider the Laplace--Beltrami operator acting on the space of functions on $M$:
	\begin{equation*}
		\Delta f = -*d*df,
	\end{equation*}
	where $*\colon \Lambda^kT^*M\to \Lambda^{n-k}T^*M$ is the Hodge operator. 
	
	Choose local coordinates $(x_1,\dots, x_n)$ and write $g=g_{ij}\, dx_i\otimes dx_j$, $|g|=\det (g_{ij})$. 
	Denoting by  $(g^{ij})$  the inverse matrix, a straightforward computation yields the local form of the Laplace--Beltrami operator:
	\begin{equation*}
		\Delta f = - |g|^{-\frac 12}\sum_{i, j=1}^n \frac \del{dx_i}\Bigl (  |g|^{\frac 12} g^{ij}\frac {\del f}{\del x_j} \Bigr)
	\end{equation*} 
	From this it is easy to compute the symbol. Indeed, if $\xi=\sum \xi_i\, dx_i$, then
	\begin{equation*}
		\sigma_\Delta (\xi)= - |g|^{-\frac 12}\sum_{i, j=1}^n \xi_i |g|^{\frac 12} g^{ij}\xi_j = -|\xi|^2,
	\end{equation*} 
	where $|\cdot|$ denotes the norm on $T^*X$. 
	In particular, the Laplace--Beltrami operator is elliptic.
\end{example}  

\begin{example}
	For any Dirac bundle $E$ the corresponding Dirac operator $\dirac\colon\Gamma(E)\to\Gamma(E)$ is elliptic.
	Indeed, it is easy to see that the principal symbol $\sigma_\xi(D)$ is just the Clifford multiplication with $\xi$.  
	This is invertible with the inverse given by the Clifford multiplication with $|\xi|^{-2} \xi$.
	
	In dimension four (and in fact in any even dimension) the chiral Dirac operators $\slD^\pm$ are also elliptic.
\end{example}



Any linear elliptic operator $L\colon C^\infty(M;\, E)\to C^\infty(M;\, F)$ of order $\ell$ extends as a  bounded linear map 
\begin{equation}
	\label{Eq_AuxExtensionEllOp}
	L\colon W^{k+\ell, p}(M;\, E)\to W^{k, p}(M;\, F)
\end{equation}
 for any $k\ge 0$ and any $p>1$.

\begin{thm}[Elliptic estimate]
	\label{Thm_EllipticEstimate}
	Let $M$ be a compact manifold.
	For any linear elliptic operator $L$  there is a constant $C>0$ with the following property: If $Ls\in W^{k,p}(M;\, F)$, then $s\in W^{k+\ell,p}(M;\, E)$ and
	\begin{equation*}
		\| s \|_{W^{k+\ell,p}}\le C\bigl ( \|Ls \|_{W^{k,p}} + \| s \|_{L^p} \bigr ). 
	\end{equation*}
	Here $C$ depends on $k$ and $p$ but not on $s$.\qed
\end{thm}

Assume that both $E$ and $F$ are equipped with an euclidean structure. 
An operator $L^*\colon C^\infty(M;\, F)\to C^\infty(M;\, E)$ is said to be formal adjoint of $L$ if
\begin{equation}
\label{Eq_FormalAdjoint}
\langle Ls, t\rangle_{L^2} = \langle s, L^* t\rangle_{L^2}
\end{equation}
holds for any $s\in C^\infty(M;\, E)$ and any $t\in C^\infty (M;\, F)$. 
One can show that $L^*$ exists and is a linear differential operator of order $\ell$. 
Moreover, $L$ is elliptic if and only if $L^*$ is elliptic. 

One of the most important results in the theory of elliptic differential operators is the following.  

\begin{thm}[Fredholm alternative]
	\label{Thm_FredholmAlternative}
	Let $L$ be elliptic, $M$ compact, and $t\in C^\infty(M;\, F)$. 
	The equation
	\begin{equation*}
		Ls = t
	\end{equation*}
	has a smooth solution if and only if $t\in \ker L^*$. \qed
\end{thm}

\begin{remark}
	A corollary of \autoref{Thm_FredholmAlternative} can be formulated as follows.
	Under the hypotheses of \autoref{Thm_FredholmAlternative} one and only one of the following statements hold:
	\begin{enumerate}[(i), itemsep=0pt]
		\item The homogeneous equation $L^* s =0$ has a non-trivial solution.
		\item The inhomogeneous equation $Ls = t$ has a unique solution for any smooth $t$.  
	\end{enumerate}
	This form of  \autoref{Thm_FredholmAlternative} is widely used in the theory of PDEs. 
\end{remark}

\begin{defn}
	A bounded linear map $B\colon X\to Y$ between two Banach spaces is called Fredholm, if the following conditions hold:
	\begin{enumerate}[(a)]
		\item $\dim\ker B<\infty$;
		\item \label{It_ClosedImageOfFredholmOp}
		$\Im B$ is a closed subspace of $Y$;
		\item $\coker B:=Y/\Im B <\infty$. 
	\end{enumerate}
	If $B$ is Fredholm, the integer
	\begin{equation*}
		\ind B:=\dim\ker B - \dim\coker B
	\end{equation*}
	is called the index of $B$. 
\end{defn}

\begin{rem}
	One can show that~\ref{It_ClosedImageOfFredholmOp} follows the other two conditions. 
	Nevertheless, it will be useful to know that $\Im B$ is closed, even if one does not necessarily need to check this.     
\end{rem}

For example, any linear map between finite dimensional spaces is Fredholm and its index equals $\dim Y -\dim X$. 
In fact, Fredholm operators resemble linear maps between finite dimensional vector spaces well known from the basic course of linear algebra and this largely explains the importance of Fredholm operators for us. 

\begin{exercise}
	Let $B_0$ be a Fredholm operator. Show that there is an $\e>0$ such that any bounded operator $B$ with
	\begin{equation*}
		\| B-B_0\| <\e
	\end{equation*} 
	is also Fredholm. 
	Here $\| \cdot \|$ means the operator norm in the space of bounded linear maps $X\to Y$. 
	Prove also that $\ind B = \ind B_0$. 
\end{exercise}

\begin{thm}
	For any elliptic operator $L$ of order $\ell>0$ on a compact manifold $M$,~\eqref{Eq_AuxExtensionEllOp} is a Fredholm operator.
	Moreover, $\ker L$ consists of smooth sections only.
\end{thm} 
\begin{proof}[Sketch of proof]
	If  $s\in W^{\ell, p}$ is in the kernel of $L$, then $s\in W^{k,p}$ for any $k\ge 0$ by \autoref{Thm_EllipticEstimate}. 
	Hence, $s$ is smooth by \autoref{Thm_SobolevEmbedding}.

	Let $s_j\in W^{\ell, p}$ be any sequence such that $Ls_j =0$ and $\| s_j \|_{W^{\ell, p}}\le 1$. 
	Then the sequence $\| s_j\|_{W^{\ell +1, p}}$ is bounded by \autoref{Thm_EllipticEstimate}.
	Hence, by \autoref{Thm_SobolevEmbedding} after passing to a subsequence if necessary, $s_j$ converges to some limit $s_\infty$ in $W^{\ell, p}$.
	Since $L\colon W^{\ell, p}\to L^p$ is bounded,  $s_\infty\in \ker L$. 
	In other words, the unit ball in $\ker L$ is compact with respect to the $W^{\ell, p}$ norm. 
	This implies that $\ker L$ is finite dimensional. 
	
	Next, let us assume that $p=2$, which simplifies the  discussion somewhat. 
	Denote $V:=\bigl (\Im L\colon W^{k+\ell, 2}\to W^{k,2} \bigr)^\perp$, where the orthogonal complement is understood in the sense of $L^2$--scalar product.   
	One can show that all sections lying in $V$ are in fact in $W^{\ell,p}$.
	Then~\eqref{Eq_FormalAdjoint} implies that $V=\ker L^*$. 
	This in turn yields that $V$ is finite dimensional (and consists of smooth sections only). 
\end{proof}

\subsection{Elliptic complexes}

Fix a manifold $M$ and a sequence of vector bundles $E_1,\dots, E_k$, which I assume to be finite for the simplicity of exposition.
Let
\begin{equation}
	\label{Eq_ComplexDiffOp}
	0\to \Gamma(E_1) \xrightarrow{\ L_1\ }
	\Gamma(E_2) \xrightarrow{\ L_2\ }\dots
	\xrightarrow{\ L_{k-1}\ } \Gamma(E_k)\to 0 
\end{equation} 
be a sequence of differential operators such that $L_j\comp L_{j-1} =0$ for all integer $j\in [1, k-1]$, i.e., \eqref{Eq_ComplexDiffOp} is a complex. 
In particular, we can define the corresponding cohomology groups
\begin{equation*}
	H^j(E):= \ker L_j/\im L_{j-1},
\end{equation*}
where $j\in \{ 1,2,\dots, k \}$.

Associated to~\eqref{Eq_ComplexDiffOp} is the sequence of principal symbols:
\begin{equation*}
	0\to \pi^*E_1\xrightarrow{\ \sigma_{L_1}\ }
	\pi^*E_2\xrightarrow{\ \sigma_{L_2}\ }
	\pi^*E_3\xrightarrow{\ \sigma_{L_3}\ }\dots \xrightarrow{\ \sigma_{L_{k-1}} \ } \pi^*E_k\to 0.
\end{equation*} 
If this is exact on the complement of the zero section, \eqref{Eq_ComplexDiffOp} is called an elliptic complex. 
For example, a very short complex $0\to \Gamma(E)\xrightarrow{\ L\ }\Gamma(F)\to 0$ is elliptic if and only if $L$ is elliptic. 

With these preliminaries at hand, we can construct the associated Laplacians:
\begin{equation*}
\Delta = \Delta_j\colon \Gamma(E_j)\to \Gamma (E_j),\qquad \Delta_j = L_j^*L_j + L_{j-1}L_{j-1}^*.
\end{equation*}

\begin{exercise}
	Prove that each $\Delta_j$ is an elliptic operator provided the initial complex is  elliptic. 
\end{exercise}

In the sequel, I will drop the index of the differentials $L_i$ to simplify the notations so that all maps in~\eqref{Eq_ComplexDiffOp} are denoted by the same symbol $L$. 
I assume also that all differentials $L$ are of order $1$, since only this case will show up in these notes, but operators of other orders could be considered as well.
Also, I assume that all bundles $E_i$ are equipped with an Euclidean structure.  

Denote
\begin{equation*}
	\cH_j(E):=\bigl\{ s\in\Gamma(E_j)\mid \Delta s =0 \bigr \}.
\end{equation*}
Elements of $\cH_j(E)$ are called harmonic sections. 
\begin{thm}
	For an elliptic complex on a compact manifold the following holds:
	\begin{enumerate}[(i)]
		\item Each $\cH_j(E)$ is a finite dimensional vector space.
		\item $s\in \cH_j(E)$ if and only if $Ls=0$ and $L^*s=0$.
		\item The natural homomorphism
		\begin{equation*}
			\cH_j(E)\to H^j(E),\qquad s\mapsto [s]
		\end{equation*}
		is an isomorphism. 
	\end{enumerate} 	
\end{thm}
\begin{proof}
	The first statement follows from the ellipticity of $\Delta$. 
	The second statement follows easily from the equality
	\begin{equation*}
		\langle \Delta s, s\rangle_{L^2} = 
		\bigl \langle (LL^* + L^*L) s, s\bigr\rangle_{L^2} = 
		\| L^*s\|_{L^2}^2 + \| Ls\|_{L^2}^2.
	\end{equation*} 
	
	It remains to prove the last claim. 
	Thus pick any $s_0\in \Gamma(E_j)$ such that $L s_0 =0$. 
	We wish to show that the equation 
	\begin{equation*}
		(L+L^*) \bigl ( s_0+ Lt \bigr) =0
	\end{equation*}
	has a solution $t\in \Gamma(E_{j-1})$.
	Notice that this equation is equivalent to 
	\begin{equation}
		\label{Eq_AuxHarmRepresentative}
		L^*L t = -L^*s_0.
	\end{equation}
	Consider instead the equation $\Delta t = -L^*s_0$, whose right hand side is clearly $L^2$-orthogonal to $\ker \Delta^* = \ker\Delta$. 
	Hence, by \autoref{Thm_FredholmAlternative} there is a unique solution of $\Delta t =-L^*s_0$, which can be rewritten as 
	\begin{equation*}
		L(L^*t) + L^*(s_0 + Lt) =0.
	\end{equation*}
	A moment's thought shows that $\Im L$ is $L^2$-orthogonal to $\Im L^*$.
	 Hence,  $t$ is a solution of~\eqref{Eq_AuxHarmRepresentative} in fact. This finishes the proof of the existence part.
	 
	 The uniqueness is easy to show: If $s_1, s_2\in\cH_j(E)$ are such that $s_1-s_2 = Lt$ for some $t\in \Gamma(E_{j-1})$, then 
	 \begin{equation*}
	 	\| Lt \|_{L^2}^2 = \langle L^*Lt, t \rangle_{L^2} = \langle L^*(s_1-s_2), t \rangle_{L^2}=0.
	 \end{equation*}
\end{proof}

A refinement of the argument in the proof of the above theorem shows that in fact we have the following decomposition
\begin{equation}
	\Gamma(E_j) = \Im L\oplus \cH_j(E)\oplus \Im L^*,
\end{equation}
which is $L^2$--orthogonal. 
Details can be found for example in~\cite{Wells80_AnalOnCxMflds}. 

\begin{exercise}
	\label{Ex_EllComplexEllOper}
	Show that the short complex
	\begin{equation*}
		0\to \Gamma(E_1) \xrightarrow{\ L_1\ }
		\Gamma(E_2) \xrightarrow{\ L_2\ } \Gamma(E_3)\to 0
	\end{equation*}
	is elliptic if and only if the operator
	\begin{equation*}
		(L_2, L_1^*)\colon \Gamma(E_2)\to \Gamma(E_3\oplus E_1)
	\end{equation*}
	is elliptic. 
\end{exercise}

\subsubsection{A gauge-theoretic interpretation}

The cohomology group $H^j(E)$ is an example of a linear gauge--theoretic moduli space, cf. Section~\ref{Sect_Intro}.
Let me spell some details. 
The manifold\footnote{One can consider $\Gamma(E_j)$ as a Banach manifold by taking a Sobolev completion. Since this is not really important at this point, I describe the setting in the smooth category.} $\cB=\Gamma (E_j)$ carries an action of the additive group $\cG:= \Gamma(E_{j-1})$ by translations: $(s, t)\mapsto s + Lt$.  
If $\Gamma(E_{j+1})$ is viewed as a trivial $\cG$--representation, 
$L\colon\Gamma (E_j)\to\Gamma(E_{j+1})$ is a $\cG$-equivariant map, which in this particular case just means that $L$ is $\cG$-invariant. 
The action of the gauge group on $L^{-1}(0)$ is not free in general, however $\ker L$ acts trivially so that the corresponding `moduli space' $L^{-1}(0)/\cG = H^j(E)$ is a finite dimensional manifold (a vector space, in fact) provided the complex is elliptic and the base manifold is compact. 
This explains our interest in the theory of elliptic operators.

Notice, however, that because of the linear setting we can not expect that the `moduli space' $H^j(E)$ will be compact.
The reason is that $H^{j}(E)$ inherits the action of $\R_{>0}$ by dilations and if we take the quotient of $H^j\setminus\{ 0 \}$ by this action the resulting space is compact.
Thus, in this setting $\dim H^j(E)<\infty$ is a suitable replacement for the compactness of the moduli space.

One more important feature we have seen is the so called gauge fixing. 
Namely, we have shown that for each point $s\in L^{-1}(0)$ there is a unique representative $h(s)$ in the `gauge-equivalence class of $s$' such that $h(s)$ is harmonic. 
Moreover, the map 
\begin{equation*}
	L^{-1}(0)\to \cH_j(E),\qquad s\mapsto h(s)
\end{equation*}  
induces a diffeomorphism $L^{-1}(0)/\cG\to \cH^j(E)$. 

Furthermore, an isomorphism class of a  finite dimensional vector space is determined by a unique non-positive integer, namely its dimension.
Hence,  $b_j(E):=\dim H^j(E)$ is an `invariant' of $E$.
In many cases, these invariants capture a subtle information about the underlying manifold $M$.

\medskip

The example of the de Rham complex   in the following subsection will make these constructions more concrete.

\subsubsection{The de Rham complex}
\label{Sect_deRhamComplex}

Recall that for any manifold $M$ of dimension $n$ we have the de Rham complex
\begin{equation*}
	0\to \Om^0(M)\xrightarrow{\ d\ }
	\Om^1(M)\xrightarrow{\ d\ }
	\Om^2(M)\xrightarrow{\ \ \ }\dots \xrightarrow{\ \ \ } \Om^n(M)\to 0.
\end{equation*}

\begin{exercise}
	Show that the de Rham complex is elliptic.
\end{exercise}

If $M$ is oriented and Riemannian, we have the Hodge operator $*\colon \Lambda^jT^*M\to \Lambda^{n-j}T^*M$.  
Using this, the formal adjoint of the exterior derivative can be expressed as
\begin{equation*}
	d^*= (-1)^{n(j-1) +1}*d\,*\colon\Om^j(M)\to \Om^{j-1}(M)
\end{equation*} 
so that the Hodge--de Rham Laplacian is $\Delta = d d^* + d^* d$. 
If $M$ is in addition compact, the space $\cH_j$ of harmonic forms of degree $j$ is naturally isomorphic to the $j$th de Rham cohomology group. 
Then $b_j(M):=\dim \cH_j$ is the $j$th Betti number of $M$.
This is a topological invariant of $M$ even though in this approach a smooth structure has been used.

\section{Fredholm maps}
\label{Sect_FredholmMaps}

\subsection{The Kuranishi model and the Sard--Smale theorem}

Let $X$ and $Y$ be Banach manifolds.

\begin{defn}
	A map $F\in C^\infty(X;\, Y)$ is said to be Fredholm if the differential $dF$ is a Fredholm linear map at each point. 
\end{defn}

Hence, for each $x\in X$  the index of $d_xF$ is well defined. 
If $X$ is connected, then $\ind d_xF$ does not depend on $x$ and this common value is denoted by $\ind F$. 

Fredholm maps have a lot in common with smooth maps between finite dimensional manifolds.
A manifestation of this is the following result.
\begin{thm}[Kuranishi model]
	\label{Thm_KuranishiModel}
	Let $X$ and $Y$ be Banach spaces and $F\colon X\to Y$ a Fredholm map. 
	Pick a point $p\in F^{-1}(0)$ and denote $X_0=\ker d_{p}F$,\ \; $Y_0:=\im d_{p}F$. 
	Furthermore, choose subspaces $X_1\subset X$ and $Y_1\subset Y$ such that 
	\begin{equation*}
		X=X_0 \oplus X_1\qquad\text{and}\qquad Y= Y_0\oplus Y_1.
	\end{equation*}
	Then there is a diffeomorphism $\varphi$ of a neighborhood of the origin in $X$ onto a neighbourhood of $p$ such that $\varphi(0) =p$, a linear isomorphism $T\colon X_1\to Y_0$, and a smooth map $f\colon X\to Y_1$ such that
	\begin{equation*}
		F\comp\varphi (x_0, x_1) = Tx_1 + f(x_0, x_1) 
	\end{equation*}
	for all $(x_0,x_1)\in X_0\oplus X_1$ in the neighborhood of the origin.
	
	In particular, if $f_0\colon X_0\to Y_1$ denotes the restriction of $f$ to $X_0$, then a neighborhood of $p$ in $F^{-1}(0)$ is homeomorphic to a neighborhood of the origin in $f_0^{-1}(0)$.\qed
\end{thm}

\begin{corollary}
	Assume the hypotheses of \autoref{Thm_KuranishiModel}. 
	Suppose also that $0$ is a regular value of $F$, i.e., the $d_pF$ is surjective for all $p\in F^{-1}(0)$. 
	Then $F^{-1}(0)$ is a smooth manifold of dimension $\ind F$.  
\end{corollary} 
\begin{proof}
	Since $\Im d_pF = Y$, we necessarily have $Y_1=\{ 0 \}$ so that $f_0$ is a constant map. 
	Hence, $F^{-1}(0)$ is diffeomorphic to a neighborhood of the origin in $X_0$, i.e., a manifold of dimension $\dim X_0=\ind F$. 
\end{proof}

For a smooth map between finite dimensional manifolds, almost any value is regular by Sard's theorem~\cite{BardenThomas03_IntroDiffMflds}*{Thm.\,9.5.4}.
There is a generalization of this statement for an infinite dimensional setting due to Smale. 
This is commonly known as the Sard--Smale theorem.

Recall that a subset $A$ of a topological space is said to be of \emph{second category}, if $A$ can be represented as a countable intersection of open dense subsets.   
If the underlying topological space is a Banach manifold, then a subset of second category is dense.

\begin{thm}
	Let $F$ be a smooth Fredholm map between paracompact Banach manifolds.
	Then the set of regular values of $F$ is of second category, in particular dense.\qed
\end{thm} 

Let $Z\subset Y$ be a smoothly embedded finite dimensional submanifold. 
A map $F$ is said to be \emph{transverse} to $Z$, if for any $z\in Z$ and any $x\in F^{-1}(Z)$ the following holds: 
\begin{equation*}
	\Im d_xF + T_zZ = T_zY.
\end{equation*}
In particular,  $F$ is transverse to $Z=\{ z \}$ is and only if $z$ is a regular value of $F$. 
It is well-known that the notion of transversality is a useful generalization of the notion of regular value, see for instance~\cite{GuilleminPollack10_DiffTop}.

Just as in the finite dimensional case we have the following result.
\begin{thm}
	Let $Z\subset Y$ be a smoothly embedded finite dimensional submanifold. 
	If $F$ is transverse to $Z$, then $F^{-1}(Z)$ is a smooth submanifold of $X$ and 
	\begin{equation*}
		\pushQED{\qed}
		\dim F^{-1}(Z)= \ind F + \dim W. 
		\qedhere\popQED
	\end{equation*} 
\end{thm}

\subsection{The $\Z/2\Z$ degree}
\label{Sect_DegMod2Fredholm}

Recall that a map $F\colon X\to Y$ is called \emph{proper} if preimages of compact subsets are compact.

Let $F\colon X\to Y$ be a proper Fredholm map between (paracompact) Banach manifolds of index zero, where $Y$ is connected. 
Then for any regular value $y\in Y$ the preimage $F^{-1}(y)$ is a compact manifold of dimension zero, hence a finite number of points.
The number
\begin{equation*}
	\deg_2 F:= \# f^{-1}(y) \mod 2
\end{equation*}
is called the $\Z/2\Z$ degree of $F$.

\begin{thm}$\phantom{a}$
	\label{Thm_Deg2FredholmMaps}
	\begin{enumerate}[(i)]
		\item \label{It_Mod2DegIsIndepOfRegValue}
		$\deg_2 F$ does not depend on the choice of the regular value $y$;
		\item If $F_0$ and $F_1$ are homotopic within the class of proper Fredholm maps of vanishing index, then $\deg_2 F_0 = \deg_2 F_1$.
	\end{enumerate}
\end{thm}
\begin{proof}
	The proof requires a number of steps.
	\setcounter{step}{0}
	\begin{step}
		\label{Step_RegValuesOpenFredholm}
		If $F$ is a proper Fredholm map, then the set of regular values of $F$ is open and dense.
	\end{step}
			First notice that the set of critical points $\mathrm{Crit}(F)$ is closed. 
			Since any proper map is closed,  the set of critical values $F(\mathrm{Crit}(F))$ is closed.
			
	\begin{step}
		\label{Step_IndexIsLocConst}
		The function
		\begin{equation*}
			y\mapsto \# F^{-1}(y)\mod 2
		\end{equation*}
		is locally constant on the set of regular values. 
	\end{step}
			Let $y$ be a regular point of $F$, $F^{-1}(y)=\{ x_1,\dots, x_k \}$. 
			By the inverse function theorem,  there is a neighborhood $U$ of $y$ and a neighborhood $V_j$ of $x_j$ such that $F\colon V_j\to U$ is a diffeomorphism. 
			Then, for any $y'\in U$ we have $\#f^{-1}(y')=k$, hence the claim.  
			
		\begin{step}
			\label{Step_Mod2DegForHomMapsWithcommonRegValue}
			Let $F_0$ and $F_1$ be homotopic so that the homotopy is within the class  of proper Fredholm maps. 
			Then 
			\begin{equation}
			\label{Eq_AuxDeg2}
			\# F_0^{-1}(y) = \# F_1^{-1}(y) \mod 2
			\end{equation}
			for any $y$, which is a regular value for both $F_0$ and $F_1$.   
		\end{step}
		Let $F_t$, $t\in [0,1]$, be a homotopy.
		Let me assume first that, $y$ is a regular value for  $F\colon X\times [0,1]\to Y$.
		Then $F^{-1}(y)$ is a 1-dimensional manifold with boundary such that $\del F^{-1}(y) = F_0^{-1}(y)\cup F_1^{-1}(y)$. 
		Hence, \eqref{Eq_AuxDeg2} holds. 
		
		If $y$ is not a regular value of $F$, then we can choose $y_1$ arbitrarily close to $y$ such that $y_1$ is a regular value for any of $F_0, F_1, F$.  
		The conclusion of this step follows by Step~\ref{Step_IndexIsLocConst}.
			
	\begin{step}
		Let $x$ be an arbitrary point in the unit ball of a Banach space $B$. 
		Then there is a diffeomorphism $\varphi$ of $B$ such that the following holds: $\varphi(0)= x$,  $\varphi$ is the identity map on the complement of the Ball of radius $2$, and $\varphi$ is homotopic to the identity map relative to the complement of the ball of radius $2$.
	\end{step}
			This is really a finite dimensional statement. 
			Indeed, choose decomposition $B = V\oplus V'$, where $V$ is finite dimensional and $x\in V$. 
			By~\cite{Milnor97_TopologyFromDiffPtView}*{P.\,22} there is a 1-parameter family $\psi_t\colon V\to V$ of diffeomorphisms such that $\psi_1(0) =x$, $\psi_0=\id_{V}$, and $\psi_t$ is identity outside the unit ball in $V$.
			
			Choose a smooth function $\chi\colon \R_{\ge 0}\to \R_{\ge 0}$ such that $\chi (0) =1$ and $\chi(t)=0$ for $t\ge 1$. 
			Then 
			\begin{equation*}
				\varphi_t(v, v'):= \psi_{t\chi(|v'|)} (v)+ v'
			\end{equation*}  
			is a family of diffeomorphisms such that $\varphi_0 = \id$ and $\varphi_1(0,0) = \psi_1(0)=x$. 
			Clearly, for any $t$ we have $\|\varphi_t(v, v')\|\le 2$ by the triangle inequality. 
			
			Let us check that $\varphi_t$ is identity on the complement of the ball of radius $2$.
			Thus, pick any $(v, v')$ such that $\| v \| + \|  v'\|\ge 2$.
			Then either $\| v\|\ge 1 $ or $\| v'\|\ge 1$.
			The first case is clear and the second one we have $\varphi_t(v, v') = \psi_0(v) + v'=v+v'$ since $\chi (\| v' \|) =0$ by the construction. 
			
	\begin{step}
		If $Y$ is a connected Banach manifold, then for any $y_1, y_2\in Y$ there is a diffeomorphism $\varphi\in \Diff_0(Y)$  such that $\varphi (y_1)=y_2$, where $\Diff_0(Y)$ denotes the subgroup of all diffeomorphisms homotopic to the identity.
	\end{step}
	
	Fix $y_2\in Y$ and denote 
	\begin{equation*}
		C(y_2):=\bigl \{ y_1\in Y\mid \exists \varphi \in \Diff_0(Y)\text{ such that }\varphi(y_1)=y_2 \bigr  \}	
	\end{equation*}
	By the previous step, $C(y_2)$ is open and non-empty, hence $C(y_2)=Y$.
	
	\begin{step}
		We prove the claim of this theorem.
	\end{step}		

	Let $F$ be as in the statement of this theorem. 
	Pick any regular values $y_1$ and $y_2$ and $\varphi\in \Diff_0(Y)$ such that $\varphi(y_1) =y_2$. 
	Then $y_2$ is a common regular value for $F$ and $\varphi\comp F$.
	Since these maps are homotopic, by Step~\ref{Step_Mod2DegForHomMapsWithcommonRegValue} we have
	\begin{equation*}
		F^{-1}(y_2)=(\varphi\comp F)^{-1}(y_2) = F^{-1}(y_1)\mod 2.
	\end{equation*}
	This proves~\textit{\ref{It_Mod2DegIsIndepOfRegValue}}.
	
	Furthermore, if $F\colon X\times [0,1]\to Y$ is a homotopy between $F_0$ and $F_1$, we can choose $y\in Y$ which is a regular value for any of the following maps: $F_0, F_1,$ and $F$. 
	This implies the second assertion.
\end{proof}

\begin{corollary}
	Let $F\colon X\to Y$ be a proper Fredholm map between (paracompact) Banach manifolds of index zero, where $Y$ is connected. 
	If $\deg_2 F\neq 0$, then $F$ is surjective.\qed
\end{corollary}

\subsection{The parametric transversality}
\label{Sect_ParamTransversality}

Given a `random' point $y$ in the target, there is no reason to expect that the preimage $F^{-1}(y)$ will be a submanifold. 
Of course, if $F$ is Fredholm the  theorem of Sard--Smale implies that we can choose $y'$ arbitrarily  close to $y$ so that $F^{-1}(y')$ is a manifold indeed.
In practice, however, it is not always the case that one has a freedom to choose a point in the target. 
This is typically the case, which will be considered in some detail below, if $X$ and $Y$ are equipped with an action of a Lie group $\cG$ and $F$ is $\cG$--equivariant.
In this case it is natural to choose $y$ as a fixed point of the $\cG$--action on $Y$ so that $F^{-1}(y)$ inherits a $\cG$--action. 
This, however, restricts severely possible choices of $y$ so that the Sard--Smale theorem is not applicable in a straightforward manner.

One way to deal with this problem is as follows. 
Assume there is a connected  Banach manifold $W$ and a smooth Fredholm map $\cF\colon X\times W\to Y$ with the following properties:
\begin{enumerate}[(A)]
	\item \label{Hyp_EachMemberOfFamilyIsFredholm}
	For each $w\in W$ the map $\cF_w\colon X\times\{ w \}\to Y$ is Fredholm;
	\item \label{Hyp_PertOfF}
	There is a point $w_0\in W$ such that $\cF_{w_0} =F$;
	\item $y$ is a regular value of $\cF$;
\end{enumerate}  
Typically, it is not too hard to construct a map $\cF$ satisfying these properties.

Since $y$ is a regular value, $\cF^{-1}(y)\subset X\times W$ is a smooth Banach submanifold. 
We have a natural projection $\pi\colon \cF^{-1}(y)\to W$, which is just the restriction of the projection $X\times W\to W$. 

\begin{lem}$\phantom{a}$
	\begin{enumerate}[(i)]
		\item $\pi$ is a Fredholm map and $\ind\pi = \ind F$;
		\item There is a subset $W_0\subset W$ of second category such that $y$ is a regular value for  $\cF_w$ for all $w\in W_0$. \qed
	\end{enumerate}
\end{lem}

Hence, by the Sard--Smale theorem, there is $w$ arbitrarily close to $w_0$ such that 
\begin{equation*}
	\pi^{-1}(w)=\cF_w^{-1}(y)
\end{equation*}
is a smooth submanifold of dimension $\ind F$. 

\medskip

Suppose that in addition to hypotheses  \ref{Hyp_EachMemberOfFamilyIsFredholm}--\ref{Hyp_PertOfF}, the following holds:
\begin{enumerate}[(D)]
	\item \label{Hyp_EachMemberIsProper}
	Each map $\cF_w$ is proper.
\end{enumerate}

In what follows I will also assume that $\ind F =0$, since this is the setting where the degree of a Fredholm map was defined. 
However, this is by no means essential. 

Thus, if  \ref{Hyp_EachMemberOfFamilyIsFredholm}--\ref{Hyp_EachMemberIsProper} holds, $\cF_w^{-1}(y) =\pi^{-1}(w)$ consists of finitely many points. 
Since $W$ is connected by assumption, there is a path connecting $w_0$ and $w$ so that $\cF_{w}$ is homotopic to $\cF_{w_0}=F$.
Hence,  
\begin{equation*}
	\deg_2 F =\# \cF_w^{-1}(y) \mod 2 .
\end{equation*}
In particular, $\# \cF_w^{-1}(y) \mod 2$ does not depend on $w$. 

Of course, this conclusion is pretty much straightforward thanks to the facts we have established in the preceding subsection. 
The point is that the number  $\# \cF_w^{-1}(y) \mod 2$ could be taken as a definition of the degree of $F$ thus omitting the construction of  \autoref{Sect_DegMod2Fredholm}. 
For instance, this is commonly used  in the equivariant setting as discussed at the beginning of this section. 

\medskip

Let me describe this alternative construction of the degree of a proper Fredholm map in some detail.
Thus, choose any $w_1, w_2\in W_0$ so that $y$ is a regular value for $\cF_j:=\cF_{w_j}$. 
For a fixed $k\ge 0$ denote 
\begin{equation*}
	\Gamma(w_1, w_2):=\Bigl \{ \gamma\in C^k\bigl ( [1,2];\, W\bigr )\mid \gamma(1)=w_1\text{ and } \gamma(2)=w_2 \Bigr \},
\end{equation*}
which is a Banach manifold. 
Consider the map
\begin{equation*}
	\mathbf{F}\colon X\times\Gamma(w_1, w_2)\times [1,2]\to Y, \qquad \mathbf F(x,\gamma, t):= \cF(x, \gamma(t)). 
\end{equation*}  
Clearly, $y$ is a regular value for $\mathbf F$ too so that $\mathbf{F}^{-1}(y)$ is a submanifold of $X\times\Gamma(w_1, w_2)\times [1,2]$. 
Applying the Sard--Smale theorem to the restriction of the projection $X\times\Gamma(w_1, w_2)\times [1,2]\to \Gamma(w_1, w_2)$ we obtain that for a generic $\gamma\in \Gamma(w_1, w_2)$ the subset 
\begin{equation*}
	\cM_\gamma:=\bigl \{  (x,t)\in X\times [1,2]\mid \mathbf F(x, \gamma, t) =y\bigr \}
\end{equation*}
is a smooth submanifold of dimension $1$ and 
\begin{equation*}
	\partial \cM_\gamma = \cF_{1}^{-1}(y)\cup \cF_2^{-1}(y).
\end{equation*}
Hence, 
\begin{equation*}
	\#\cF_{1}^{-1}(y) =\#\cF_2^{-1}(y)\mod 2
\end{equation*}
so that the number $\#\cF_{w}^{-1}(y) \mod 2$ does not depend on $w\in W_0$. 
This common value, as we already know, is just $\deg_2(F)$.

\subsection{The determinant line bundle}

If $X$ and $Y$ are closed \emph{oriented} manifolds of the same finite dimension, for any map $F\colon X\to Y$ the degree can be extended to take values in $\Z$ rather than $\Z/2\Z$ \cite{GuilleminPollack10_DiffTop}.  
This generalization requires orientations of the background manifolds and this tool is not readily available in infinite dimensions. 
In this section I will describe how to deal with this problem.

\medskip

Let $\cP$ be a topological space and $\{T_p\mid p\in \cP\}$ be a continuous family of linear Fredholm maps. 
This means that the map 
\begin{equation*}
	\cP\to \mathrm{Fred}(X; Y),\qquad p\mapsto T_p
\end{equation*} 
is continuous with respect to the operator norm, where $X, Y$ are Banach spaces and $\mathrm{Fred}(X; Y)$ denotes the subspace of linear Fredholm maps. 

Pick a point $p\in \cP$. Since both the kernel and cokernel of $T_p$ are finite dimensional, we can construct the (real) line
\begin{equation*}
	\det T_p:=\Lambda^{\mathrm{top}}\ker T_p\otimes \Lambda^{\mathrm{top}}(\coker T_p)^*.
\end{equation*} 
As $p$ varies, we obtain a family of vector spaces of dimension one.
It turns out that this family is actually a locally trivial vector bundle, which is somewhat non-obvious given that the dimensions of the kernel as well as cokernel may `jump' as $p$ varies.  

To understand why this is the case, pick a point $p_0$ and a finite dimensional subspace  $V\subset Y$ transverse to $\Im T_{p_0}$.
The linear map 
\begin{equation*}
	T_{p, V}\colon X\oplus V\to Y,\qquad T_{p, V}(x,v) = T_{p}\, x +v
\end{equation*} 
is  surjective for $p=p_0$ and, hence, also for all $p$ sufficiently close to $p_0$.  
Therefore, we have the exact sequence
\begin{equation}
	\label{Eq_AuxSESDetBundle}
	0\to \ker T_p\to\ker T_{p, V}\to V \to V/\Im T_{p}\cap V \to 0
\end{equation}
Notice that $\coker T_p = (\Im T_p + V)/\Im T_p = V/\Im T_{p}\cap V$. 

\begin{exercise}
	Let $0\to U_0\to U_1\to U_2\to 0$ be a short exact sequence of real vector spaces. 
	Show that the `inner product map'
	\begin{equation*}
		\Lambda^\top U_0\otimes \Lambda^{\top}U_1^*\to \Lambda^\top (U_1/U_0)^* =\Lambda^\top U_2^*
	\end{equation*}
	induces an isomorphism $\Lambda^\top U_1\cong \Lambda^\top U_0\otimes \Lambda^\top U_2$.
	
	More generally, show that for any exact sequence of real vector spaces $0\to U_0\to U_1\to U_2\to \dots\to U_k\to 0$ there is a canonical isomorphism
	\begin{equation*}
		\bigotimes\Lambda^\top U_{\mathrm{even}}\cong \bigotimes \Lambda^\top U_{\mathrm{odd}}.
	\end{equation*}
\end{exercise}

Hence, by~\eqref{Eq_AuxSESDetBundle} we have a canonical isomorphism
\begin{equation*}
	\det T_p \cong \Lambda^\top\ker T_{p,V}\otimes \Lambda^\top V^*.
\end{equation*}
It can be shown  not only that $\dim \ker T_{p, V}$ is constant in $p$ near $p_0$ but also  that $\ker T_{p, V}$ is trivial in a neighborhood $W$ of $p_0$. 
Details can be found for instance in~\cite{McDuffSalamon12_JholomCurvesSymplTop}*{Thm.\,A.2.2}. 
Thus, the family of real lines $\{ \det T_p\mid p\in W\, \}$ admits a trivialization, i.e., $\det T$ is a vector bundle. 

\begin{defn}
	Two continuous families of Fredholm operators $\{ T_{p,\,i}\mid p\in \cP \}, \ i\in\{ 0,1 \}$ are said to be \emph{homotopic}, if there is a continuous family of Fredholm operators $\{ T_{p,\,t}\mid (p,t)\in \cP\times [0,1]\,  \}$, whose restriction to $\cP\times\{ 0 \}$ and $\cP\times\{ 1 \}$ yields the initial families.  
\end{defn}

Assume that $\{ T_{p,\,0} \}$ is homotopic to $\{T_{p,\,1}\}$ and the determinant line bundle $\det T_{p,\, 1}$ is trivial. 
Since $\cP\times [0,1]$ is homotopy equivalent to $\cP$, the bundle $\det T_{p,\,t}$ is also trivial. 
In particular, $\det T_{p,\, 0}$ is trivial. 
More precisely, we have the following statement.

\begin{proposition}
	Let $\{ T_{p,\,0} \}$ and $\{T_{p,\,1}\}$ be homotopic families of Fredholm operators such that $\det T_{p,\, 1}$ is trivial.
	Then $\det T_{p,\,0}$ is also trivial. 
	Moreover, a choice of a trivialization of $\det T_{p,\, 1}$ and a homotopy $\{ T_{p,t}\mid t\in [0,1] \}$ yields a trivialization of $\det T_{p,\,0}$.
	This is unique up to a multiplication with a positive function.  \qed
\end{proposition}

\subsection{Orientations and the $\Z$--valued degree}
\label{Sect_OrientAndDeg}

Let $X, Y,$ and $W$ be Banach manifolds. 
Let $F\colon X\times W\to Y$ be a smooth map such that the following holds:
\begin{enumerate}[(a)]
	\item \label{It_HypFamOfFredholmMaps}
	$F_w=F|_{X\times\{w \}}\colon X\to Y$ is a Fredholm map of index $d$ for each $w\in W$;
	\item $y\in Y$ is a regular value of $F$;
	\item $F_w^{-1}(y)$ is compact for any $w\in W$;
	\item \label{It_HypTrivOfDetBundle}
	$\det d_x F_w$ is trivial over $X\times W$. 
\end{enumerate}
Furthermore, I assume that a trivialization of the determinant line bundle has been fixed.

Pick any $w\in W$ such that $y$ is a regular value of $F_w$. 
Then for any $x\in F_w^{-1}(y)$ we have 
\begin{equation*}
	\det d_xF_w = \Lambda^\top \ker d_xF_w = \Lambda^\top T_x F_w^{-1}(y) 
\end{equation*}
so that the hypotheses above imply that $\cM_w:=F_w^{-1}(y)$ is an  oriented $d$--manifold. 

Furthermore, let $\gamma\colon [0, 1]\to W$ be a  path connecting $w_0$ and $w_1$ akin to the situation considered in \autoref{Sect_ParamTransversality}. 
For generic $\gamma$ the space
\begin{equation*}
	\cM_\gamma:= \bigl \{  (x,t)\in X\times [0, 1]\mid F_{\gamma(t)}(x) =y \bigr \}
\end{equation*}    
is an oriented manifold of dimension $d+1$ such that 
\begin{equation}
	\label{Eq_BdryOfCobordism}
	\partial \cM_\gamma = \cM_{w_1}\sqcup \overline{\cM}_{w_0},
\end{equation} 
where $\overline{\cM}_{w_0}$ means that the orientation of $\cM_{w_0}$ is reversed.

In the particular case $d=0$, $\cM_{w}$ is just a finite collection of points $\{ m_1,\dots, m_k \}$ equipped with signs $\{ \e_1,\dots, \e_k\,\}$ so that we can define
\begin{equation}
	\deg F_w :=\sum_{i=1}^k \e_k \in\Z.
\end{equation}
In fact, for any two choices  $w_0$ and $w_1$ as above, \eqref{Eq_BdryOfCobordism} implies that $\deg F_{w_0}= \deg F_{w_1}$ provided $w_0$ and $w_1$ are in the same connected component. 
Thus,  $\deg F_w$ does not depend on $w$ as long as $W$ is connected.

The following result summarizes the considerations above.  

\begin{thm}$\phantom{a}$
	\label{Thm_DegFredholmMaps}
	Assume $F\colon X\times W\to Y$ satisfies Hypotheses  \ref{It_HypFamOfFredholmMaps}--\ref{It_HypTrivOfDetBundle} above.
	Then $\deg F_w$ does not depend on $w$. \qed
\end{thm}

Clearly, the case $d=0$ is not really special. 
Indeed, as we already know for any $d\ge 0$, $\cM_w = F_w^{-1}(y)$ is a smooth oriented manifold for all $w$ in a dense subset in $W$. 
Then \eqref{Eq_BdryOfCobordism} shows that  for any two choices $w_0$ and $w_1$ in this subset, $\cM_{w_0}$ and $\cM_{w_1}$ are cobordant, i.e., the oriented cobordism class of $[\cM_w]$ is well--defined and does not depend on $w$. 
One can take this oriented cobordism class as an invariant, however, in practice it may be hard to deal with.
One way to extract a number out of this cobordism class is as follows. 

Let $P\to X$ be a (principal) bundle.
Assume there are characteristic classes $\a_1,\dots\a_k$ of $P$ such that $\a:=\a_1\cup\dots\cup\a_k\in H^d(X;\, \Z)\subset H^d(X;\, \R)$.
Since the restriction of $\a$ to $\cM_w$ can be represented by a closed $d$--form, say $\om$,  Stokes' theorem implies that 
\begin{equation*}
	\langle \a, [\cM_{w}] \rangle =\int_{\cM_{w}}\om
\end{equation*}
does not depend on $w$. 
In fact, this is an integer, since $[\om]$ represents an integral cohomology class. 

\subsection{An equivariant setup}
\label{Sect_EquivSetup}

Let $X$ be a Banach manifold equipped with an action of a Banach Lie group $\cG$. 
For any $x\in X$ the infinitesimal action of $\cG$ at $x$ is given by the linear map
\begin{equation}
	R_x\colon \Lie(\cG)\to T_xX,
\end{equation} 
whose image is the tangent space to the orbit through $x$.

For the sake of simplicity of exposition let me assume that $\cG$ acts freely on $X$. 
It will be also convenient to assume that $X$ and $\cG$ are Hilbert manifolds.

\begin{defn}
	A Hilbert submanifold $S\subset X$ containing $x$ is said to be a local slice of the $\cG$-action at $x$, if  the set $\cG S:=\{ g\cdot s\mid g\in \cG, \ s\in S \}$ is open in $X$ and 	the natural map
	\begin{equation*}
		\cG\times S\to \cG S,\qquad (g, s)\mapsto g\cdot s
	\end{equation*}
	is a diffeomorphism. 
\end{defn}

\begin{proposition}
	Assume $\cG$ acts freely on $X$. 
	If the $\cG$--action admits a slice at any point, then the quotient $X/\cG$ is a manifold. 	
\end{proposition}

The proof of this proposition is clear: $S$  can be identified with a neighbourhood of the orbit $\cG\cdot x$ in the quotient space $X/\cG$.

\begin{example}
	The group $\U(1)$ acts on $S^2$ by rotations around the $z$-axis. 
	If we remove the north and the south poles, this action is free. 
	The submanifold 
	\begin{equation*}
		S:= \bigl\{  (x, 0, z)\mid x^2 + z^2 =1, x>0 \bigr \}\cong (-1, 1)
	\end{equation*}
	is a global slice for the $\U(1)$--action.
	In particular, the quotient is a manifold, which in this case is naturally diffeomorphic to an interval.
\end{example}

\bigskip

Let $Y$ be a smooth manifold equipped with a $\cG$--action and  $F\colon X\times W\to Y$ be a smooth map such that each $F_w\colon X\to Y$ is $\cG$--equivariant.  
Let $y\in Y$ be a fixed point of the $\cG$--action, i.e., $g\cdot y=y$ for all $g\in \cG$. 
For any $x\in F_w^{-1}(y)$ we can construct the sequence 
\begin{equation}
	\label{Eq_AbstrDefComplex}
	0\to \Lie(\cG)\xrightarrow{\, R_x\ } T_xX\xrightarrow{\ d_xF_w\ } T_yY\to 0,
\end{equation}
which is in fact a complex.
This follows immediately from the equivariancy of $F_w$ (and the assumption that $y$ is a fixed point). 
This is called \emph{the deformation complex} at $x$.

The zero's cohomology group of the deformation complex is just the Lie algebra of the stabilizer of $x$.
Our assumption implies that this is trivial. 

The second cohomology group is just the cokernel of $d_xF_w$.
This is trivial if and only if $y$ is a regular value for $F_w$.

It is also easy to understand the meaning of the first cohomology group.
Indeed, $\ker d_xF_w$ is the Zariski tangent space to $F_w^{-1}(y)$. 
Since $y$ is fixed, $\cG$ acts on $F_w^{-1}(y)$ so that the first cohomology group can be thought of as the tangent space to the `moduli space'
\begin{equation*}
	\cM_w:=F_w^{-1}(y)/\cG
\end{equation*}
at $\cG\cdot x$. 

Since $T_xX$ is by assumption a Hilbert space, we have a linear map
\begin{equation}
	\label{Eq_AbstrDeformOperator}
	D_x:=(R_x^*,\, d_xF_w)\colon T_xX\to \Lie(\cG)\oplus T_yY
\end{equation}
whose kernel can be identified with the first cohomology group of the deformation complex. 

\begin{thm}
	\label{Thm_AbstrModuliSpaceIsMfld}
	Let $y$ is a fixed point of the $\cG$--action on $Y$.
	Assume that the following holds:
	\begin{enumerate}[(i)]
		\item \label{It_HypGActsFreely}
		$\cG$ acts freely on $X$;
		\item $y$ is a regular value for $F_w$;
		\item There is a local slice at each point $x\in F_w^{-1}(y)$; 
		\item \label{It_HypDefOperFredholm}
		$D_x$ is a Fredholm linear map of index $d$.
	\end{enumerate}
Then $\cM_w$ is a smooth manifold of dimension $d$.
\end{thm}
\begin{proof}
	The statement is local, so we can restrict our attention to a neighborhood of a point $x\in F_w^{-1}(y)$. 
	Let $S$ be a slice at $x$ so that $T_xS$ and $\Im R_w$ are complementary subspaces in $T_xX$. 
	Then $y$ is still a regular value for $F_w|_{S}$ and 
	\begin{equation*}
		\ker d_x F_w|_{S}=\ker (R_x^*,\; d_xF_w)
	\end{equation*}
	so that $F_w^{-1}\cap S$ is a manifold of dimension $\dim \ker D_x = d$.
\end{proof}

By tracing through the discussion of \autoref{Sect_OrientAndDeg} it is easy to see that the following theorem holds. 

\begin{thm}
	Assume that in addition to Hypotheses~\ref{It_HypGActsFreely}--\ref{It_HypDefOperFredholm} of \autoref{Thm_AbstrModuliSpaceIsMfld} the determinant line bundle $\det D_x$ is trivialized and this trivialization is preserved by the action of $\cG$. 
	Then $\cM_w$ is oriented. 
	If in addition $\cM_w$ is compact for any $w$, then the oriented bordism class of $\cM_w$ does not depend on $w$. \qed 
\end{thm}

\section{The Seiberg--Witten gauge theory}

\subsection{The Seiberg--Witten equations}
Recall that in dimension $4$ we have an isomorphism of $\Spin(4)$ representations
\begin{equation}
	\label{Eq_R4HomS+S-}
	\R^4\otimes\C\cong \Hom (\slS^+;\, \slS^-),
\end{equation} 
where $\Spin(4)$ acts on  $\R^4$ via the homomorphism $\Spin(4)\to \SO(4)$. 
\eqref{Eq_R4HomS+S-} is still valid as an isomorphisms of $\Spin^c(4)$--representations. 
Somewhat more explicitly, the standard inclusion $\R^4\to Cl(\R^4)$ yields a monomorphism 
\begin{equation*}
	\R^4\to \Hom(\slS^+;\, \slS^-),\qquad v\mapsto (\psi\mapsto v\cdot \psi).
\end{equation*}  
 Of course, we have also  an inclusion $\R^4\to \Hom(\slS^-;\, \slS^+)$.
 Hence, the Clifford multiplication with a 2--form yields a map $\Lambda^2\R^4\to \End(\slS^\pm)$, whose kernel is $\Lambda^2_\mp\R^4$ as a straightforward computation shows.
 Moreover, the image of this map consists of skew-Hermitian endomorphisms so that we obtain an isomorphism
 \begin{equation*}
 	\Lambda^2_+\R^4\to \su(\slS^+),
 \end{equation*}
 whose complexification yields $\Lambda^2_+\R^4\otimes \C\cong \End_0(\slS)$, cf.~\eqref{Eq_BasicIsoOfSUreps}. 

Notice also that we have a quadratic map
\begin{equation*}
	\mu\colon\slS^+\to i\, \su(\slS^+)\subset \End_0(\slS^+),\qquad \mu(\psi) = \psi\psi^* - \frac 12 |\psi|^2, 
\end{equation*} 
where the expression on the right hand side means the following: $\mu(\psi)(\varphi) = \langle \varphi, \psi\rangle \psi + \frac 12 |\psi|^2\,\varphi$. 
Somewhat more concretely, $\mu$ is just the map 
\begin{equation*}
\C^2\to  i\,\su(2),\qquad 
\begin{pmatrix}
\psi_1\\ \psi_2
\end{pmatrix}
\mapsto 
\frac 12 
\begin{pmatrix}
|\psi_1|^2 - |\psi_2|^2 & 2\, \psi_1\bar\psi_2\\
2\,\bar\psi_1\psi_2 & |\psi_2|^2 - |\psi_1|^2
\end{pmatrix}.
\end{equation*}
Hence, we can think of $\mu(\psi)$ as a purely imaginary self--dual 2-form.

\medskip

A more global version of these identifications  is as follows.
Pick an oriented Riemannian four--manifold $M$ equipped with a spin$^c$ structure. 
Denote by $\slS^\pm$ the corresponding spinor bundles, see \autoref{Sect_SpinSpinCstr} for details.
Then we have the isomorphisms of vector bundles 
\begin{align}
	T_\C^*M\cong \Hom(\slS^+;\, \slS^-), \\
	i\,\Lambda^2_+T^*M\cong i\,\su(\slS^+) 
\end{align}
and a fiberwise quadratic map 
\begin{equation*}
	\mu\colon\slS^+\to i\, \Lambda^2_+T^*M. 
\end{equation*}

With this understood, \emph{the Seiberg--Witten equations} for a pair $(\psi, A)\in \Gamma(\slS^+)\times\cA(P_{\mathrm{det}})$  are as follows:
\begin{equation}
	\label{Eq_SW4D}
	\slD_A^+\psi =0\qand\quad F_A^+ = \mu(\psi).
\end{equation}
It is also convenient to introduce \emph{the Seiberg--Witten map} by
\begin{equation*}
	SW\colon \Gamma(\slS^+)\times\cA(P_{\mathrm{det}})\to \Gamma(\slS^-)\times\Om^2_+(M;\, \R i),\qquad SW(\psi, A) = \Bigl ( \slD^+\psi, F_A^+ -\mu(\psi) \Bigr). 
\end{equation*}

\subsubsection{The gauge group action}

Since the structure group of $P_{\det}$ is the abelian group $\U(1)$, we have an identification $\cG:=\cG(P_{\det})\cong C^\infty(M;\, \U(1))$. 
This acts on $\cA(P_{\det})$ on the right by gauge transformations:
\begin{equation}
	\label{Eq_GaugeGpActOnAbConn}
	A\cdot g = A +2\, g^{-1}dg.
\end{equation} 
We can extend this to the right action of $\cG$ on \emph{the configuration space} $\Gamma(\slS^+)\times \cA(P_{\det})$ as follows:
\begin{equation*}
	(\psi, A)\cdot g = (\bar g\psi,\, A\cdot g). 
\end{equation*}

We let also $\cG$ act on $\Gamma(\slS^-)$ on the left in the obvious manner.
Extending this action by the trivial one on $\Om^2_+(M;\, \R i)$, we obtain a left action of $\cG$ on   $\Gamma(\slS^-)\times\Om^2_+(M;\, \R i)$.  

\begin{lem}
	The Seiberg--Witten map is $\cG$-equivariant, i.e.,
	\begin{equation*}
	\pushQED{\qed}
		SW \bigl ( (\psi, A)\cdot g \bigr ) =\bar g\cdot SW(\psi, A). 
	\qedhere\popQED
	\end{equation*}
\end{lem}

From this we obtain that $\cG$ acts on the space of solutions of~\eqref{Eq_SW4D}. 
The quotient
\begin{equation*}
	\cM_{SW} = \bigl\{ (\psi, A) \text{ is a solution of } \eqref{Eq_SW4D}\,\bigr \}/\cG
\end{equation*}
is called \emph{the Seiberg--Witten moduli space.} 

Notice that the action of $\cG$ on the configuration space is \emph{not} free. 
Indeed, if $g\in \cG$ is the stabilizer of a point $(\psi, A)$ in the configuration space, then~\eqref{Eq_GaugeGpActOnAbConn} implies that $g$ is constant. 
Hence, 
\begin{equation*}
	\Stab(\psi, A)\neq \{ 1 \}\quad\Longleftrightarrow\quad \psi \equiv 0.
\end{equation*}
The point $(\psi, A)$ with a non-vanishing spinor  $\psi$ are called \emph{irreducible}, while points of the form $(0, A)$ are called \emph{reducible}. 

Denote also
\begin{equation*}
	\cM_{SW}^{irr}:= \bigl\{ (\psi, A) \text{ is an irreducible solution of } \eqref{Eq_SW4D}\,\bigr \}/\cG. 
\end{equation*}

\subsubsection{The deformation complex}

As we know from \autoref{Sect_EquivSetup}, for any solution $(\psi, A)$ of the Seiberg--Witten equations, we can associate the deformation complex\footnote{Strictly speaking, at this point the constructions of \autoref{Sect_EquivSetup} are not applicable, but we will see below how to fix this.}:
\begin{equation}
	\label{Eq_SWDefComplex}
	0\to \Om^0(M;\, \R i)\xrightarrow{R_{(\psi, A)}} \Gamma(\slS^+)\oplus \Om^1(M;\, \R i)\xrightarrow{\ d_{(\psi, A)} SW\ } \Gamma(\slS^-)\oplus \Om^2_+(M;\, \R i)\to 0.
\end{equation}
To explain, since $\cA(P_{\det})$ is an affine space modelled on $\Om^1(M;\, \R i)$, the tangent space to the configuration space at any point can be naturally identified with the middle space of the complex. 
It is easy to compute the infinitesimal action of the gauge group:
\begin{equation*}
	R_{(\psi, A)}\,\xi = \bigl (  -\xi\psi,\;  2\,d\xi \,\bigr ),\qquad \xi\in \Om^0(M;\, \R).  
\end{equation*}  

\begin{lem}
	We have
	\begin{equation*}
		d_{(\psi, A)}SW(\dot\psi, \, \dot a) = \Bigl (  \slD_A^+\dot\psi + \frac 12\,\dot a\cdot \psi,\ d^+\dot a - 2\mu(\psi,\dot\psi) \Bigr ),
	\end{equation*}
	where the second summand of the first component means the Clifford multiplication of $\dot a$ and $\psi$, $d^+\dot a$ is the projection of $d\dot a$ onto the space of self--dual 2-forms, and $\mu(\cdot, \cdot)$ is the polarization of $\mu$.  \qed
\end{lem}

\begin{proposition}
	\label{Prop_SWDefComplexElliptic}
	For any solution $(\psi, A)$ of the Seiberg--Witten equations~\eqref{Eq_SWDefComplex} is an elliptic complex. 
\end{proposition}

The proof of this proposition hinges on the following result, which is of independent interest. 

\begin{proposition}
	For any Riemannian oriented four--manifold $X$, the Atiyah complex
	\begin{equation*}
		0\to \Om^0(X)\xrightarrow{\ d\ }\Om^1(X)\xrightarrow{\ d^+\ }\Om^2_+(X)\to 0
	\end{equation*}
	is elliptic.\qed
\end{proposition}

One can prove this proposition either by computing the principal symbols in local coordinates or by noticing that $d^+ + d^*$ is a twisted Dirac operator. 
I leave the details to the reader.

\medskip

\begin{proof}[Proof of \autoref{Prop_SWDefComplexElliptic}]
	Modulo zero order terms, which are clearly immaterial for the statement of this proposition, \eqref{Eq_SWDefComplex} can be written as the direct sum of the Atiyah complex and 
	\begin{equation*}
		0\to 0\to \Gamma(\slS^+)\xrightarrow{\ \slD^+\ } \Gamma(\slS^-)\to 0.
	\end{equation*}
	The claim follows from the ellipticity of both complexes.
\end{proof}

\subsubsection{Sobolev completions}

The smooth category, which was used to define the Seiberg--Witten map, does not allow us to use the technique developed in \autoref{Sect_FredholmMaps}. 
Hence, we will work with Sobolev completions of the spaces under considerations.
This is described next. 

Pick any $A_0\in\cA(P_{\det })$ so that we can identify $\cA(P_{\det })$ with $\Gamma(T^*M\otimes \R i)$ even if not canonically so. 
For any fixed $(k,p)$ the space
\begin{equation*}
	\cA^{k,p}(P_{\det }):=A_0 + W^{k,p}(T^*M\otimes \R i).
\end{equation*}
 is an affine Banach space, hence a Banach manifold.
 It is easy to see that the resulting structure is independent of the choice of $A_0$.
 
 Just like in the case of the configuration space, it is also convenient to complete the gauge group. 
 Namely, a map $M\to S^1$ is said to be of class $W^{k,p}$ if the composition $M\to S^1\subset \R^2$ is in $W^{k,p}(M;\, \R^2)$.
 The subset $\cG^{k,p}$ of all such maps is not a vector space, however this is a Banach manifold. 
 Moreover, if $kp>4$ the Sobolev multiplication theorem shows that  $\cG^{k,p}$ is closed under the pointwise multiplication, so that $\cG^{k,p}$ is in fact a Banach Lie group. 
 Its Li algebra is given by 
 \begin{equation*}
 	\Lie (\cG^{k,p}) = W^{k,p}(M;\, \R i).
 \end{equation*}
 
 \begin{proposition}
 	For any $k$ and any  $p>1$ such that $kp>4=\dim M$ the Seiberg--Witten map extends as a smooth map
 	\begin{equation*}
 		SW\colon W^{k+1, p}(\slS^+)\times\cA^{k+1, p}(P_{\mathrm{det}})\to W^{k,p}(\slS^-)\times W^{k,p}(\Lambda^2_+T^*M\otimes\R i).
 	\end{equation*}
 	The action of the gauge group extends to a smooth action of $\cG^{k+2,p}$ and $SW$ is equivariant with respect to this action.  
  \end{proposition}  
 \begin{proof}
 	Pick a smooth connection  $A_0$ as a reference point. 
 	Then recalling~\eqref{Eq_SpincDiracChangeOfConn}  we obtain by \autoref{Thm_SobolevEmbedding}~\ref{It_SobolevMultipl}
 	\begin{align*}
 		\slD_{A_0 + a}^+\psi &=\slD_{A_0}^+\psi + \frac 12\,a\cdot \psi\in W^{k,p}(\slS^-),\\
 		F_{A_0 + a}^+ - \mu(\psi) &=F_{A_0}^+ + d^+ a - \mu(\psi)\in W^{k,p}(\Lambda^2_+T^*M\otimes\R i),
 	\end{align*}
 	where $a\in W^{k+1, p}(T^*M\otimes\R i)$. 
 	
 	The fact that the action of the gauge group extends follows from the Sobolev multiplication theorem and~\eqref{Eq_GaugeGpActOnAbConn}, which explains the choice $k+2$ when completing the gauge group.    
 \end{proof}
 
 In what follows, any $(k,p)$ with $k$ sufficiently large would work. 
 For the sake of definiteness, I will stick to $(k,p) = (5,2)$, which suffices for the arguments invoked below.

 \subsubsection{Compactness of the Seiberg--Witten moduli space}
 
 The most important property of the Seiberg--Witten moduli space is its compactness. 
 In this section I explain why $\cM_{SW}$ enjoys this property. 

Before going into details, let me briefly explain the approach. 
Given any sequence $(\psi_n, A_n)$ we need to show that there is a convergent subsequence. 
This would follow from the compactness of Sobolev embedding, if we could establish that the sequence  $\| (\psi_n, A_n)\|_{W^{6,2}}$ is bounded from above by a constant independent of $n$. 
Here we have the freedom to change solutions by the gauge group action, since we want to extract a subsequence, which converges in the quotient space $\cC/\cG$. 
In fact, we will see below that $(\psi_n, A_n)$ is bounded in $W^{k,2}$ for any $k\ge 0$ possibly after applying  gauge transformations.
 
 The formal proof requires a number of technical lemmas. 
 The key property is as follows.
 
 \begin{lem}
 	\label{Eq_LemPointwiseEstSpinor}
 	There is a constant $C>0$ such that for any solution 
 	\begin{equation*}
 		(\psi, A)\in \cC^{5,2}:= W^{5,2}(\slS^+)\times \cA^{5,2}(P_{\det })	
 	\end{equation*}
 	 of the Seiberg--Witten equations we have 
 	\begin{equation*}
 		\| \psi \|_{C^0}\le C. 
 	\end{equation*}
 \end{lem}
 \begin{proof}
 	First notice that $\psi\in W^{5,2}$ implies that $\psi\in C^2$. 
 	
 	Let $x_0$ be the point of maximum of the function $|\psi|^2$. 
 	The pointwise equality
 	\begin{equation*}
 		\Delta |\psi|^2 = 2\,\langle\nabla_A^*\nabla_A\psi, \psi\rangle - 2\,|\nabla_A\psi|^2
 	\end{equation*}
 	implies that 
 	\begin{equation*}
 		\langle\nabla_A^*\nabla_A\psi, \psi\rangle \ge |\nabla_A\psi|^2 \qquad \text{at } x_0.
 	\end{equation*}
 
	 Furthermore, notice also that we have the following  pointwise equality:
	 \begin{equation*}
	 	\langle \mu(\psi)\psi,\psi  \rangle = \bigl\langle |\psi|^2 \psi -\frac 12 |\psi|^2\psi,\; \psi \bigr\rangle  = \frac 12\, |\psi|^4.
	 \end{equation*}
 	Combining this with the Weitzenb\"ock formula, we obtain
 	\begin{equation*}
 		0 =\langle \nabla_A^*\nabla_A\psi, \psi  \rangle +\frac 14 s_g|\psi|^2 + \frac 12\, |\psi|^4.
 	\end{equation*}
 	Hence, at $x_0$ the following holds:
 	\begin{equation*}
 		\frac 14 s_g(x_0)|\psi|^2(x_0) + \frac 12\, |\psi|^4(x_0) \le -|\nabla_A\psi|^2(x_0)\le 0.
 	\end{equation*}
 	If $|\psi(x_0)| =0$, there is nothing to prove. 
 	If $|\psi|(x_0)|>0$, then the above inequality yields
 	\begin{equation*}
 		|\psi|^2(x_0)\le -\frac 12\, s_g(x_0),
 	\end{equation*}
 	thus providing the required estimate. 
 \end{proof}

 \begin{remark}
	Notice that the proof of this lemma does not go through for the equations $\slD^+_A\psi =0, F_A^+=-\mu(\psi)$, which differ from~\eqref{Eq_SW4D} just by a sign. 
\end{remark}
 
 \begin{corollary}
 	\label{Lem_L4AprioriBound}
 	For any $p>1$ there is a non-negative constant $\kappa_p$, which depends on the background Riemannian metric $g$ only, such that for any solution $(\psi, A)\in \cC^{5,2}$
 	of the Seiberg--Witten equations the following estimate holds: $\| \psi \|_{L^p}\le \kappa_p$.\qed
 \end{corollary}

 \begin{corollary}
 	\label{Prop_L2AprioriBoundsCurvature}
 	For any solution $(\psi, A)$ of the Seiberg--Witten equations we have the estimates
 	\begin{equation}
	 	\label{Eq_AprioriCurvEst}
 		\| F_A^+ \|_{L^2}\le C\qquad \text{and}\qquad \| F_A^- \|_{L^2}\le C -4\pi^2\, c_1(L_{\det})^2,  
 	\end{equation}
 	where the constants depend on the background metric only. 
 \end{corollary}
 \begin{proof}
 	The first inequality follows immediately from  \autoref{Lem_L4AprioriBound}.
	 To prove the second bound, recall that the first Chern class of $L_{\det}$ is represented by $\frac {i}{2\pi}F_A$ so that we have
	 \begin{equation*}
		 \begin{aligned}
			 	c_1(L_{\det})^2 &=\frac {i^2}{(2\pi)^2}\int_M F_A\wedge F_A = -\frac 1{4\pi^2}\int_M (F_A^+\wedge F_A^+ + F_A^-\wedge F_A^-)\\
			 	& = \frac 1{4\pi^2} \bigl ( \| F_A^+ \|^2_{L^2} - \| F_A^-\|^2_{L^2}\bigr),
		 \end{aligned}
	 \end{equation*}	
	 which yields the required bound.
 \end{proof}
 
 \begin{rem}
 	Of course, the right hand side of the second inequality of~\eqref{Eq_AprioriCurvEst} is just a constant independent of the solution. 
 	However, this explicit form will be useful below.
 \end{rem}

 The proof of the proposition below hinges on the following technical lemma, which follows essentially from the elliptic estimate.
 
 \begin{lem}[\cite{Morgan96_SWequations}*{Lemma\,5.3.1}]
 	\label{Lem_UhlenbLemma}
 	Let $L$ be any Hermitian line bundle over $M$. 
 	Fix a smooth reference connection $A_0$.
 	For sny $k\ge 0$ there are positive constants $C_1$ and $C_2$ with the following property: For any $W^{k,2}$--connection $A$ on $L$ there is a gauge transformation $g\in \cG^{k+1, 2}$ such  that $A\cdot g = A_0 + \a$, where $\a\in W^{k, 2}(T^*M\otimes \R i)$ satisfies 
 	\begin{equation*}
	 		d^*\alpha = 0\qquad\text{and}\qquad 
	 		\| \a \|_{W^{k,2}}\le C_1 \| F_A^+ \|_{W^{k-1, 2}} + C_2.
 	\end{equation*} 
 	Moreover, the harmonic component $\a_h$ of $\a$ can be assumed to be bounded in $L^2$ by a constant independent of $k$.\qed
 \end{lem} 
 
 \begin{proposition}
 	\label{Prop_WkpAprioriEstimForSW}
 	For each $k\ge 0$ there are positive constants $C_k>0$ with the following property: For any solution $(\psi, A)$ of the Seiberg--Witten equations there is a gauge transformation $g\in \cG^{k+1, 2}$ such that 
 		\begin{equation}
 		\label{Eq_AuxAprioriBoundForPsiAlpha}
 		\| (\psi, \a) \|_{W^{k,2}}\le C_k,
 		\end{equation}
 		where $A\cdot g=A_0 +\a$.
 \end{proposition}
 \begin{proof}
 	The proof is given via the induction on $k$. 
 	However, a few first values of $k$ require a special treatment.
 	
 	For $k=0$ we know already that $\| (\psi, F_A^+ )\|_{L^2}$ is bounded by a constant independent of $(\psi, A)$. 
 	\autoref{Lem_UhlenbLemma} yields immediately the required estimate for $\a$.

 	For $k=1$, using the Seiberg--Witten equations we have
 	\begin{equation*}
 		0=\slD^+_A\psi = \slD^+_{A_0}\psi + \frac 12 \, \a \cdot \psi,
 	\end{equation*}
 	where the second summand is bounded in $L^2$ by \autoref{Eq_LemPointwiseEstSpinor}.
 	Invoking the elliptic estimate for $\slD_{A_0}^+$, we obtain a  bound for $\psi$ in $W^{1,2}$.
 	This proves~\eqref{Eq_AuxAprioriBoundForPsiAlpha} for $k=1$.

 	Let us consider the case $k=2$.
 	Notice that by the Seiberg--Witten equations we have the pointwise estimate
 	\begin{equation*}
 		|\nabla^{LC} F_A^+| = |\nabla^{LC}\mu (\psi)| \le C |\nabla^{A_0}\psi||\psi|\le C|\nabla^{A_0}\psi|,
 	\end{equation*}
 	where the first inequality follows from the fact that $\mu$ is quadratic and the second one follows by \autoref{Eq_LemPointwiseEstSpinor}.
 	This clearly implies a bound on the $W^{1,2}$-norm of  $F_A^+$.
 	This in turn, yields a bound on the $W^{2,2}$--norm of $\a$.  
 	
 	Furthermore, the bound for $\| \a\|_{W^{2,2}}$ together with the Sobolev multiplication theorem yields
 	that $\a\cdot \psi$ is bounded in $W^{1,2}$. 
 	By the same token as above, this yields the $W^{2,2}$--bound on $\psi$. 
 	
 	By now, the reader will have no difficulties in proving~\eqref{Eq_AuxAprioriBoundForPsiAlpha} for $k=3$.
 	I leave this as an exercise. 
 	
 	We are now prepared to make the induction step. 
 	Thus, assume that~\eqref{Eq_AuxAprioriBoundForPsiAlpha} has been established for some $k\ge 3$. 
 	Since $W^{k, 2}$ is an algebra for $k\ge 3$, we have a bound on $F_A^+$ in $W^{k,2}$.
 	\autoref{Lem_UhlenbLemma} yields a $W^{k+1,2}$--bound on $\a$.
 	
 	By a similar argument,  $\a\cdot \psi$ is bounded in $W^{k,2}$ so that the elliptic estimate yields a $W^{k+1, 2}$--bound for $\psi$.
 	This finishes the proof of this proposition.  
 \end{proof}
 
 \begin{corollary}
 	\label{Cor_CompactnessOfTheModuli}
 	The Seiberg--Witten moduli space 
 	\begin{equation*}
 		\cM:=\bigl \{  (\psi, A)\in \cC^{5,2} \mid SW(\psi, A) =0\bigr \}/\cG^{6,2}
 	\end{equation*}
 	is compact.
 \end{corollary}
 \begin{proof}
 	Any sequence of solutions $(\psi_n, A_n)$ is gauge equivalent to a sequence $(\psi_n, A_0 +\alpha_n)$ such that  $(\psi_n, \a_n)$ is bounded in $W^{6,2}$. 
 	By the Sobolev embedding theorem, a subsequence converges in $\cC^{5,2}$ and the limit is a solution of the Seiberg--Witten equations. 
 \end{proof}
 
 In fact, \autoref{Cor_CompactnessOfTheModuli} can be substantially  strengthened, as the following result shows.
 
 \begin{thm}
 	For each solution  $(\psi, A)\in \cC^{5,2}$ of the Seiberg--Witten equations there is a gauge transformation $g\in \cG^{6,2}$ such that $(\psi, A)\cdot g$ is smooth. Furthermore, $\cM$ is homeomorphic to  
 	\begin{equation*}
 		\cM^\infty:=\bigl \{  (\psi, A)\in \Gamma(\slS^+)\times \cA(P_{{\det}}) \mid SW(\psi, A) =0\bigr \}/ C^\infty (M;\, \U(1)). 
 	\end{equation*}
 	This space is compact in the $C^\infty$--topology.  
 \end{thm} 
 \begin{proof}
 	Let $(\psi, A)\in \cC^{5,2}$ be any solution.
 	By \autoref{Prop_WkpAprioriEstimForSW}, $F_A = F_{A\cdot g} = F_{A_0} + d\a\in W^{k-1,2}$ for all $k$.
 	In particular, $F_A$ is smooth.
 	
 	Writing $A\cdot g=A_0 +\a$ as before, we obtain
 	\begin{equation*}
 		(d+d^*)\,\a = d\a = F_{A} - F_{A_0}\in C^\infty. 
 	\end{equation*}
 	By the elliptic regularity, $\a$ is smooth, i.e., $A$ is smooth. 
 	Hence, $\psi$ is also smooth as an element of the kernel of a smooth elliptic differential operator $\slD^+_A$.
 	
 	Furthermore, let $(\psi_n, A_n)$ be any sequence of smooth solutions. 
 	By \autoref{Cor_CompactnessOfTheModuli}, this contains a subsequence still denoted by $(\psi_n, A_n)$, which converges in the $W^{5,2}$--topology after possibly applying a sequence of $\cG^{6,2}$ gauge transformations and the limit lies also in $\cC^{5,2}$.  
 	
 	Since $(\psi_n, A_n)$ is bounded in $W^{7,2}$, a subsequence converges in  $W^{6,2}$ after possibly applying a sequence of $\cG^{7,2}$ gauge transformations. 
 	Repeating this process for each $k\ge 6$ and choosing the diagonal subsequence, we obtain the claim.  
 \end{proof}
 
  \subsubsection{Slices}
  
  In this section we wish to construct local slices for the gauge group action on the subspace of irreducible configurations
  \begin{equation*}
  \cC_{\mathrm{irr}}^{5,2}:= \bigl \{ (\psi, A)\in \cC^{5,2} \mid \psi\not\equiv 0\, \bigr \},
  \end{equation*}
  where $\cG^{6, 2}$ acts freely.

  The tangent space to the slice at a point $(\psi, A)$ must be transversal to $\Im R_{(\psi, A)}$, hence it is natural to consider the kernel of the formal adjoint operator
  \begin{equation*}
  R_{(\psi, A)}^* (\dot\psi, \dot a) =2\, d^*\dot a + i\,\Re \langle \psi, i\,\dot \psi \rangle. 
  \end{equation*}
  Here $\langle\cdot, \cdot \rangle$ denotes the Hermitian scalar product on the spinor bundle.
  
  \begin{proposition}
  	For any irreducible configuration $(\psi, A)$ the subspace
  	\begin{equation*}
  	(\psi, A) + \ker R_{(\psi, A)}^*
  	\end{equation*}
  	is a slice for the $\cG^{6,2}$--action on $\cC_{\mathrm{irr}}^{5,2}$. 
  	Here $\ker R_{(\psi, A)}^*$ means the kernel of  the map $R_{(\psi, A)}^*\colon W^{5,2}\to W^{4,2}$.\qed  
  \end{proposition}

 \subsubsection{A perturbation}
 
 In general, there is no reason to expect that the origin will be a regular value for the Seiberg--Witten map.
 However, a family of perturbations just like in \autoref{Sect_ParamTransversality} can be constructed by hand. 
 
 \begin{proposition}
 	The origin is a regular value of the map
 		\begin{gather*}
	 		\cS\cW\colon \cC_{\mathrm{irr}}^{5,2}\times W^{4,2}(\Lambda^2_+T^*M\otimes \R i )\to W^{4,2}(\slS^-)\times W^{4,2}(\Lambda^2_+T^*M\otimes\R i),\\
	 		\cS\cW(\psi, A, \eta) = \bigl ( \slD_A^+\psi,\; F_A^+ - \mu(\psi) -\eta \bigr ).
 		\end{gather*}
 \end{proposition}
 \begin{proof}
  The perturbation has been chosen so that 
  \begin{equation*}
  	\pr_2\,\frac {\del \cS\cW}{\del\eta}\colon W^{4,2}(\Lambda^2_+T^*M\otimes\R i)\to W^{4,2}(\Lambda^2_+T^*M\otimes\R i)
  \end{equation*}
  is surjective.
  Hence, it is enough to show that the map 
  \begin{equation*}
  T:=\pr_1\,d_{(\psi, A)}\cS\cW_\eta\colon T_{(\psi, A)}\cC^{5,2}\to W^{4,2}(\slS^-),\qquad (\dot\psi,\, \dot A)\mapsto \slD_A\dot\psi + \frac 12\, \dot A\cdot \psi
  \end{equation*}
  is also surjective for any solution $(\psi, A)$ of the perturbed Seiberg--Witten equations. 
  
  Assume this is not the case.
  Then there is a non-zero vector 
  \begin{equation*}
  	\varphi\in \Im T^\perp:= \bigl\{  \varphi\in W^{4,2}\mid \langle T(\dot\psi,\, \dot A),\, \varphi \rangle_{L^2} =0\quad \forall (\dot\psi,\, \dot A)\,\bigr \}.
  \end{equation*}
  In particular, we have $0=\langle \slD^+_A\dot\psi, \varphi \rangle_{L^2} = \langle \dot\psi, \slD_A^-\varphi \rangle_{L^2}$, which implies
  \begin{equation}
	  \label{Eq_AuxElemInCokerHarm}
	  	\slD_A^-\varphi =0. 
  \end{equation}
  Also, 
  \begin{equation}
	  \label{Eq_AuxL2ConstrAdot}
  	0=\langle T(0,\dot A),\, \varphi \rangle_{L^2}= \langle \dot A\cdot\psi,\, \varphi \rangle_{L^2}
  \end{equation} 
  for all $\dot A\in W^{5,2}(T^*M\otimes\R i)$.

  Since $W^{5,2}\subset C^0$ in dimension four,  $\psi$ is continuous and by assumption does not vanish identically. 
  Hence, there is a point $m\in M$ such that $\psi$ does not vanish on a neighborhood $U$ of $m$. 
  
  Furthermore, notice  that the Clifford multiplication with a fixed non-zero vector $\psi_0\in \slS^+$ is surjective, i.e., the map
  \begin{equation*}
  	\R^4\to\slS^-,\qquad \rv\mapsto \rv\cdot\psi_0
  \end{equation*}
  (here $\slS^\pm$ are thought of as $\Spin^c(4)$--representations) is surjective. 
  This implies the following: if $\varphi$ does not vanish on $U$, then there is some $\dot A$ supported in $U$ such that $\langle \dot A\cdot\psi,\, \varphi \rangle_{L^2}>0$.
  However, this contradicts~\eqref{Eq_AuxL2ConstrAdot} so that we conclude that $\varphi$ must vanish on an open set. 
  Then, by a theorem of Aronszajn \cite{Aronszajn57_UniqueContinuation},  
  \eqref{Eq_AuxElemInCokerHarm} implies that $\varphi$ vanishes identically. 
  This in turn proves the surjectivity of $T$ and finishes the proof of this proposition. 
 \end{proof}

 The deformation complex at a solution of the perturbed Seiberg--Witten equations
 \begin{equation*}
 	\slD_A^+\psi =0,\qquad F_A^+  =\mu(\psi) +\eta
 \end{equation*}
 is given again by \eqref{Eq_SWDefComplex} since the differentials of $SW = \cS\cW_0$ and $\cS\cW_\eta$ coincide. 
 Hence, this  complex is elliptic and therefore the operator
 \begin{equation}
	 \label{Eq_PertDefOp}
 	\bigl ( d_{(\psi, A)}\cS\cW_\eta, \; R_{(\psi, A)}^* \bigr )\colon W^{5, 2}(\slS^+\oplus T^*X\otimes \R i) \to W^{4,2}(\slS^- \oplus \Lambda^2_+ T^*X\otimes \R i)
 \end{equation} 
 is also elliptic by Exercise~\ref{Ex_EllComplexEllOper}.
 Thus, \eqref{Eq_PertDefOp} is Fredholm so that  by  appealing to  \autoref{Thm_AbstrModuliSpaceIsMfld}, we obtain the following result. 
 
 \begin{corollary}
 	\label{Cor_GenericPert}
 	There is a subset $\Eta\subset W^{4,2}(\Lambda^2_+T^*X\otimes \R i )$ of the second category such that for any $\eta\in\Eta$ the space
 	\begin{equation*}
 		\cM^{\textrm{irr}}_\eta:= \bigl \{  (\psi, A)\in\cC^{5,2}\mid \cS\cW(\psi, A,\eta) =0,\quad\psi\not\equiv 0 \,\bigr \}/\cG^{6,2}
 	\end{equation*}
 	is a smooth manifold of dimension 
 	\begin{equation}
	 	\label{Eq_VirtDimModSpace}
 		d=\frac 14\Bigl (  c_1(L_{\det})^2 - 2\,\chi(M) - 3\,\sign(M)  \Bigr ),
 	\end{equation}
 	where $\chi(M)$ and $\sign(M):=b_2^+ - b_2^-$ are the Euler characteristic and the signature of $M$ respectively.
 \end{corollary}
 \begin{proof}
 	We only need to prove the formula for the dimension of the moduli space. 
 	To this end the following fact will be useful: The index is a locally constant function on the space of all Fredholm operators. 
 	Equivalently, if $\{ T_t\mid t\in [0,1] \}$ is a one-parameter family of Fredholm operators, then $\ind T_0 = \ind T_1$. 
 	
 	With this understood, consider the operator 
 	\begin{equation}
	 	\label{Eq_SWDeformOperator}
 		D_{(\psi, A)}:= \bigl (d_{(\psi, A)}SW_\eta, \ R_{(\psi, A)}^*  \bigr)\colon T_{(\psi, A)}\cC^{5,2}\to W^{4,2}\bigl ( \slS^-\oplus\ \Lambda^2_+T^*M\otimes\R\, i\ \oplus\ \underline\R\, i\bigr).
 	\end{equation} 
 	Writing
 	\begin{equation}
	 	\label{Eq_D0}
 		D_{(\psi, A)} = 
	 		\begin{pmatrix}
		 		\slD_A^+ &  \\
		 		  & d^+ + d^*
	 		\end{pmatrix}
	 		+ B = D_0 + B,
 	\end{equation}
 	where $B$ is a zero order operator, we see that $D_{(\psi, A)}$ is homotopic through Fredholm operators to $D_0$. 
 	Since $D_0$ decouples, we have 
 	\begin{equation*}
 		\ind D_0 = \ind \slD_A^+ +\ind(d^+ +d^*).
 	\end{equation*}
 	The latter index is easily computed: $\ind(d^+ +d^*) = b_1-b_0 - b_2^+=\frac 12 (\chi + \sign)$.
 	The index of $\slD_A^+$ can be computed by applying the Atiyah--Singer index theorem: 
 	\begin{equation*}
	 	\ind_\C\slD_A^+ = \frac 18\bigl ( c_1(L_{\det})^2-\sign \bigr ).
 	\end{equation*}
 	This yields the result. 
 \end{proof}
 
 \begin{remark}
 	\label{Rem_CompactnessForPerturbedCase}
 	Strictly speaking, at this point we should redo the analysis of the compactness of the Seiberg--Witten moduli space for the perturbed equations. 
 	However, this requires cosmetic changes only. 
 	The reader should have no difficulties to check that this is indeed the case. 
 \end{remark}
 
 \subsubsection{Reducible solutions}
 
 While \autoref{Cor_GenericPert} yields a smooth moduli space, by removing some points we lost an essential property, namely compactness. 
 The following lemma demonstrates how to deal  with this problem.

 \begin{proposition}
 	Assume $b_2^+\ge 1$. 
 	Then there is an affine subspace  $Q\subset W^{4,2}(\Lambda^2_+T^*M\otimes\R\, i)$ of codimension $b_2^+$ with the following property: If $\eta\notin Q$, then there are no reducible solutions of the Seiberg--Witten equations. 
 \end{proposition}
 \begin{proof}
 	Pick a reference connection $A_0$ on $L_{\det}$ and denote $Q:=F_{A_0}^+ + \Im d^+$, which is a subspace of codimension $b_2^+$. 
 	Clearly, if $\eta\notin Q$, then there are no reducible solutions.
 \end{proof}
 
 \begin{corollary}
 	Assume $b_2^+\ge 1$. 
 	Then for a generic $\eta$ the perturbed Seiberg--Witten moduli space $\cM_\eta$ contains no reducible solutions. \qed
 \end{corollary}

 \subsubsection{Orientability of the Seiberg--Witten moduli space}
 
 As we have already seen in the proof of \autoref{Cor_GenericPert}, $D_{(\psi, A)}$ defined by~\eqref{Eq_SWDeformOperator} is homotopic through elliptic operators to $D_0$, which is given by~\eqref{Eq_D0}.
 This can be also used to orient the Seiberg--Witten  moduli space. 
 The key is the following simple observation.
 \begin{lem}
 	Let $\cP$ be a topological space. 
 	If $\{  T_{p,\, 0}\mid p\in \cP\, \}$ and $\{  T_{p,\, 1}\mid p\in \cP\, \}$ are two homotopic families of linear Fredholm maps, then $\det T_0\cong \det T_1$.
 	More precisely, this means that $\det T_1$ is trivial if and only if $\det T_0$ is trivial and a trivialization of $\det T_0$ induces a trivialization of $\det T_1$. 
 	The latter is well-defined up to a multiplication with an everywhere positive function.  
 \end{lem} 
 \begin{proof}
 	Let $\{  T_{p,\, t}\mid (p,t)\in \cP\times [0,1]\, \}$ be a homotopy through linear Fredholm maps. 
 	Then $\det T$ is well-defined over $\cP\times [0,1]$ and restricts to $\det T_0$ and $\det T_1$ on the corresponding components of the boundary.  
 	This implies the statement of this lemma.
 \end{proof}

\begin{remark}
	Let $\det T$ be as in the proof of the lemma above.
	One can construct an isomorphism between $\det T_0$ and $\det T_1$ explicitly by introducing a connection on $\det T$ and taking the parallel transport along the curves $t\mapsto (p, t)$. 
\end{remark} 

With this understood, to orient the Seiberg--Witten moduli space it suffices to check that $\det D_0$ is trivial and pick a trivialization of this bundle. 
Since $D_0$ splits, we have
\begin{equation*}
	\det D_0\cong \det \slD_A^+\otimes \det (d^+ + d^*).
\end{equation*}
Notice that $\slD_A^+$ is a complex linear map.
In particular, both $\ker D_A^+$ and $\coker D_A^+\cong \ker D_A^-$ are complex linear subspaces, hence oriented. 
This implies that the real determinant bundle $\det \slD_A^+$ is trivial.

Furthermore, we have 
\begin{equation*}
	\ker (d^+ + d^*) = H^1_{dR}(M;\R\, i),\qquad \coker (d^+ + d^*) = H^0(M;\R\, i)\oplus H^2_+(M;\R\, i).
\end{equation*}
Hence, by picking an orientation of these cohomology groups, a trivialization of $\det (d^+ + d^*)$  is fixed.
This in turn yields a trivialization of $\det D_0$, hence also of $\det D_{(\psi, A)}$.
Thus, the Seiberg--Witten moduli space is orientable and a choice of orientations of the cohomology groups $H^0(M;\R)$, $H^1(M;\R\, i)$ and $H^2_+(M;\R\, i)$  yields an orientation of the Seiberg--Witten moduli space.
 
\begin{remark}
	Strictly speaking, the orientation of the Seiberg--Witten moduli space we constructed in this section is \emph{not} canonical. 
	First, there are at least two commonly used non-equivalent conventions how to orient complex linear spaces. 
	Namely, if $V$ is a complex linear space and $(\rv_1, \dots, \rv_k)$ is a complex basis of $V$, the underlying real vector space $V_\R$ can be oriented by saying that 
	\begin{equation*}
		(\rv_1,\dots, \rv_k, i\rv_1,\dots,i\rv_k)\qquad \text{or}\qquad 
		(\rv_1, i\rv_1,\dots, \rv_k, i\rv_k)
	\end{equation*}
	determines an orientation of $V_\R$. 
	This yields the opposite orientations in the case $k=\dim_\C V$ is even.
	
	Secondly, orientations of  $H^0(M;\R)$, $H^1(M;\R\, i)$ and $H^2_+(M;\R\, i)$ is again a choice.
	Notice however, that these vector spaces depend on the topological structure of $M$ only.  
\end{remark}

 \subsection{The Seiberg--Witten invariant}
 
 Combining results obtained in the preceding  sections  we arrive at our main result.
 
 \begin{thm}
 	Assume $b_2^+(M)\ge 2$. For any spin$^c$-structure $\sigma\in\cS(M)$ and any generic $\eta$ the perturbed Seiberg--Witten moduli space $\cM_\eta$ is a smooth compact oriented manifold of dimension $d$, which is given by~\eqref{Eq_VirtDimModSpace}. 
 	Moreover, if $\eta_0$ and $\eta_1$ are any two generic perturbations, then $\cM_{\eta_0}$ and $\cM_{\eta_1}$ are oriented-bordant.\qed
 \end{thm}

The only thing, which perhaps needs an explanation, is the hypothesis $b_2^+(M)\ge 2$. 
The point is that this ensures that a generic bordism between $\cM_{\eta_0}$ and $\cM_{\eta_1}$ does not contain reducible solutions and therefore is smooth.  

\begin{remark}
	Notice that $\cM_\eta$ depends on the choice of $\sigma$ but this dependence is suppressed in the notations.
\end{remark}

With this understood, we can proceed to the definition of the Seiberg--Witten invariant. 
Pick a point $m_0\in M$ and consider \emph{the based gauge group}  $\cG_0:=\{ g\in W^{6,2}(M;\U(1))\mid g(m_0) =1\, \}$, which fits into the exact sequence
\begin{equation*}
	\{ 1 \}\to \cG_0\to\cG\xrightarrow{\ \ev_{m_0}\ }\U(1)\to \{ 1\},
\end{equation*}
 where $\ev_{m_0}$ is the evaluation at $m_0$. 
 Then the quotient
 \begin{equation*}
 	\hat \cM_\eta:=\{  (\psi, A)\in \cC^{5,2}\mid SW_\eta(\psi, A) =0 \}/\cG_0
 \end{equation*}
 is called \emph{the framed moduli space} and is equipped with a free action of $\U(1)=\cG/\cG_0$. 
 Clearly, the quotient $\hat \cM_\eta/\U(1)$ is just $\cM_\eta$ so that $\hat \cM_\eta$ is a principal $\U(1)$--bundle over $\cM_\eta$. 
Let $\mu$ be the first Chern class of this $\U(1)$--bundle. 
Notice that this $\hat \cM_\eta$ is just a restriction to $\cM_\eta$ of the following bundle: $\cC_{\mathrm{irr}}/\cG_0\to \cC_{\mathrm{irr}}/\cG$. 

\begin{defn}
	The Seiberg--Witten invariant of $M$ is a function $\rs\rw\colon \cS(M)\to \Z$ defined by 
\begin{equation*}
	\rs\rw (\sigma):=
	\begin{cases}
		\bigl \langle [\cM_\eta], \mu^{d/2}\bigr \rangle  &\text{if } d\text{ is even},\\
		0 & \text{if } d\text{ is odd.}
	\end{cases}
\end{equation*}
\end{defn}

\begin{thm}
	The Seiberg--Witten invariant vanishes for all but finitely many spin$^c$ structures.
\end{thm}
\begin{proof}
	If $\rs\rw(\sigma)\neq 0$, then $\cM_\eta$ is a  smooth manifold of dimension $d\ge 0$. Hence, 
	\begin{equation*}
		c_1(L_{\det})^2\ge 2\,\chi(M) + 3\,\sign(M).
	\end{equation*}
	By \eqref{Eq_AprioriCurvEst}, we have\footnote{At this point we should have used an analogous a priory estimate for the perturbed Seiberg--Witten equations, cf. \autoref{Rem_CompactnessForPerturbedCase}.  However, this does not change the argument.} $\| F_A^-\|^2_{L^2}\le C$, where $C$ depends on the background Riemannian metric and the topology of $M$ only.
	This yields that  $\| F_A\|^2_{L^2}$ is also bounded by a constant, which depends on the background Riemannian metric and the topology of $M$ only.
	This in turn shows that $|c_1(L_{\det})|^2 = |c_1(L_{\det})\,\cup\, \PD(c_1(L_{\det}))| = \frac 1{4\pi^2} \| F_A\|^2_{L^2}\le C$, where $\PD$ stays for the Poincar\'e dual class. 
	This proves the statement of this theorem. 
\end{proof}

\subsubsection{Sample application of the Seiberg--Witten invariant}

The Seiberg--Witten invariant are most well--known for its applications to the topology of smooth four--manifolds. 
I restrict myself just to a list of some sample applications.

\begin{thm}[Witten]
	If $M$ with $b_2^+\ge 2$ admits a metric of positive scalar curvature, then $\rs\rw_M\equiv 0$.\qed
\end{thm}

The proof of this theorem follows easily from the Weitzenb\"ock formula. 

\begin{thm}[Witten]
	Let $M_1$ and $M_2$ be closed four-manifolds both with $b_2^+\ge 1$. 
	Then the Seiberg--Witten invariant of the connected sum $M_1\# M_2$ vanishes.\qed 
\end{thm}

These  vanishing theorems are complemented by the following non-vanishing result, which can be also viewed as the statement about obstructions for the existence of symplectic strutures on closed smooth four-manifolds.  

\begin{thm}[Taubes]
	If $M$ is symplectic and $b_2^+(M)\ge 2$, then $\rs\rw_M\not\equiv 0$. \qed
\end{thm}

The results below were originally obtained by other methods, however can be also obtained with the help of the Seiberg--Witten theory.

\begin{thm}[Donaldson]
	There are (many) closed topological four-manifolds, which do not admit a smooth structure.\qed
\end{thm}

\begin{thm}[Fintushel--Stern]
	There are infinitely many closed four-manifolds, which are all homeomorphic, but pairwise non-difeomorphic.\qed
\end{thm}

Put differently, the last two theorems say the following: On a given closed topological four-manifold there may or may not be a smooth structure and if a smooth structure exists, there may well be infinitely many such structures.

\bibliography{references}

\end{document}